\documentclass[11pt]{article}
\usepackage{amsmath,amsthm,amssymb}
\usepackage[T1]{fontenc}
\usepackage[utf8]{inputenc}
\usepackage[dvipdf]{graphicx}
\usepackage{color}
\usepackage{epstopdf}
\usepackage{dsfont}
\usepackage{mathtools}
\usepackage{enumerate}
\usepackage{mathrsfs}
\usepackage[normalem]{ulem}
\usepackage[colorlinks=true,citecolor=red,linkcolor=db,urlcolor=blue,pdfstartview=FitH]{hyperref}
\usepackage[numbers]{natbib}

\definecolor{db}{RGB}{0, 0, 130}
\definecolor{rp}{rgb}{0.25, 0, 0.75}
\definecolor{dg}{rgb}{0, 0.6, 0}

\newtheorem{theorem}{Theorem}[section]

\newtheorem{definition}{Definition}[section]

\newtheorem{corollary}[definition]{Corollary}
\newtheorem{example}[definition]{Example}
\newtheorem{assumption}[definition]{Assumption}

\newtheorem{proposition}[definition]{Proposition}
\newtheorem{remark}[definition]{Remark}
\def\R{\mathbb{R}}
\def\D{\mathbb{D}}
\def\E{\mathbb{E}}
\def\N{\mathbb{N}}
\def\Q{\mathbb{Q}}
\def\x{\times}
\def\Om{\Omega}
\def\om{\omega}
\def\Fc{\mathcal{F}}
\def\Ab{\mathbb{A}}
\def\F{\mathbb{F}}
\def\P{\mathbb{P}}
\def\G{\mathbb{G}}
\def\V{\mathbb{V}}
\def\L{\mathbb{L}}

\def\ox{\otimes}

\def\eps{\varepsilon}

\def\Yh{\widehat{Y}}
\def\Tc{\mathcal{T}}
\def\lt{\widetilde{\mathbf{l}}}

\def\Lh{\widehat{L}}

\def\Lo{\overline{L}}
\def\Po{\overline{\mathbb{P}}}

\def\Lt{\widetilde{L}}

\def\Bc{\mathcal{B}}
\def\Gc{\mathcal{G}}
\def\Lc{\mathcal{L}}

\def\Pc{\mathcal{P}}

\def\x {\times}

\def\Yh{\widehat{Y}}
\def\fh{\widehat{f}}
\def\Et{\widetilde E}

\def\Yt{\widetilde Y}

\def\Fcu{\widetilde{\Fc}}

\def\ft{\tilde f}
\def\ut{\tilde u}

\def\Omo{\overline{\Om}}
\def\etah{\overline{\eta}}

\def\tauh{\hat \tau}

\def\mt{\widetilde{m}}

\def\Thetah{\overline{\Theta}}
\def\thetah{\overline{\theta}}
\def\thetau{\underline{\theta}}

\def\esssup{{\rm ess}\!\sup}
\def\essinf{{\rm ess}\!\inf}

\def \endproof{\hbox{ }\hfill$\Box$}
\def\Zh{\widehat Z}

\def\S{\mathbb{S}}
\def\Fcb{\overline{\Fc}}
\def\Omb{\overline{\Omega}}
\def\omb{\bar \om}
\def\Pb{\overline{\P}}
\def\Fbb{\overline\F}
\def\Pcb{\overline \Pc}

\title{A mean-field version of Bank--El Karoui’s representation of stochastic processes\thanks{We are grateful to Peter Bank for very helpful discussions during this project.}}
\author{
Xihao He
\footnote{Department of Mathematics, University of Michigan, Ann Arbor. hexihao@umich.edu.}
\and
Xiaolu Tan
\footnote{Department of Mathematics, The Chinese University of Hong Kong. xiaolu.tan@cuhk.edu.hk. The research of Xiaolu Tan is supported by Hong Kong RGC General Research Fund (project 14302622).}
\and
Jun Zou
\footnote{Department of Mathematics, The Chinese University of Hong Kong. zou@math.cuhk.edu.hk. The  research of Jun Zou was substantially supported by Hong Kong RGC General Research Fund
(Projects 14306921 and 14306719)}}
\date{\today}

\begin{document}

\maketitle

\begin{abstract}

	We investigate a mean-field version of Bank--El Karoui's representation theorem of stochastic processes.
	Under three different technical conditions, we established some existence and uniqueness results, using three different approaches.
	As motivation and first applications, the results of mean-field representation provide a unified approach for studying various mean-field games (MFGs) in the setting with common noise and multiple populations,
	including the MFG of timing and the MFG with singular control, etc.
	As a crucial technical step, a stability result was provided on the classical Bank--El Karoui’s representation theorem.
	It has its own interest and other applications,
	such as deriving stability results of optimizers (in the strong sense) for a class of optimal stopping and singular control problems.

\end{abstract}


\vspace{0.5em}

\noindent\textbf{MSC2010 subject classification:} 60G40, 93E20, 60G07, 93E15

\section{Introduction}

	Let $(\Om, \Fc, \P)$ be a complete probability space, equipped with the filtration $\F = (\Fc_t)_{t \ge 0}$ satisfying the usual conditions.
	For $T \in [0, \infty]$, $\Tc$ denotes the space of all stopping times taking values in $[0,T]$.
	Let $f: [0,  T] \x \Om \x \R \longrightarrow \R$ be such that
	$\ell \longmapsto f(t, \om, \ell)$ is continuous and strictly increasing from $-\infty$ to $\infty$ for every $(t, \om) \in [0, T] \x \Om$,
	and $(t, \om) \longmapsto f(t, \om, \ell)$ is progressively measurable with $\E \big[\int_0^T |f(t, \ell)| dt \big]< \infty$ for every $\ell \in \R$.
	Bank and El Karoui \cite{BankKaroui2004} proved that any optional process $Y$, satisfying $Y_T = 0$ and some further integrability and path regularity conditions,
	can be represented with an optional process $L$ by the following expression:
	\begin{equation} \label{eq:BankEK_repres_intro}
		Y_\tau = \E\bigg[\int_{\tau}^T f\Big(t, \sup_{\tau \le s < t}L_s \Big) dt \Big| \Fc_\tau\bigg],
		~\mbox{a.s., for all}~
		\tau \in \Tc.
	\end{equation}
	Notice that, as $L$ is an optional process, the variable $\sup_{\tau \le s < t} L_s$ is measurable in the complete probability space by the classical measurable selection theorem.
	Such a representation result provides a powerful tool in stochastic analysis, particularly for solving various stochastic optimization problems,
	such as the optimal consumption problem in Bank and Riedel \cite{BankRiedel01},
	the optimal stopping problem (see, e.g., Bank and F\"ollmer \cite{BankFollmer}), and
	the singular control problem (see, e.g., Bank \cite{Bank2004}, Bank and Kauppila \cite{BankKauppila2017}, Chiarolla and Ferrari \cite{ChiarollaFerrari2014}).
	It has also been used to solve an exit contract design problem by He, Tan, and Zou \cite{HeTanZou2021}.
	For extensions of this classical representation theorem, let us mention the work of Bank and Besslich \cite{BankBesslich2021}  for the case with Meyer-measurable processes,
	and the work of Ma and Wang \cite{MaWang2009}, Qian and Xu \cite{QianXu2018} for some variations of the reflected stochastic backward differential equation.

	\vspace{0.5em}

	In this study, we study a mean-field version of the representation result \eqref{eq:BankEK_repres_intro} in the form:
	\begin{equation} \label{eq:MF_repres_intro}
		Y_{\tau}
		=
		\E \Big[
		\int_{\tau}^T  f\Big(t , \Lc(L), \sup_{\tau \le s < t }L_s \Big) dt
		\Big|\Fc_\tau \Big],
		~\mbox{a.s., for all}~
		\tau \in \Tc,
		~\mbox{a.s.},
	\end{equation}
	where the new added mean-field term $\Lc(L)$ represents the distribution of the process $L$.
	Motivated by its applications, we will in fact study a more general version, by considering the conditional distribution $\Lc(L|\Gc)$ of $L$ knowing a sub-$\sigma$-filed $\Gc \subset \Fc$ and by letting the process $Y$ depend also on conditional distribution $\Lc(L|\Gc)$.

	\vspace{0.5em}
	
	To establish the representation result in \eqref{eq:MF_repres_intro}, we will apply three different approaches, each under a distinct set of assumptions which complement each other.
	Concretely, we  will consider it as a fixed-point problem and apply different fixed-point theorems.
	Given a distribution $m$, one replaces $\Lc(L)$ by $m$ in \eqref{eq:MF_repres_intro} and then applies the classical representation \eqref{eq:BankEK_repres_intro} to obtain a representation process $L^m$, which induces a distribution $\hat m := \Lc (  L^m )$.
	The problem then reduces to find a fixed point of the map $m \longmapsto \hat m$.
	In a first approach, we applied the Schauder fixed-point theorem to find a fixed point or equivalently to obtain a mean-field representation result.
	In this approach, a key step is to derive the continuity of $m \longmapsto L^m$ in some sense.
	Therefore, we establish a novel stability result of Bank--El Karoui's representation theorem \eqref{eq:MF_repres_intro}.
	In a second approach, we applied Tarski's fixed-point theorem to establish the mean-field representation result,
	in which a partial order was introduced on the space for $\Lc(L)$ (or the conditional law) and formulate some technical conditions on $f$ to verify the required conditions in Tarski’s fixed-point theorem.
	In a third approach, we consider a special structure condition which reduces the fixed point problem to the one-dimensional case satisfying a monotonicity condition,
	and then obtain an existence and uniqueness result.

	\vspace{0.5em}

	The main motivation to introduce and to study the mean-field representation in \eqref{eq:MF_repres_intro} is its applications in the mean-field game (MFG) theory.
	The MFG consists of studying the limit behavior of the Nash equilibrium of a stochastic differential game with a large number of agents in a symmetric setting.
	In this symmetric game, each agent interacts with others through the empirical distribution induced by the states (and/or optimal actions) of all agents,
	which converges to the distribution (or conditional distribution in the case with common noise) of the state (and/or optimal actions) of one representative agent as the population number increases.
	This (conditional) distribution term is called the mean-field interaction term.
	The MFG was first introduced by Lasry and Lions \cite{LasryLions2007} and by Huang, Caines, and Malham\'{e} \cite{HuangCainesMalhame2006}.
	Since then, MFG has been attracting considerable research interest from various communities because of its widespread applications in fields, such as economics, biology, and finance.
	In particular, we refer to Cardaliaguet \cite{Cardaliaguet2010} for a basic introduction and to Carmona and Delarue \cite{CarmonaDelarue2018} for recent advances.

	\vspace{0.5em}

	Although the initial formulation of the MFG is obtained by using a classical optimal control problem for each agent,
	other stochastic optimization problems can be considered to obtain other versions of the MFG problem.
	For example, when agents solve the optimal stopping problems with interaction in terms of the (conditional) distribution of the optimal stopping times (or stopped processes),
	the MFG of timing is obtained.
	The MFG of timing has been initially studied by Nutz \cite{Nutz2018} and by Carmona, Delarue, and Lacker \cite{CarmonaDelarueLacker2017} by using mainly a probabilistic approach.
	In the setting with the underlying process being given by a stochastic differential equation (SDE),
	the MFG of timing has been studied by Bertucci \cite{Bertucci2018} with a PDE approach,
	by Bouveret, Dumitrescu, and Tankov \cite{BouveretDumitrescuTankov2020}
	and by Dumitrescu, Leutscher, and Tankov \cite{DumitrescuLeutscherTankov2021} with a linear programming approach.
		
	\vspace{0.5em}

	When the agents in the game solve the singular control problem, the MFG with singular control is obtained.
	Such a formulation of the MFG has been studied by Fu and Horst \cite{FuHorst2017},
	Fu \cite{Fu2019}, Cao, Guo, and Lee \cite{CaoGuoLee2022},
	where the main results consist in the existence of solution, its approximation by those of regular controls, and its approximation by the corresponding $N$-player games, etc.
	Other formulations, such as the ergodic MFG with singular control was studied by Cao, Dianetti, and Ferrari \cite{CaoDianettiFerrari2021},
	a MFG of finite-fuel capacity expansion with singular control was studied in Campi, De Angelis, Ghio, and Livieri \cite{CampiDeAngelisGhio2020}.
	Recently, Dianetti, Ferrari, Fischer, and Nendel \cite{DianettiFerrari2022} provides a unified framework for different MFG with common noise under submodularity assumptions,
	which also includes the MFG of timing, the MFG with singular control, etc.

	\vspace{0.5em}
	
	The mean-field representation result \eqref{eq:MF_repres_intro} provides a unified method for studying various versions of the MFG.
	In fact, for various classical stochastic optimization problems, the optimal solution can be derived from process $L$ in the classical Bank--El Karoui's representation \eqref{eq:BankEK_repres_intro} for an appropriate process $Y$.
	In the mean-field setting, the (conditional) distribution of $L$ induces the (conditional) distribution of the optimizer, which would be exactly the required interaction term in the corresponding MFG problem.
	Consequently, by specifying the dependence of the generator function $f$ in $\Lc(L)$ in the mean-field representation result \eqref{eq:MF_repres_intro} according to various applications,
	a solution to \eqref{eq:MF_repres_intro} would induce solutions to some MFGs.
	Concretely, similar to the applications of \eqref{eq:BankEK_repres_intro} in the optimal stopping problem, singular control problem, optimal consumption problem,
	one can apply \eqref{eq:MF_repres_intro} to study the MFG of timing, MFG with singular control, and MFG of optimal consumption.
	Crucially, for a given process $Y$, the same induced process $L$ in \eqref{eq:BankEK_repres_intro} provides simultaneously optimal solutions to various optimization problems.
	Based on this result, the mean-field representation result \eqref{eq:MF_repres_intro} can be used to study the MFG with multiple populations.

	\vspace{0.5em}

	Our first main contribution consists in providing a nontrivial extension of the classical Bank--El Karoui's representation theorem to the case in which the process $Y$ and generator function $f$ depend on the (conditional) distribution of the representation process $L$, as in \eqref{eq:MF_repres_intro}.
	As first applications, the novel representation results provide a unified approach for studying various MFGs (e.g., MFG of timing and MFG with singular control), in a setting with (possibly) multiple populations and common noise.
	In the literature, these MFG problems are typically studied case by case, and in the one population setting.
	Furthermore, because our representation theorem is formulated on a fixed filtered probability space, the derived solutions (or $\eps$-solution) of the MFGs are strong solutions.
	This would be a major difference of our results compared with existing MFG literature (see more discussions in Remarks \ref{rmk:eps-mfe} and \ref{rmk:comparison_singular_control}).
	As a potential application, one can consider an optimal stopping version of Lions' \cite{Lions} MFG planning problem, that is, design a MFG (of timing) such that the equilibrium solution follows the given marginal distributions.
	Similar to Ren, Tan, Touzi, and Yang \cite{MFG_Planning}, the principal-agent type arguments can be used in \cite{HeTanZou2021} with our mean-field representation results to solve a MFG planning problem of timing.

	\vspace{0.5em}
	
	Moreover, as a crucial technical step, a stability result was established on Bank--El Karoui's representation \eqref{eq:BankEK_repres_intro};
	that is, when $(Y, f)$ changes slightly, the representation process $L$ changes slightly in some sense.
	In its applications to various stochastic optimization problems, such as the optimal stopping, singular control, and optimal consumption problems,
	process $L$ explicitly induces the optimal solutions.
	Consequently, stability of $L$ implies the stability of the corresponding optimal solutions.
	We will provide some stability results of the optimal stopping times (resp. the optimal control) for a class of optimal stopping (resp. singular control) problems,
	which should also be novel in the optimal stopping/control theory.

	\vspace{0.5em}
	
	The rest of the paper is organized as follows.
	In Section \ref{sec:main}, we formulate our mean-field version of Bank--El Karoui's representation theorem
	and then establish some existence and uniqueness results under different technical conditions.
	Then, we demonstrate how the mean-field representation results can be applied to solve different MFG problems.
	In Section \ref{sec:stability}, we provide a stability result of the classical Bank--El Karoui's representation theorem,
	which constitutes a key technical step for establishing the representation theorem in the first approach,
	and induces stability of the optimal solutions of a class of optimal stopping and singular control problems.
	Finally, some technical proofs of the results are reported in Section \ref{sec:proofs}.

\section{A mean-field version of Bank--El Karoui's representation of stochastic processes and its applications in MFGs}
\label{sec:main}

	We first describe some notations and then formulate our mean-field version of Bank--El Karoui's representation of stochastic processes in Section \ref{subsec:Prelim}.
	In Section \ref{subsec:existence}, we provide some existence and uniqueness results of the mean-field representation under different technical conditions.
	In Section \ref{subsec:applications}, we show how the representation results can be applied to study different MFGs.

\subsection{Preliminaries and formulation of the mean-field representation}
\label{subsec:Prelim}

	Let $(\Om,\Fc,\P)$ be a complete probability space equipped with a filtration $\F \vcentcolon= (\Fc_t)_{t \in [0, T)}$ satisfying the usual conditions, that is, $t \longmapsto \Fc_t$ is right-continuous, and $\Fc_0$ contains all $\P$-null sets in $\Fc$.
	For $T \in [0, +\infty]$, let us denote by $\Tc$ the collection of all $\F$-stopping times taking values in $[0,T]$.

	\vspace{0.5em}
	
	Let $D$ be a (nonempty) Polish space,
	we denote by $\D$ the space of all $D$-valued c\`adl\`ag paths on $[0,T]$ when $T < \infty$, or the space of all $D$-valued  c\`adl\`ag paths on $[0, \infty)$ when $T = \infty$.
	The space $\D$ is equipped with the Skorokhod metric $d_{\D}$.
	For a constant $\eta \in \R \cup\{ -\infty\}$, let $\V^+_{\eta}$ denote the space of all increasing and left-continuous functions $\mathbf{l}: [0,T) \longrightarrow [\eta, \infty)$ such that $\mathbf{l}(0) = \eta$ and $\mathbf{l}$ is $\R$-valued on $(0,T)$.
	The space $\V^+_{\eta}$ is equipped with the L\'evy metric $d_L$ defined as follows:
	\begin{equation*}
		d_L(\mathbf{l}_1, \mathbf{l}_2)
		\vcentcolon=
		\inf \big\{
			\eps \geq 0 ~:
			\mathbf{l}_1\big((t - \eps)\vee 0\big) - \eps \leq \mathbf{l}_2(t),
			~\mathbf{l}_2\big((t - \eps)\vee 0\big) - \eps \leq \mathbf{l}_1(t),
			~\forall t \in (0,T)
		\big\},
	\end{equation*}
	when $T < +\infty$ and
	\begin{equation*}
		d_L(\mathbf{l}_1,\mathbf{l}_2)
		~ \vcentcolon= ~
		\sum_{n = 1}^{+\infty}
		2^{-n}(d_L(\mathbf{l}_1|_{[0,n)},\mathbf{l}_2|_{[0,n)}) \wedge 1),
	\end{equation*}
	when $T = +\infty$,
	so that it is a Polish space (see Appendix \ref{sec:spaceV} for a detailed proof).
	Finally, let us define
	$$
		\V^+_{\circ} ~:=~ \cup_{\eta \in \R} \V^+_{\eta},
		~~\mbox{and}~~
		\V^+ ~:=~ \V^+_{-\infty}.
	$$

	Let $E$ be a (nonempty) Polish space, we denote by $\Pc(E)$ the space of all (Borel) probability measures on $E$.
	The space $\Pc(E)$ is equipped with the weak convergence topology, under which it is also a Polish space.
	We denote by $\L^0_{\Fc}(\Om, E)$ (resp. $\L^0_{\Gc}(\Om, E)$ the space of all $\Fc$-measurable (resp. $\Gc$-measurable)) $E$-valued random variables,
	which is equipped with the topology induced by the convergence in probability.
	Similarly, $\L^0_{\Gc}(\Om, \Pc(E))$ denotes the space of all $\Gc$-measurable random measures,
	and $\L^0_{\Fc}(\Om, \D \x \V^+)$ denotes the space of all random variables taking values in $\D \x \V^+$.

	Given two measurable spaces $(X_1,\Fc_1)$ and $(X_2,\Fc_2)$, together with a probability measure $\mu$ on $X_1$,
	and a measurable mapping $f: X_1 \longrightarrow X_2$,
	we recall that the pushforward measure $f\#\mu$ is defined by
  \begin{equation*}
    f\#\mu (B) := \mu(f^{-1}(B)),
    ~\mbox{for all}~
    B \in \Fc_2.
  \end{equation*}
	We are given a family $(X^m, Y^m, f^m)_{m \in \L^0_{\Gc}(\Om, \Pc( E))}$,
	where for each $m \in \L^0_{\Gc}(\Om, \Pc(E))$, $f^m: [0, T] \x \Om \x \R \longrightarrow \R$,
	$X^m$ is an adapted $D$-valued c\`adl\`ag process, and $Y^m$ is a $\R$-valued optional process.
	Throughout the paper, we assume the following conditions on $(Y^m, f^m)_{m \in \L^0_{\Gc}(\Om, \Pc(E))}$.

	\begin{assumption} \label{assum:Yfm}
		For each $m \in \L^0_{\Gc}(\Om, \Pc(E))$,
		$Y^m$ is a $\R$-valued optional process of class (D) and upper semi-continuous in expectation (in the sense of Definition \ref{def:USCE}) such that $Y^m_T = 0$,
		and $f^m$ satisfies that $\ell \longmapsto f^m(t, \om, \ell)$ is continuous and strictly increasing from $-\infty$ to $\infty$ for all $(t, \om)$,
		and $(t, \om) \longmapsto f^m(t, \om, \ell)$ is progressively measurable with
		\begin{equation} \label{eq:cond_Ym}
			\E \Big[ \int_0^T \big| f^m(t, \ell) \big| dt \Big] < \infty,
			~\mbox{for all}~
			\ell \in \R.
		\end{equation}
	\end{assumption}

	Given a fixed measurable function
	$$
		 \Psi: \Om \x \D \x \V^+ \longrightarrow E,
	$$
	our mean-field version of Bank--El Karoui's representation theorem consists of finding a couple $(L, m)$,
	where $L$ is an optional process and $m \in \L^0_{\Gc}(\Om, \Pc(E))$, satisfying the following:
	\begin{equation} \label{eq:MF_representation}
		Y^m_{\tau}
		=
		\E\bigg[
			\int_{\tau}^{T} f^m \Big(t, \sup_{s \in [\tau,t)} L_s \Big) dt
			\bigg| \Fc_\tau
		\bigg],
		~~\mbox{and}~
		m = \Lc \big( \Psi(X^m, \Lh) \big| \Gc \big),
	\end{equation}
	for all $\tau \in \Tc$ and $\Lh$ being the running maximum process of $L$ defined by, with the convention that $\sup \emptyset = -\infty$,
	$$
		\Lh_t \vcentcolon= \sup_{s \in [0,t)}L_s,
		~~t \in [0,T).
	$$
	Throughout the paper, $\Lc \big(\Psi(X^m, \Lh) \big| \Gc \big)$ denotes the conditional distribution of $\Psi(X^m, \Lh)$ knowing $\Gc$.
	In particular, $\Lc \big(\Psi(X^m, \Lh) \big| \Gc \big) \in \L^0_{\Gc}(\Om, \Pc(E))$.

	\begin{remark}\label{rmk:Lh}
	
		$\mathrm{(i)}$ For the mean-field term $m$ in \eqref{eq:MF_representation}, the (conditional) distribution of the running maximum process $\Lh$ is considered rather than that of $L$ itself.
		The technical reason is that, given an optional process $Y$, the solution $L$ in the classical Bank--El Karoui's representation \eqref{eq:BankEK_repres_intro} may not be unique, but the corresponding running maximum process $\Lh$ is unique.
		Moreover, because process $L$ does not have a priori path regularity, its induced mean-field term would be a probability measure on $\R^{[0,T]}$.
		By considering $\Lh$, the mean-field term becomes a probability measure on $\V^+$, which has better topological (metric) and order structure than $\R^{[0,T]}$.
		Crucially, in all the well-known applications of Bank--El Karoui's representation theorem in stochastic optimization,
		the optimal solutions can be characterized by $\Lh$.
		
		\vspace{0.5em}
		
		\noindent $\mathrm{(ii)}$ 
		The process $X^m$ would be a general stochastic process depending on the interaction term $m$,
		which can serve as an underlying process, so that $(Y^m, f^m)$ are functionals of $X^m$.
		This is also why we include the (conditional) law of $X^m$ in the mean-field term $m$.
	
		\vspace{0.5em}

		As an example of the underlying process in the MFG, $X^m$ could be a diffusion process, with $D = \R^d$, defined by
		\begin{equation} \label{eq:diffusion_Xm}
			X^m_t = X^m_0 + \int_0^t b^m(s, X^m_s) ds + \int_0^t \sigma^m(s, X^m_s) dW_s,
		\end{equation}
		where $(b^m, \sigma^m): \Om \x \R_+ \x \R^d \longrightarrow \R^d \x \S^d$ are coefficient functions such that $(\om, t) \longmapsto (b^m, \sigma^m)(\om, t, x)$ is progressively measurable for all $x \in \R^d$, and $W$ is a standard Brownian motion.
		
		\vspace{0.5em}
		
		\noindent $\mathrm{(iii)}$ In some applications of the classical Bank--El Karoui's representation \eqref{eq:BankEK_repres_intro} in stochastic optimization problems,
		the optimal solution is given by a transformation of the running maximum process $\Lh$.
		This  is our main reason to consider the (conditional) law of $\Psi(X, \Lh)$ in the mean-field term, for some functional $\Psi: \Om \x \D \x \V^+ \longrightarrow E$.
		For example, $\Psi$ can be defined by, with $E = \D \x \V^+$ and the given processes $(X^+, X^-, \Lh^+, \Lh^-): \Om \longrightarrow \D \x \D \x \V^+ \x \V^+$
		such that $X^- \le X^+$ and $\Lh^- \le \Lh^+$, a.s.,
		$$
			\Psi(\om, \mathbf{x}, \mathbf{l}) ~:=~ \big( X^-_t(\om) \vee \mathbf{x}_t \wedge X^+_t(\om), ~\Lh^-_t(\om) \vee \mathbf{l}_t \wedge \Lh_t^+(\om) \big)_{t \in [0,T)}.
		$$
		
		\noindent $\mathrm{(iv)}$ 
		Although we consider a common noise $\sigma$-field $\Gc$ in our formulation,
		it nevertheless covers the setting with common noise filtration $\G = (\Gc_t)_{t \in [0,T]}$ in many MFG problems. 

		Indeed, in many MFGs, $(X^m, Y^m, f^m)$ depends on the mean-field term $m$ in an adaptive way, 
		that is, at each time $t \ge 0$, $(X^m_t, Y^m_t, f^m(t, \cdot))$ depends on the conditional distribution $\Lc( ( X^m_{t\wedge \cdot}, L^m_{t \wedge \cdot}) | \Gc_t)$.
		When $\G$ is a Brownian filtration (or a sub-filtration of $\F$ satisfying the (H)-hypothesis), such that 
		$$
			\Lc( ( X^m_{t\wedge \cdot}, L^m_{t \wedge \cdot}) | \Gc_t)
			=
			\Lc( ( X^m_{t\wedge \cdot}, L^m_{t \wedge \cdot}) | \Gc_T), 
			~\mbox{for all}~
			t \in [0,T],
		$$
		one can deduce the flow of measures $\big( \Lc( ( X^m_{t\wedge \cdot}, L^m_{t \wedge \cdot}) | \Gc_t) \big)_{t \in [0,T]}$
		from the conditional distribution $\Lc( (X^m_{\cdot}, L^m_{\cdot}) | \Gc_T)$ knowing the common noise $\sigma$-field $\Gc = \Gc_T$.
		
		\vspace{0.2em}
		
		For a concrete example of the dependence of$f^m$ on $m$,
		we can consider $m = \Lc(X_{\cdot}| \Gc)$ for some underly process $X$, 
		and with $m_t = \Lc(X_t| \Gc)$, $f^m(t,\om,\ell) := F(t,\om,\ell,m_t(\phi))$ for some functions $F$ and $\phi$, 
		where $m_t(\phi):= \int_\R \phi(x) m_t(dx)$.
		
		\vspace{0.5em}
		
		In particular, when $\Gc = \{\emptyset, \Om\}$, the conditional distribution knowing $\Gc$ becomes a deterministic probability measure.
		Thus, any $\Gc$-measurable random variable is a constant, and $\L^0_{\Gc}(\Om, \Pc(E))$ can be identified as $\Pc(E)$.
	\end{remark}
		
	\begin{remark} \label{rem:YT0}
		In Assumption \ref{assum:Yfm}, it is assumed that $Y^m_T = 0$.
		When $(Y_T^m)_{m \in \L^0_{\Gc}(\Om, \Pc(E))}$ is a family of integrable random variables, one can define the processes $\Yh^m$ by $\Yh^m_t := Y^m_t - \E[\xi^m | \Fc_t]$,
		so that $\Yh^m_T = 0$ and
		the corresponding representation result for $\Yh^m$ in \eqref{eq:MF_representation} is equivalent to the following:
		$$
		Y^m_{\tau}
		=
		\E\bigg[
			Y^m_T +
			\int_{\tau}^{T} f^m \Big(t, \sup_{s \in [\tau,t)} L_s \Big) dt
			\bigg| \Fc_\tau
		\bigg],
		~~\mbox{and}~
		m = \Lc \big( \Psi(X^m, \Lh) \big| \Gc \big).
		$$
	\end{remark}

\subsection{Existence and uniqueness results}
\label{subsec:existence}

	We now provide some existence and uniqueness results on the mean-field representation \eqref{eq:MF_representation} under different (abstract) technical conditions.

\subsubsection{A setting with continuity conditions}

	Recall that we are given a family of tuples $(X^m, Y^m, f^m)_{m \in \L^0_{\Gc}(\Om, \Pc(E))}$, $\Pc(E)$ is equipped with the weak convergence topology, and $\L^0_{\Gc}(\Om, \Pc(E))$ is equipped with the topology of convergence in probability.

	\begin{theorem}\label{thm:continuity}
		Let Assumption \ref{assum:Yfm} hold true,
		assume that for each $m \in \L^0_{\Gc}(\Om, \Pc(E))$, $t \longmapsto Y^m_t$ has almost surely upper semi-continuous paths.
		Suppose also that
		\begin{itemize}
		\item for a countable partition $(A_i)_{i \ge 1}$ of $\Om$ such that $A_i \in \Fc$ and $\P[A_i] > 0$ for each $i\ge 1$, $\cup_{i \ge 1}A_i = \Om$ and $A_i \cap A_j = \emptyset$ for all $i \neq j$,
		one has
		\begin{equation} \label{eq:Gc_countable}
			\Gc ~=~ \sigma(A_i ~: i \ge 1),
		\end{equation}
		so that $\L^0_{\Gc}(\Om, \Pc(E))$ can be identified as the space $(\Pc(E))^{\N}$;
		
		\item the map $(\mathbf{x}, \mathbf{l}) \longmapsto \Psi(\om,\mathbf{x}, \mathbf{l})$ is continuous for all $\om \in \Om$,
		and for some convex compact subset $K \subset \L^0_{\Gc}(\Om, \Pc(E))$, we have
		$$
			\big\{ \Lc\big(\Psi(X^m, \Lh) \big| \Gc \big) ~: m \in \L^0_{\Gc}(\Om, \Pc(E)),~ \Lh \in \L^0_{\Fc}(\Om, \V^+) \big\} \subseteq K;
		$$
		
		\item for all $\ell \in \R$, $\{m^n\}_{n \in \N} \subset K$ with $m^n \longrightarrow m^{\infty}$ in probability
		for some $m^{\infty} \in K$, one has
		$$
			\lim_{n \to \infty}
			\E \bigg[ d_{\D} (X^{m^n}, X^{m^{\infty}})
			+ \sup_{t \in [0,T]}
			\big| Y^{m^n}_{t}
			- Y^{m^{\infty}}_{t} \big|
			+
			\int_0^{T}
			\!\! \big|f^{m^n}(s,\ell) - f^{m^{\infty}}(s,\ell)\big|
			ds
			\bigg]
			=
			0.
		$$
		\end{itemize}
		Then there exists a couple $(L,m )$ as a solution to the mean-field representation \eqref{eq:MF_representation}.
	\end{theorem}

	The proof of Theorem \ref{thm:continuity} is reported in Section \ref{subsec:proof_theorems}, which is mainly based on the Schauder fixed-point theorem (recalled in Appendix \ref{sec:fixed_point}).

	\begin{remark}

		$\mathrm{(i)}$
		The continuity condition on $\Psi$ and $m \longmapsto (X^m, Y^m, f^m)$ is used to verify the continuity condition in  Schauder fixed-point theorem.
		Furthermore, we prove a continuous property of $(Y, f) \longmapsto \Lh$, where $\Lh$ is the running maximum of $L$ in \eqref{eq:BankEK_repres_intro},
		which is a stability result of Bank--El Karoui's representation \eqref{eq:BankEK_repres_intro} (see Section \ref{sec:stability}).

		\vspace{0.5em}

		\noindent $\mathrm{(ii)}$
		By assuming that $\Gc$ is generated by a countable partition of $\Om$, one can identify  $\L^0_{\Gc}(\Om, \Pc(E))$ as $(\Pc(E))^{\N}$, which is a Polish space and a subspace of a Hausdorff locally convex topological vector space,
		as required in Schauder fixed-point theorem.
		This is also the main reason why we assume such a technical condition. 
		This condition would be restrictive and it excludes the case in which $\Gc$ is generated by a Brownian motion.
		However, it is not surprising as our representation result \eqref{eq:MF_representation} stays in a strong formulation such that the probability space $(\Om, \Fc, \P)$ with the filtration $\F$ and the sub-$\sigma$-field $\Gc$ is fixed.
		This is also the main difficulty in the literature of the MFG with common noise.
		Nevertheless, our setting can be considered as a first step for studying mean-field problems with general common noise.
		For more general countably generated $\sigma$-field $\Gc$, a further approximation technique can be applied to obtain existence result.
		See, for example, Carmona, Delarue, and Lacker \cite{CarmonaDelarueLacker2016} for a MFG with common noise.
		However, this usually leads to a weak formulation of the mean-filed problem at the limit, which is out of the scope of this paper.
 	\end{remark}

	\begin{example}\label{eg:continuity}
		Let $E = \D \x \V^+$, with $\Lh^-, \Lh^+ $ being two stochastic processes with paths in $\V^+$ such that $\Lh^-_t \le \Lh^+_t$, a.s. for all $t \in [0,T)$,
		we define $\Psi: \Om \x \D \x \V^+ \longrightarrow E$ by
		$$
			\Psi(\om, \mathbf{x}, \mathbf{l}) := \big(\mathbf{x}_t, \Lh^-_t(\om) \vee \mathbf{l}_t \wedge \Lh^+_t(\om) \big)_{t \in [0,T)}.
		$$
		which satisfies the continuity condition of $\Psi$ in Theorem \ref{thm:continuity}.
		
		\vspace{0.5em}
		
		Furthermore, let $D = \R^d$, and $\{X^m\}_{m \in \L^0_{\Gc}(\Om, \Pc(E))}$ be the diffusion processes defined as the unique solution of \eqref{eq:diffusion_Xm}
		with coefficient functions $(b^m, \sigma^m): \Om \x \R_+ \x \R^d \longrightarrow \R^d \x \S^d$,
		which are uniformly bounded and uniformly Lipschitz continuous in $x$.
		Assume also that $(X^m_0)_{m \in \L^0_{\Gc}(\Om, \Pc(E))}$ is uniformly bounded,
		and for $m^n \longrightarrow m^{\infty}$ in probability, one has $(b^{m^n}, \sigma^{m^n}) \longrightarrow (b^{m^{\infty}}, \sigma^{m^{\infty}})$ uniformly.
		Then it is standard to check that
		$$
			\big\{ \P \circ (X^m)^{-1} ~: m \in \L^0_{\Gc}(\Om, \Pc(E)) \big\} ~\mbox{is tight in}~\Pc(\D),
			~~\mbox{and}~~
			\E\big[d_{\D} (X^{m^n}, X^{m^{\infty}}) \big] \longrightarrow 0.
		$$
		Next, we can further obtain that (see also Section \ref{subsec:proofs_examples} for a brief proof)
		$$
			\big\{ \Lc \big( \Psi(X^m, \widehat L) \big| \Gc \big) ~: m \in \L^0_{\Gc}(\Om, \Pc(E)), ~\widehat L \in \L^0_{\Fc}(\Om, \V^+) \big\} ~\mbox{is tight in}~ \L^0_{\Gc}(\Om, \Pc(E)),
		$$
		so that finding a convex compact set $K$ as required in Theorem \ref{thm:continuity} is easy.
	\end{example}

	\begin{example}\label{eg:continuity2}
		Let $\phi: \R \longrightarrow \R$ be an increasing, continuous function, uniformly bounded by some constant $C> 0$,
		and $E = \V_{\circ}$ be the space of all $\R$-valued bounded increasing paths on $[0,T)$,
		we define $\Psi: \Om \x \D \x \V^+ \longrightarrow E$ by
		$$
			\Psi(\om, \mathbf{x}, \mathbf{l}) := (\phi(\mathbf{l}_t))_{t \in [0, T)}.
		$$
		Because $\phi$ is uniformly bounded by $C$, for any processes $(X, \Lh) \in \L^0_{\Fc}(\Om, \D) \x \L^0_{\Fc}(\Om, \V^+)$, the induced process $\Psi(X, \Lh)$ is increasing and uniformly bounded by $C$.
		Under the L\'evy metric on $\V$, we can check that
		$$
			\big\{ \Lc \big( \Psi(X^m, \widehat L) \big| \Gc \big) ~: m \in \L^0_{\Gc}(\Om, \Pc(E)), \widehat L \in \L^0_{\Fc}(\Om, \V^+) \big\} ~\mbox{is tight}.
		$$
	\end{example}

\subsubsection{A setting with monotonicity conditions}

	In a second setting, where $E$ is a partially ordered Polish space, one can apply Tarski's fixed-point theorem to obtain an existence result under some monotonicity condition.
	Notice that the definition of the partial order, partially ordered Polish space,
	the corresponding complete lattice as well as Tarski's fixed-point theorem are recalled in Appendix \ref{sec:fixed_point}.
	
	\vspace{0.5em}
	
	Let $E$ be a partially ordered Polish space with partial order $\le_E$, a function $\phi: E \longrightarrow \R$ is said to be increasing if $\phi(e_1) \le \phi(e_2)$ for all $e_1, e_2 \in E$ such that $e_1 \le_E e_2$.
	We introduce  a partial order $\le_p$ on $\L^0_{\Gc}(\Om, \Pc(E))$ as follows:
	for all $m^1, m^2 \in \L^0_{\Gc}(\Om, \Pc(E))$, we say
	\begin{equation} \label{eq:def_lep}
		m^1 \le_{p} m^2
		~~\mbox{if and only if}~
		\int_E \phi(x) m^1(dx)
		\le
		\int_E \phi(x) m^2(dx),
		~\mbox{a.s.,}
	\end{equation}
	for all $\R$-valued bounded increasing measurable functions $\phi: E \longrightarrow \R$.

	\begin{theorem}\label{thm:monotonicity}
		Let Assumption \ref{assum:Yfm} hold true,
		$\D$ be equipped with a partial order $\le_{\D}$ such that $(\D, \le_{\D})$ becomes a partially ordered Polish space,
		$E$ be a partially ordered Polish space with partial order $\le_E$,
		and $\Psi: \Om \x \D \x \V^+ \longrightarrow E$ satisfy that,
		$$
			\mathbf{x}^1 \le_{\D} \mathbf{x}^2,~ \mathbf{l}^1_t \le \mathbf{l}^2_t, ~t \in [0,T)
			~\Longrightarrow~
			\Psi(\om, \mathbf{x}^1, \mathbf{l}^1) ~\le_E~  \Psi(\om, \mathbf{x}^2, \mathbf{l}^2).
		$$
		Suppose in addition that there exists a complete lattice $(K, \le_p)$ as subset of $\L^0_{\Gc}(\Om, \Pc(E))$ such that,
		for all $m^1, m^2 \in K$,
		\begin{equation} \label{eq:XYf_order}
			m^1 \leq_p m^2
			~\Longrightarrow~
			X^{m^1} \le_{\D} X^{m^2},\mbox{a.s.},
			~
			Y^{m^1} \!-\! Y^{m^2}~\mbox{is supermartingale and}~
			f^{m^1}(\cdot) \ge f^{m^2}(\cdot).
		\end{equation}
		Then there exists a couple $(L, m)$ being a solution to the mean-field representation \eqref{eq:MF_representation}.
	\end{theorem}
	
	\begin{remark}
		The technical supermartingale condition in \eqref{eq:XYf_order} is adapted from Assumption B in Carmona, Delarue and Lacker \cite{CarmonaDelarueLacker2017}.
		Similar conditions are typically used in submodular games, see, e.g., Topkis \cite{Topkis1998},
		or Dianetti, Ferrari, Fischer, and Nendel \cite{DianettiFerrari2022}.
	\end{remark}

	We next provide two examples of $\Psi$ and the complete lattice $(K, \le_p)$ in the setting of Theorem \ref{thm:monotonicity},
	where the corresponding proofs are reported in Section \ref{subsec:proofs_examples}.

	\begin{example} \label{exam:orderedVD}
	
		Let $C > 0$ be a constant,
		$$
			E := \{ (v^1, v^2) \in \V^+\x \V^+ ~: |v^1(t)| \le C, ~|v^2(t)| \le C ~t \in [0,T)\},
		$$
		we introduce a partial order $\le_E$ by
		$$
			\mathbf{l}^1 \le_E \mathbf{l}^2 ~~\mbox{if}~~ \mathbf{l}^1_t \le \mathbf{l}^2_t, ~\mbox{for all}~t \in [0,T).
		$$
		Then $(E, \le_E)$ is a complete lattice and the corresponding space $\L^0_{\Gc}(\Om, \Pc(E))$ with the induced order $\le_p$ in \eqref{eq:def_lep} is also a complete lattice.
		
		\vspace{0.5em}
		
		Furthermore, given  an increasing and left continuous function $\phi: \R \longrightarrow \R$ uniformly bounded by $C > 0$,
		we define
		$$
			\Psi(\om, \mathbf{x}, \mathbf{l}) := \Big( \Big( \sup_{s \in [0,t)}\phi(\mathbf{x}_s) \Big)_{t \in[0,T)}, \big( \phi(\mathbf{l}_t) \big)_{t \in[0,T)} \Big).
		$$
		Let us consider the partial order $\le_{\D}$ on $\D$ by $\mathbf{x}^1 \le_{\D} \le \mathbf{x}^2$ if $\mathbf{x}^1_t \le \mathbf{x}^2_t$ for all $t \in [0,T]$,
		then it is clear that $\Psi$ satisfies the monotonicity condition in Theorem \ref{thm:monotonicity}.
	\end{example}

	\begin{example}\label{eg:monotonicity}

		Let $E = \V^+$, $\Lh^-, \Lh^+$ be two stochastic processes with paths in $\V^+$ such that $\Lh^-_t \le \Lh^+_t$, a.s. for all $t \in [0,T)$,
		and $\Psi: \Om \x \D \x \V^+ \longrightarrow E$ be defined by
		$$
			\Psi(\om, \mathbf{x}, \mathbf{l}) := \big( \Lh^-_t(\om) \vee \mathbf{l}_t \wedge \Lh^+_t(\om) \big)_{t \in [0,T)}.
		$$
		Then it is direct to check that $\Psi$ satisfies the monotonicity condition in Theorem \ref{thm:monotonicity}.
		
		\vspace{0.5em}
	
		Further, on the space $\V^+$ and $\L^0_{\Fc}(\Om, \V^+)$, let us introduce respectively the partial order $\le_v$ and $\le_l$ as follows:
		$$
			\mathbf{l}^1 \le_v \mathbf{l}^2 ~~\mbox{if}~~ \mathbf{l}^1_t \le \mathbf{l}^2_t ~\mbox{for all}~ t \in [0,T),
		$$
		and
		$$
			\Lh^1 \le_l \Lh^2 ~~\mbox{if and only if}~~ \Lh^1_t \le_v \Lh^2_t,~\mbox{a.s. for all}~t \in [0,T).
		$$
		Let
		$$
			\L_0
			~:=~
			\big \{
			L \in \L^0_\Fc(\Om, \V^+):
			\Lh^- \leq_l L \leq_l \Lh^+,
			\big \}.
		$$
		Then, $(\L_0, \le_l)$ is a complete lattice.
		Moreover, let
		$$
			K := \big\{ \Lc(\Lh| \Gc) ~: \Lh \in \L_0 \big\}.
		$$
		Then, with the partial order $\le_p$ induced by $(E = \V^+, \le_v)$ as in \eqref{eq:def_lep}, the space $(K, \le_p)$ is also a complete lattice.

    \end{example}

	\begin{example}\label{exam:Yf_order}
		Let $T < \infty$, $E = \R$, and $\phi: \R \longrightarrow \R$ be a monotone function,
		$\fh: [0,T] \x \Om \x \R \x \R \longrightarrow \R$ be such that $x \longmapsto \fh(t, \om, \ell, x)$ is decreasing,
		$g: [0,T] \x \Om \x \R \longrightarrow \R$ be such that $g(\cdot, x)$ is $\F$-adapted for all $x \in \R$, and $x \longmapsto g(\cdot, x)$ is decreasing,
		$h: [0,T] \x \Om \x \Pc(\R) \longrightarrow \R$ be such that $h(\cdot, m)$ is $\F$-adapted  for all $m \in \L^0_{\Gc}(\Om, \Pc(\R))$.
		Moreover, assume that
		\begin{equation*}
			\E\bigg[\int_{0}^{T}
			\big(|g(s, x)| + h^2(s,m)\big)ds
			\bigg] < +\infty.
		\end{equation*}
		Let $f^{m}(\cdot)$ and $Y^m$ be given as follows:
		\begin{equation*}
			f^m(\cdot) := \fh(\cdot, \langle\phi,m\rangle),
			~~~
			Y^m_t := Y^m_0 + \int_{0}^{t}g(s, \langle\phi,m\rangle)ds
			+ \int_{0}^{t}h(s, m)dW_s,
			~~t \in [0,T].
		\end{equation*}
		Then it is easy to verify that $f^m$ and $Y^m$ satisfy the property \eqref{eq:XYf_order}.
	\end{example}

\subsubsection{A setting with the dimension-reduction structure condition }

	We finally consider a case with special structure conditions on $f$, such that one can obtain existence and uniqueness result of the representation \eqref{eq:MF_representation}.
	
	\vspace{0.5em}
	
	Let $E = \V^+$, $\Gc = \{\emptyset, \Om\}$, so that $\L^0_{\Gc}(\Om, \Pc(E))$ can be identified as the space $\Pc(E)$.
	Let us take $E = \V^+$ as canonical space with canonical process $\Lh$, for $m \in \Pc(E)$, 
	and define
	$$
		m_t(dx)  := m \circ \Lh_t^{-1} \in \Pc(\R), ~\mbox{for all}~t \in [0,T),
	$$
	and recall that
	$$
		m_t(\phi) := \int_\R \phi(x) m_t(dx),
		~\mbox{for all bounded and measurable functions}~ 
		\phi: \R \longrightarrow \R.
	$$
	
	\begin{theorem}\label{thm:contraction}
		Let Assumption \ref{assum:Yfm} hold true, and $\Psi: \Om \x \D \x \V^+$ be defined by
		$$
			\Psi(\om, \mathbf{x}, \mathbf{l}) := \mathbf{l}, ~~\mbox{for all}~ (\om, \mathbf{x}, \mathbf{l}) \in \Om \x \D \x \V^+.
		$$
		Further, for some stochastic processes $\widetilde X$ and $\widetilde Y$, together with $\widetilde f: \Om \x [0,T) \x \R \longrightarrow \R$, one has
		$$
			(X^m, Y^m) = (\widetilde X, \widetilde Y),
			~\mbox{and}~
			f^m(\om, t, \ell) := \widetilde f \big(\om, t, \ell - m_t(\phi) \big),
			~~\mbox{for all}~m \in \Pc(E),
		$$
		where $\phi: \R \longrightarrow \R$ is bounded differentiable satisfying $\phi'(x) \in [0,1)$ for all $x \in \R$.
		The optional process $L$ in the representation theorem \eqref{eq:BankEK_repres_intro} of $\Yt$ w.r.t. $\ft$ has almost surely non-decreasing paths.

		\vspace{0.5em}

		Then there exists a unique solution $(L^m, m)$ to the mean-field representation \eqref{eq:MF_representation}.
	\end{theorem}

	\begin{remark}
		With the above special structure, the representation problem  \eqref{eq:MF_representation} can be reduced to the following fixed point problem:
		for each $t \in [0,T]$, one looks for $y_t \in \R$ such that
		$$
			y_t = \E[\phi(L_t + y_t)].
		$$
		By considering the above fixed-point problem, one can deduce the existence and uniqueness result (see its proof in Section \ref{subsec:proof_theorems}).

	\end{remark}

\subsection{Applications in the mean-field games}
\label{subsec:applications}

	We now present some first applications of the above mean-field representation results in the mean-field games (MFGs) to obtain some existence results.
	In particular, we stay in a strong formulation in the sense that we fix a filtered probability space $(\Om, \Fc, \F, \P)$ and a common noise $\sigma$-field $\Gc \subset \Fc$.

\subsubsection{A MFG of timing with multiple populations}
\label{subsubsec:MFG_timing}

	We consider a MFG with possibly infinitely number of populations indexed by $\ell \in \R$,
	where each agent solves an optimal stopping problem with time horizon $T < \infty$.
	The mean-field (interaction) term is then a measure on $[0,T]^{\R}$, which is equipped with the product topology.
	Recall that $\D$ represents the space of all $D$-valued c\`adl\`ag paths on $[0,T]$, with a Polish space $D$.
	We denote by $\Pc( \D \x [0,T]^{\R})$ the collection of all Borel probability measures on $\D \x [0,T]^{\R}$,
	and $\L^0_\Gc\big(\Om;\Pc\big(\D \x [0,T]^\R\big)\big)$ denotes the space of all
	$\Gc$-measurable random measures $m: \Om \longrightarrow \Pc\big(\D \x [0,T]^\R\big)$.
	
	\vspace{0.5em}
	
	For a given interaction term $\mu \in \L^0_\Gc\big(\Om;\Pc\big(\D \x [0,T]^\R\big)\big)$, there is an underlying process $X^{\mu}$ having paths in $\D$.
	Further, let $\Tc$ denote the collection of all $\F$-stopping times taking value in $[0,T]$,
	$G: [0,T] \x \Om \x \L^0_\Gc\big(\Om;\Pc\big(\D \x [0,T]^\R\big)\big) \longrightarrow \R$
	and $g_\ell: [0,T] \x \Om \x \L^0_\Gc\big(\Om;\Pc\big(\D \x [0,T]^\R\big)\big) \longrightarrow \R$, $\ell \in \R$, be the reward functions.
	In each population $\ell \in \R$, a representative agent solves the following optimal stopping problem:
	\begin{equation}\label{eq:optimal_stopping}
		\sup_{\tau \in \Tc}
		~J_\ell(\tau, \mu),
		~\mbox{with}~
		J_\ell(\tau, \mu) :=
		\E\Big[
			\int_{0}^{\tau} g_\ell(t,\mu) dt  + G_{\tau} (\mu)
		\Big].
	\end{equation}

	\begin{definition}\label{def:MFE_optimal_stopping}
		$\mathrm{(i)}$ Let $\eps \ge 0$, a $\eps$-solution of the MFG of timing is a couple $\{\mu^*, (\tau^*_\ell)_{\ell \in \R}\}$, such that
		$\mu^* = \Lc\big( \big(X^{\mu^*}, (\tau^*_\ell)_{\ell \in \R}\big) \big| \Gc \big)$, and for each  $\ell \in \R$, $\tau^*_\ell \in \Tc$ satisfies that
		\begin{equation*}
			J_\ell(\tau^*_\ell, \mu^*)
			~ \geq ~
			\sup_{\tau \in \Tc} J_\ell(\tau,\mu^*) - \varepsilon.
		\end{equation*}

		\noindent $\mathrm{(ii)}$ When $\eps = 0$, a $\eps$-solution is also called a solution to the MFG of timing.
	\end{definition}

	Let us refer to Nutz \cite{Nutz2018}, and Carmona, Delarue and Lacker \cite{CarmonaDelarueLacker2017} for a detailed interpretation/justification of the formulation of the MFG of timing, as well as its solution as a Nash equilibrium.

	\begin{assumption}\label{ass:Optimal_stopping}
		$\mathrm{(i)}$
		For all $(t, \om, \mu) \in [0,T] \x \Om \x \Pc\big(\D \x [0,T]^\R\big)$,
		the map $\ell \longmapsto g_{\ell}(t, \om, \mu) $ is strictly increasing and continuous.

		\vspace{0.5em}

		\noindent $\mathrm{(ii)}$
		For each $\mu \in \L^0_{\Gc}\big(\Om, \Pc\big(\D \x [0,T]^\R\big)\big)$ and $\ell \in \R$,
		the process $(t, \om) \longmapsto g_\ell(t, \om, \mu)$ is progressively measurable and
		$$
			\E\bigg[\int_{0}^{T}|g_\ell (t,\mu)| dt\bigg] < +\infty,
		$$
		and the process $t \longmapsto G_t(\mu)$ is $\F$-optional of class(D) and u.s.c. in expectation (see Definition \ref{def:USCE}).
	\end{assumption}


	The MFG of timing in Definition \ref{def:MFE_optimal_stopping} is defined with a family of optimal stopping problems given a fixed interaction term $\mu$.
	Let us first recall how the classical Bank-El Karoui's representation provides the solutions of a family of optimal stopping problems.

	\begin{theorem}[Bank and El Karoui \cite{BankKaroui2004}]
		Let Assumption \ref*{ass:Optimal_stopping} hold true, and $\mu \in \L^0_\Gc(\Om;\Pc(\D \x [0,T]^\R))$ be fixed.
		Then there exists an optional process $L^\mu$ such that 
		$$
			G_\tau(\mu) = \E\bigg[ G_T(\mu) + \int_{\tau}^{T} g_{\ell} (t, \mu) \big|_{\ell = \sup_{s \in [\tau,t)}L^{\mu}_s} dt \Big| \Fc_\tau\bigg],
			~\mbox{a.s., for all}~
			\tau \in \Tc.
		$$
		Moreover, for each $\ell \in \R$,
		the stopping times
		\begin{equation*}
			\tau_{\ell}
            ~\vcentcolon=~
            \inf \big\{ t\geq 0 ~: L^\mu_{t-} \geq \ell \big\},
			~\mbox{and}~
            \tauh_{\ell}
            ~\vcentcolon=~
            \inf \big\{ t\geq 0 ~: L^\mu_{t-} > \ell \big\}
		\end{equation*}
		are respectively the smallest solution and the biggest solution of the optimal stopping problem \eqref{eq:optimal_stopping}.
	\end{theorem}

	As the solutions of the optimal stopping problems \eqref{eq:optimal_stopping} is given as hitting times of $L^{\mu}$,
	the (conditional) distribution of $(\mu_{\ell})_{\ell \in \R}$ can be then deduced from the (conditional) distribution of $L^{\mu}$.
	It is then natural to deduce solutions of MFG of timing from the mean-field representation results in Section \ref{subsec:existence}.
	
	\vspace{0.5em}	

	Nevertheless, although the representation process $L^\mu$ provides optimal solutions to a family of optimal stopping problems,
	it is only true under the strict increase condition on $\ell \longmapsto g_{\ell}(\cdot)$ which may be restrictive in practice.
 	As consequence, this increase condition will also be assumed in the following existence results for the MFG of timing. 
        
  \vspace{0.5em}
        
	Let us first provide an existence result for the MFG of timing in Definition \ref{def:MFE_optimal_stopping}, 
	which is based on Theorem \ref{thm:continuity} and under some continuity conditions on $\big( X^{\mu}, g_{\ell}(t,\mu), G_t(\mu) \big)$ in $\mu$.

	\begin{proposition}\label{prop:optimal_stopping_application}
		Let Assumption \ref{ass:Optimal_stopping} hold true, $\Gc$ be generated by a countable partition of $\Om$ as in \eqref{eq:Gc_countable},
		and the map $t \longmapsto G_t(\mu)$ has almost surely left upper semi-continuous paths for each $\mu \in \L^0_{\Gc}\big(\Om, \Pc\big(\D \x [0,T]^\R\big)\big)$.
		Suppose in addition that

		\vspace{0.5em}

		\noindent $\mathrm{(i)}$ The map $\ell \longmapsto g_\ell$ is uniformly continuous in the sense that,
		for any $\varepsilon > 0$,
		there exists a $\delta_\varepsilon > 0$ independent of $\om$ and $\mu$
		such that
		$$
			|g_{\ell_1}(t, \om, \mu) - g_{\ell_2}(t, \om, \mu)| ~ \leq ~ \varepsilon,
			~~\mbox{whenever}~
			|\ell_1 - \ell_2| \leq \delta_\varepsilon.
		$$

		\noindent $\mathrm{(ii)}$ The set $\big\{  \Lc\big( X^\mu  \big| \Gc \big): \mu \in \L^0_{\Gc}\big(\Om, \Pc\big(\D \x [0,T]^\R\big)\big)\big\}$ is tight in the space $\L^0_{\Gc}(\Om, \Pc(\D))$.

 		\vspace{0.5em}

		\noindent $\mathrm{(iii)}$
		for $\mu^n, \mu^\infty \in \L^0_{\Gc}\big(\Om;\Pc\big(\D \x [0,T]^\R\big)\big)$
		with
		$\lim_{n \to \infty} \mu^n = \mu^\infty$ almost surely,
		we have for all $\ell \in \R$,
		$$
			\lim_{n \to \infty}
    			\E \bigg[
				d_{\D} (X^{\mu^n}, X^{\mu^\infty})
				+
				\sup_{t \in [0,T]}
				\big| G_t(\mu^n) - G_t(\mu^\infty) \big|
				+
				\int_0^{T}
				\big| g_\ell(t, \mu^n\big) - g_\ell(t, \mu^\infty\big) \big| ds
			\bigg]
			=  0.
		$$
		Then for any $\varepsilon > 0$, there exists a $\varepsilon$-solution 
		$(\mu^*,(\tau^*_\ell)_{\ell \in \R})$ to the MFG of timing in Definition \ref{def:MFE_optimal_stopping}, where $(\tau^*_{\ell})_{\ell \in \R}$ is defined, with some optional process $L^*$, by
		$$
			\tau^*_{\ell} = \inf \{t \ge 0 ~: L^*_t \ge \ell \},
			~~\mbox{for all}~ \ell \in \R.
		$$
	\end{proposition}

	The proof will be reported in Section \ref{sec:proofs_applications}.

	\begin{remark}
		The above result gives only existence of $\eps$-solution, but not solution to the MFG.
		The main reason is that the first hitting time $\tau_{\ell}$ of a process $\Lh$ to some level $\ell \in \R$ is generally not continuous w.r.t. the paths of $\Lh$.
		We therefore need to modify the process $\Lh$ slightly to make its paths be strictly increasing, so that the corresponding hitting time becomes continuous w.r.t. the path.
		Consequently, an equilibrium in terms of the distribution of $\Lh$ can only induce a $\eps$-equilibrium in terms of the distribution of $\tau_{\ell}$.
		Nevertheless, as a $\eps$-solution in the strong sense, it consists still of a novel result, see more discussions in Remark \ref{rmk:eps-mfe}.
	\end{remark}

	We next consider a setting with monotonicity conditions.
	Let $\le_{\D}$ be a partial order on $\D$ such that $(\D, \le_\D)$ is a partially ordered Polish space (see Appendix \ref{sec:fixed_point} for its definition and Example \ref{exam:orderedVD} for an example).
	We first introduce a partial order on the space $\L^0_\Gc(\Om, \Pc(\D \x [0,T]^\R))$,
	for $\mu^1, \mu^2 \in \L^0_\Gc\big(\Om, \Pc \big(\D \x [0,T]^\R\big)\big)$,
	we say $\mu^1 \le_{st} \mu^2$ if
	for all $\Pi_{\ell \in \R}[t_\ell,T] \subset \Pi_{\ell \in \R}[0,T]$ with $t_\ell \in [0,T]$, $\ell \in \R$, and $t_\ell \neq 0$ for only finitely many $\ell$,
	\begin{equation*}
		\int_{\D} \phi( e ) \mu^1 \big( d e,\Pi_{\ell \in \R}[t_\ell, T] \big)
		\le
		\int_{\D} \phi(e) \mu^2 \big(d e,\Pi_{\ell \in \R}[t_\ell, T] \big),
		~\mbox{a.s.,}
	\end{equation*}
	for all bounded increasing functions $\phi: \D \longrightarrow \R$.

	\begin{proposition}\label{prop:optimal_stopping_application_order}
		Let Assumption \ref{ass:Optimal_stopping}  hold true,
		and assume in addition that,

		\noindent $\mathrm{(i)}$
		there exists a bounded increasing map $\varphi: \D \longrightarrow I_c$, with $I_c$ being a closed interval of $\R$,
		such that for any $\mu^1,\mu^2 \in \L^0_{\Gc}(\Om;\Pc(\D \x [0,T]^\R))$,
		\begin{equation}\label{eq:assumption_timing_order}
			(\varphi \otimes \mathrm{Id}) \# \mu_1 = (\varphi \otimes \mathrm{Id}) \# \mu_2
    			~\Longrightarrow~
			\big( X^{\mu_1}, G(\mu^1), g_{\cdot}(\cdot, \mu^1) \big) = \big(X^{\mu_2}, G(\mu^2), g_{\cdot}(\cdot, \mu^2) \big).
		\end{equation}

		\noindent $\mathrm{(ii)}$
		for all $\mu^1,\mu^2 \in \L^0_{\Gc}(\Om;\Pc(\D \x [0,T]^\R))$,
		$(\varphi \otimes \mathrm{Id}) \# \mu_1 \le_{st} (\varphi \otimes \mathrm{Id}) \# \mu_2$ implies
		\begin{equation*}
		X^{\mu^1} \le_{\D} X^{\mu^2},
		~
			G(\mu^1) - G(\mu^2)
		~\mbox{is a supermartingale and}~
			g_\cdot(\cdot, \mu^1) \ge g_\cdot(\cdot, \mu^2).
		\end{equation*}

		\noindent $\mathrm{(iii)}$
		for any $m \in \L^0_{\Gc}(\Om;\Pc(I_c \x [0,T]^\R))$, there exists a $\mu \in \L^0_{\Gc}(\Om;\Pc(\D \x [0,T]^\R))$ such that
		$$
			m = (\varphi \otimes \mathrm{Id}) \# \mu.
		$$

		Then there exists a solution $(\mu^*,(\tau^*_\ell)_{\ell \in \R})$ to the MFG of timing in Definition \ref{def:MFE_optimal_stopping}, 
		where $(\tau^*_{\ell})_{\ell \in \R}$ is defined, with some optional process $L^*$, by
		$$
			\tau^*_{\ell} = \inf \{t \ge 0 ~: L^*_t \ge \ell \},
			~~\mbox{for all}~ \ell \in \R.
		$$
	\end{proposition}

    We also refer to Example \ref{exam:Yf_order} for an example satisfying Condition $\mathrm{(ii)}$ in Proposition \ref{prop:optimal_stopping_application_order}.

	\begin{remark}\label{rmk:timing_finite_population}
		Notice that the formulation of our MFG of timing with infinitely many population in Definition \ref{def:MFE_optimal_stopping} as well as the results in Propositions \ref{prop:optimal_stopping_application} and \ref{prop:optimal_stopping_application_order} cover the case with finitely many populations (or one population).
		Indeed, given finitely many $(g_i)_{i = 1}^n$, let us define a family of functionals $(g_\ell)_{\ell \in \R}$ as follows:
		for all $\ell \in \R$, $(t, \om, \mu) \in [0,T] \x \Om \x \L^0_{\Gc} \big( \Pc\big(\D\x[0,T]^n\big) \big)$,
		\begin{align} \label{eq:f_interpolation}
		g_\ell(t, \om, \mu) ~:=
		\begin{cases}
			g_1(t, \om, \mu) + (\ell - 1),
			& \ell \in (-\infty,1],
				\\
			(i+1 - \ell) g_i(t, \om, \mu) + (\ell - i) g_{i+1}(t, \om, \mu),
			~& \ell \in [i, i+1], ~1 \le i \le n - 1, 
				\\
			g_n(t, \om, \mu) + (\ell - n),
			&\ell \in [n,+\infty).
		\end{cases}
	    \end{align}
	        Further, a function defined on $\L^0_{\Gc} \big( \Pc\big(\D\x[0,T]^n\big) \big)$ can be easily extended on $\L^0_{\Gc} \big( \Pc\big(\D\x[0,T]^{\R}\big) \big)$
	        by using the projection $\pi((t_\ell)_{\ell \in \R}) = (t_i)_{i = 1}^n$.
	\end{remark}

	\begin{remark}
	\label{rmk:eps-mfe}
	The MFG of timing has already been studied considerably in the literature.
	As first works, Nutz \cite{Nutz2018} provided some examples of the MFG of timing with explicit solution,
	and Carmona, Delarue and Lacker  \cite{CarmonaDelarueLacker2017} studied general MFG of timing with common noise.
	They obtained a strong existence result under submodularity and a weak existence result under a general continuity assumption, and studied the corresponding limit theory.
	Let us also mention the work of  Dianetti, Ferrari, Fischer, and Nendel \cite{DianettiFerrari2022} for submodular MFG with common noise,
	which includes the MFG of timing.
	Further, Bertucci \cite{Bertucci2018, Bertucci2020} studied the problem with PDE arguments.
	In \cite{BouveretDumitrescuTankov2020}, Bouveret, Dumitrescu and Tankov studied the existence and uniqueness of the MFG of timing in a setting with diffusion underlying process, by a relaxed formulation approach.
	This has been further extended to the setting with control and optimal stopping in \cite{DumitrescuLeutscherTankov2021}.
	Recently, Burzoni and Campi  \cite{BurzoniCampi2022} studied a MFG with control and optimal stopping in a setting with absorption and common noise.

	\vspace{0.5em}
	
	While all the above works stay in a setting with one population.
	our mean-field representation results allow one to automatically consider a setting with multiple populations indexed by $\ell \in \R$.

	\vspace{0.5em}
	
	Further, without the monotonicity conditions, most of the above works studied the weak/relaxed solutions of the MFG, in the sense that the probability space may not be fixed so that the strategy is randomized as stated in \cite{CarmonaDelarueLacker2017}.
	Our results provide solutions in the strong sense as the probability space is fixed. In other words, the corresponding strategies are pure.
	
	\vspace{0.5em}
	
	Under the monotonicity conditions, our Proposition \ref{prop:optimal_stopping_application_order} can be considered as an extension of Theorem 3.5 in \cite{CarmonaDelarueLacker2017}, or Dianetti, Ferrari, Fischer, and Nendel \cite{DianettiFerrari2022},
	to the case with infinitely many populations.
	\end{remark}


\subsubsection{A MFG with singular control}

	We next study a MFG problem with singular control (based on the  singular control problem in Bank \cite{Bank2004}), in a setting with possibly many populations.

	\vspace{0.5em}

	Let $(\thetau_i)_{i \in \N} \in \R^\N$ be a sequence of real constants,
	for each $i \ge 1$, we denote by $\V^+_{\thetau_i}$ the space of all increasing and left continuous functions $v$ on $[0,T)$ such that $v(0) = \thetau_i$.
	Recall also that $\D$ denotes the space of all $D$-valued c\`adl\`ag paths on $[0,T]$ if $T< \infty$, or on $[0,T)$ if $T = \infty$.
	The mean-field interaction in this MFG is a $\Gc$-measurable random measure $\mu \in \L^0_{\Gc}(\Om, \Pc(\D \x \Pi_{i \in \N}\V^+_{\thetau_i}))$.	
	
	\vspace{0.5em}
	
	Next, for each $\mu \in \L^0_{\Gc}(\Om, \Pc(\D \x \Pi_{i \in \N}\V^+_{\thetau_i}))$, there exists a $D$-valued c\`adl\`ag processes $X^{\mu}$.
	Then for each population $i \in \N$, a representative agent solves the following singular control problem:
	\begin{equation}\label{eq:Bank_singular_control}
		\inf_{\Theta \in \Ab_i} J(\Theta, \mu),
		~~\mbox{with}~
		J(\Theta, \mu)
		:=
		\E\bigg[ \int_{0}^{T}c(t, \mu, \Theta_t) dt
				+ \int_{0}^{T}k(t,\mu) d\Theta_t
		\bigg],
	\end{equation}
	where $c:[0,T] \x \Om \x \L^0_{\Gc}(\Om, \Pc(\D \x \Pi_{i \in \N}\V^+_{\thetau_i})) \x \R \longrightarrow \R$ represents the running cost of the problem,
	$k: [0,T] \x \Om \x \L^0_{\Gc}(\Om, \Pc(\D \x \Pi_{i \in \N}\V^+_{\thetau_i})) \longrightarrow \R$ is the cost related to the control $\Theta$,
	and with an optional process $\overline \Theta^i$ having paths in $\V^+_{\thetau_i}$,
	$$
		\Ab_i \vcentcolon= \big\{
			\mbox{Optional processes}~\Theta ~\mbox{having paths in}~ \V^+_{\thetau_i} ~\mbox{such that}~
			\Theta_t \leq \overline \Theta^i_t
			~\mbox{for all}~
			t \geq 0, ~\mbox{a.s.}
		\big\}.
	$$
	Let us refer to  \cite{Bank2004} for the the motivation and application of the above singular control problem.
	Notice that for different populations, the cost functions $c$ and $k$ are common, but the constraints process $\overline \Theta^i$ are different.

	\begin{definition}\label{def:MFE_singular_control}
		A solution to the MFG with singular control is a pair $(\mu^*,(\Theta^*_i)_{i \in \N})$,
		such that $\mu^* \in \L^0_\Gc\big(\Om, \Pc(\D \x \Pi_{i \in \N}\V^+_{\thetau_i})\big)$ and $\Theta^*_i \in \Ab_i$, $i \in \N$,
		satisfy
		$$
			\mu^* =\Lc \big( X^{m^*}, (\Theta^*_i)_{i \in \N} \big| \Gc \big)
			~\mbox{and}~
			J(\Theta^*_i, \mu^*) = \inf_{\Theta \in \Ab_i} J(\Theta, \mu^*), ~\mbox{for each}~i \in \N.
		$$
        \end{definition}

        Now we recall one of the main theorems in \cite{Bank2004}, which shows how the Bank-El Karoui's representation theorem can be used to construct the unique solution to the singular control problem \eqref{eq:Bank_singular_control} with given $m$.

        \begin{assumption}\label{ass:Bank_singular_control}
            \noindent $\mathrm{(i)}$
            For any $(t, \om, \mu) \in [0,T] \x \Om \x \Pc(\D \x \Pi_{i \in \N}\V^+_{\thetau_i})$, the map $ \ell \longmapsto c(t,\om, \mu,\ell)$ is strictly convex with continuous derivative $c^\prime(t,\om,m,\ell) = \frac{\partial}{\partial \ell} c(t,\om,m,\ell)$ strictly increasing from $-\infty$ to $+\infty$, i.e. $c^\prime (t,\om,m, -\infty) = -\infty$ and $c^\prime (t,\om, m, +\infty) = +\infty$.

		\vspace{0.5em}
	
            \noindent $\mathrm{(ii)}$
            For each $(\mu, \ell) \in \L^0_\Gc(\Om,\Pc(\D \x \Pi_{i \in \N}\V^+_{\thetau_i})) \x \R$, the process $(t, \om) \longmapsto c(\om,t, \mu,\ell)$ is progressively measurable and
            $$
                \E\bigg[\int_{0}^{T}|c(t,\mu,\ell)|dt\bigg]
                 <
                +\infty,
                ~
                \E\bigg[\int_{0}^{T}\inf_{\ell \in [\thetau_i,(\thetah_i)_t]}|c(t,\mu,\ell)|dt\bigg]
                 <
                +\infty,
                ~
                i \in \N.
            $$

            \noindent $\mathrm{(iii)}$
            For each $\mu \in \L^0_\Gc\big(\Om, \Pc(\Pi_{i \in \N}\V^+_{\thetau_i})\big)$, the process $(t, \om) \longmapsto k(t, \om, \mu)$ is  an optional process of Class (D), continuous in expectation with $k(T, \mu) = 0$, and
            $$
                \sup_{\Theta \in \Ab_i}
                \E\Big[ \int_0^T \big| k(t,\mu) \big| d\Theta_t\Big]
                <
                +\infty,
                ~
                i \in \N.
            $$
        \end{assumption}

        \begin{theorem}[Bank \cite{Bank2004}]\label{thm:Bank_singular_control}
		Let Assumption \ref{ass:Bank_singular_control} hold true.
		Then each fixed $\mu \in \L^0_\Gc(\Om,\Pc(\D \x \Pi_{i \in \N}\V^+_{\thetau_i}))$, the unique minimizer for the problem \eqref{eq:Bank_singular_control} is given by
		\begin{equation} \label{eq:def_Theta_L}
		\Theta^{\mu,i, *}_t
		:=
		\sup_{s \in [0,t)} L^\mu_t \wedge \overline \Theta^i_{t} \vee \thetau_i,
                ~
                t \in [0,T],
                ~
                i \in \N,
            \end{equation}
            where $L^\mu$ is the optional process solving the representation problem
            \begin{equation} \label{eq:BEK_represent_singular_ctrl}
                -k(\tau,\mu)
                ~ = ~
                \E\bigg[\int_{\tau}^{T}
                    c^\prime\Big(t, \mu, \sup_{s \in [\tau,t)}L^\mu_s\Big) dt \Big| \Fc_\tau\bigg],
			    ~\mbox{a.s., for all}~
			    \tau \in \Tc.
            \end{equation}
        \end{theorem}

	\begin{remark}
		Theorem \ref{thm:Bank_singular_control} reveals that one can reduce the singular control problem \eqref{eq:Bank_singular_control} to the corresponding representation problem \eqref{eq:BEK_represent_singular_ctrl},
		thus one can also reduce the corresponding MFG with singular control to a mean-field version of the representation \eqref{eq:MF_representation}.
        \end{remark}

	First, based on the mean-field representation results in Theorem \ref{thm:continuity}, one can deduce the following existence result for the above MFG with singular control.

	\begin{proposition}\label{prop:singular_control_application}
		Let Assumption \ref{ass:Bank_singular_control} hold true and $\Gc$ be generated by a countable partition of $\Om$ as in \eqref{eq:Gc_countable},
		and $t \longmapsto k(t, \cdot, \mu)$ have almost surely left upper semi-continuous paths for each $\mu \in \L^0_\Gc\big(\Om, \Pc(\Pi_{i \in \N}\V^+_{\thetau_i})\big)$.
		Suppose in addition that

		\vspace{0.5em}

		\noindent $\mathrm{(i)}$
		for $\mu^n, \mu^\infty \in \L^0_{\Gc}\big(\Om;\Pc\big(\D \x \Pi_{i \in \N}\V^+_{\thetau_i}\big)\big)$
		with
		$\lim_{n \to \infty} \mu^n= \mu^\infty$ in probability,
		we have for all $\ell \in \R$,
		\begin{align*}
    			\lim_{n \to \infty}
    			\E \bigg[
                    d(X^{\mu^n}, X^{\mu^\infty})
                    ~ & + ~
                    \sup_{t \in [0,T]}
                        \big| k(t, \om, \mu^n)
                            - k(t, \om, \mu^\infty) \big|
                    \\ ~ & + ~
                        \int_0^{T}
                            \big|
                            c^\prime(t, \om, \mu^n, \ell)
                            - c^\prime(t, \om, \mu^\infty, \ell)
                            \big|
                        ds
                    \bigg]
                ~ = ~
                0.
    		\end{align*}

		\noindent $\mathrm{(ii)}$
		the collection of probability measures
		$\big\{\Lc(X^\mu|\Gc): \mu \in \L^0_{\Gc}\big(\Om;\Pc\big(\D \x \Pi_{i \in \N}\V^+_{\thetau_i}\big)\big)\big\}$ is tight in $\L^0_{\Gc}(\Om, \Pc(E))$ by identifying  $\L^0_{\Gc}(\Om, \Pc(E))$ as $(\Pc(E))^{\N}$ under the weak convergence topology.
		
		\vspace{0.5em}
		
		Then there exists a solution $(\mu^*,(\Theta^*_i)_{i \in \N})$ to the MFG with singular control,
		where $\Theta^*_i$ is defined as in \eqref{eq:def_Theta_L} for some optional process $L^{*}$ (in place of $L^{\mu}$).
	\end{proposition}

	We next provide an existence result based on Theorem \ref{thm:monotonicity} under some monotonicity conditions.
	Let $\le_{\D}$ be a partial order on $\D$ such that $(\D, \le_\D)$ is a partially ordered Polish space.
	Recall that the partial order $\le_{v}$ on $\V^+_{\thetau_i}$ is defined by $\mathbf{l}^1 \le_{v} \mathbf{l}^2$ if $\mathbf{l}^1_t \le \mathbf{l}^2_t$ for all $t \in [0,T)$, $i \in \N$.
	Then, by Example \ref{exam:p.o.PolishSpace}, $\V^+_{\thetau_i}$, for all $i \in \N$, is a partially ordered Polish space,
	and $\Pi_{i \in \N}\V^+_{\thetau_i}$ is also a partially ordered Polish space with the product order.
	Moreover, there exists a partial order $\le_p$ on $\L^0_{\Gc}\big(\Om;\Pc\big(\D \x\Pi_{i \in \N}\V^+_{\thetau_i}\big)\big)$ defined by \eqref{eq:def_lep} by letting $E = \D \x \Pi_{i \in \N}\V^+_{\thetau_i}$.

	\begin{proposition}\label{prop:singular_control_application_order}
		Suppose that the function $c$ and $k$ satisfy Assumption \ref{ass:Bank_singular_control}.
		Moreover,  we assume that

        \noindent $\mathrm{(i)}$
		there exists a bounded increasing map $\varphi: \D \longrightarrow I_c$, with $I_c$ being a closed interval of $\R$,
		such that for any $\mu^1,\mu^2 \in \L^0_{\Gc}\big(\Om;\Pc\big(\D \x \Pi_{i \in \N}\V^+_{\thetau_i}\big)\big)$,
		\begin{equation}\label{eq:assumption_singular_control_order}
			(\varphi \otimes \mathrm{Id}) \# \mu_1 = (\varphi \otimes \mathrm{Id}) \# \mu_2
    			~\Longrightarrow~
			\big( X^{\mu_1}, k(\cdot, \mu^1), c^\prime(\cdot, \mu^1, \cdot) \big) = \big(X^{\mu_2}, k(\cdot, \mu^2), c^\prime(\cdot, \mu^2, \cdot) \big),
		\end{equation}

        \noindent $\mathrm{(ii)}$
        for all $\mu^1,\mu^2 \in \L^0_{\Gc}\big(\Om;\Pc\big(\D \x \Pi_{i \in \N}\V^+_{\thetau_i}\big)\big)$,
		$(\varphi \otimes \mathrm{Id}) \# \mu_1 \le_{p} (\varphi \otimes \mathrm{Id}) \# \mu_2$ implies
		\begin{equation*}
		k(\cdot, \mu^2) - k(\cdot, \mu^1)
			~\mbox{is a supermartingale,}~
    			c^\prime(\cdot, \mu^1, \cdot) \ge c^\prime(\cdot, \mu^2, \cdot),
		\end{equation*}

        \noindent $\mathrm{(iii)}$
        for any $m \in \L^0_{\Gc}\big(\Om;\Pc\big(I_c \x \Pi_{i \in \N}\V^+_{\thetau_i}\big)\big)$, there exists a $\mu \in \L^0_{\Gc}\big(\Om;\Pc\big(\D \x \Pi_{i \in \N}\V^+_{\thetau_i}\big)\big)$ such that
        $$
            m = (\varphi \otimes \mathrm{Id}) \# \mu.
        $$

		Then there exists a solution $(\mu^*,(\Theta^*_i)_{i \in \N})$ to the MFG with singular control,
		where $\Theta^*_i$ is defined as in \eqref{eq:def_Theta_L} for some optional process $L^{*}$ (in place of $L^{\mu}$).
		
        \end{proposition}

        \begin{remark}\label{rmk:comparison_singular_control}
		The MFG with singular control has been investigated in several works during the last years.
		In Fu and Horst \cite{FuHorst2017}, the authors studied a general mean field game with both regular control and singular control,
		where the interaction takes place only in states, and provided existence of solutions as well as approximation results by regular controls.
		Fu \cite{Fu2019} extended the model in \cite{FuHorst2017} to the case with jumps, where the interaction takes place in both states and controls.
		In both \cite{FuHorst2017} and \cite{Fu2019}, the solutions of the MFG are in the relaxed form.
		In Campi, De Angelis and Ghio \cite{CampiDeAngelisGhio2020}, the authors studied a MFG with special interaction terms,
		and obtained existence of solution as well as convergence rate of the corresponding $N$-players symmetric game.
		Cao and Guo \cite{CaoGuo2020} considered a MFG with singular control with special dynamic and reward function, and obtained an explicit solution.
		Cao, Dianetti and Ferrari \cite{CaoDianettiFerrari2021} studied an ergodic MFG with singular control, where the interaction takes a special form,
		and proved the existence and uniqueness of solution in the strong sense.
		In Dianetti, Ferrari, Fischer and Nendel \cite{DianettiFerrari2022}, the authors provided a unifying framework for several different submodular games including the MFG with singular control, and obtained existence of strong mean field equilibrium by Tarski's fixed point theorem based on order structure assumptions.
		Let us also mention the work of Bertucci \cite{Bertucci2020_2} where a MFG with singular control problem has been studied by PDE approaches.

		\vspace{0.5em}

		Our framework of MFG with singular control is quite close to that in \cite{FuHorst2017} and \cite{DianettiFerrari2022}.
		Compared with \cite{FuHorst2017}, we consider the MFG with only singular control but not regular control by adapting the framework of Bank \cite{Bank2004}.
		While MFG in \cite{FuHorst2017} is a one population model and the solutions are in relaxed sense, our Proposition \ref{prop:singular_control_application} gives existence results for MFG of multiple populations and the solution is in strong sense, i.e. the solutions are in a fixed filtered probability space.
		Our Proposition \ref{prop:singular_control_application_order} can be considered as a multiple population extension of the results in \cite{DianettiFerrari2022}.
		Technically, our  Proposition \ref{prop:singular_control_application_order} is based on the representation results in Theorem \ref{thm:monotonicity}, whose proof is based on Tarski's fixed point theorem under some order structure conditions,
		which is similar to the technique used in \cite{DianettiFerrari2022}.

        \end{remark}

	\begin{remark}
		In contrast to the MFG of timing in Section \ref{subsubsec:MFG_timing} where we use an uncountable set of populations, 
		we consider a countable set of populations for the MFG with singular control.
		The main reason is that, by considering a countable set of populations in the MFG with singular control, the interaction term $\mu$ is a probability measure on $\sum_{i\in \N} \V^+_{\theta_i}$, which is a Polish space.
	
		\vspace{0.5em}
	
		For the MFG of timing with an uncountable set of population, the interaction is a probability measure on $[0,T]^{\R}$.
		Although $[0,T]^{\R}$ is not a Polish space as it is separable, we can use the order structure of the space $[0,T]$ to identify a Polish subspace of $[0,T]^{\R}$ so that the interaction is in fact a probability measure on the Polish subspace.
	\end{remark}

\subsubsection{A mean field game of optimal consumption}

	We next introduce a MFG of optimal consumption with one population, based on the infinite time horizon (i.e. $T = \infty$) model in Bank and Riedel \cite{BankRiedel01}.
	Let us denote by $\D_-$ the space of all $\R$-valued l\`adc\`ag functions on $\R_+$ equipped with the Skorokhod topology, so that it is a Polish space.
	In this MFG, the mean-field interaction term is a $\Gc$-measurable random measure $\mu \in \L^0_\Gc\big(\Om, \Pc(\D \x \D_-)\big)$.
	Given $\mu$, the market interest rate is a progressively measurable process $(r^{\mu}_s)_{s \ge 0}$,
	and a consumption process $C$ is left-continuous and right-limit (l\`adc\`ag) adapted increasing process.
	Next, with the total budget $b \in \R$, the set of all admissible consumption processes is defined by
	$$
		\Ab(b) \vcentcolon= \bigg\{
			C~\mbox{is l\`adc\`ag  adapted increasing process s.t.}~
			\E\bigg[
				\int_{0}^{+\infty} \!\!\! e^{-\int_0^t r^{\mu}_s ds} d C_t
			\bigg]
			\leq b
		\bigg\}.
        $$
	Further, given a consumption process $C$, with the initial satisfaction level $\eta > 0$ and discount constant $\beta > 0$,
	a satisfaction process $Y^C$ is defined by
	$$
		Y^C_t  \vcentcolon= \eta e^{-\beta t}
		+
		\int_{0}^{t}\beta e^{-\beta(t-s)} dC_s,
		~ t \geq 0.
        $$
	Then with the utility function $u: \R_+ \x \Om  \x \R \x \Pc(\D \x \D_-) \longrightarrow \R$,
	 a representative agent in the MFG solves the optimal consumption problem
	\begin{equation}\label{eq:Bank_optimal_consumption}
		\sup_{C \in \Ab(b)} U(C, \mu),
		~~\mbox{with}~~
		U(C, \mu)
		\vcentcolon=
		\E\bigg[
			\int_{0}^{+\infty}u(t, Y^C_t, \mu) dt
		\bigg].
	\end{equation}

	Let us refer to Bank and Riedel \cite{BankRiedel01} for an economic interpretation of the above optimal consumption model.
	Moreover, in  \cite{BankFollmer}, a characterization of the optimal consumption process $C^*$ is derived by using  Bank-El Karoui's representation \eqref{eq:BankEK_repres_intro}.

	\begin{assumption}\label{ass:Bank_optimal_consumption}
		\noindent $\mathrm{(i)}$
		For every $(t, \om, \mu) \in [0,T] \x \Om \x \L^0_{\Gc}(\Pc(\D \x \D_-))$, the map $ \ell \longrightarrow u(t, \om, \mu, \ell)$ is strictly concave, increasing in the variable $\ell \in [0,+\infty)$ with continuous derivative
		$$
			u^\prime(t, \om,\mu,\ell) ~:=~ \frac{\partial}{\partial \ell} u(t, \om, \mu,\ell)
		$$
		decreasing from $+\infty$ to $0$, i.e. $u^\prime (t, \om, \mu, 0) = +\infty$ and $u^\prime (t, \om, \mu, +\infty) = 0$.

		\vspace{0.5em}

		\noindent $\mathrm{(ii)}$
		For all $b \in \R_+$, and $(\mu, \ell) \in \L^0_\Gc\big(\Om, \Pc(\D \x \D_-)\big) \x \R$,
		the processes $(t, \om) \mapsto u(t, \om, \mu,\ell)$ and $e^{-\int_0^\cdot r_s(\mu) ds}$ are progressively measurable, and
		$$
			\E\bigg[\int_{0}^{+\infty}|u(t,\mu,\ell)|dt\bigg]
			~ < ~
			+\infty,
			~~~
			\sup_{C \in \Ab(b)}U(C, \mu) < +\infty.
		$$
	\end{assumption}

	\begin{theorem}[Bank and Follmer \cite{BankFollmer}]\label{thm:Bank_optimal_consumption}
		Let Assumption \ref{ass:Bank_optimal_consumption} hold true and $\mu \in \L^0_\Gc(\Om;\Pc(\D \x \D_-))$ be fixed.
		Then for any Lagrange multiplier $\lambda > 0$, the discounted price deflator process $\lambda e^{-\beta \cdot-\int_0^\cdot r^{\mu}_s ds}$ admits the representation,
		\begin{equation}\label{eq:Optimal_consumption_representation}
			\lambda e^{-\beta \tau-\int_0^\tau r^{\mu}_s ds}
			~ = ~
			\E\bigg[\int_{\tau}^{+\infty}
			\beta e^{-\beta t}u^\prime\Big(t, {\mu}, \frac{- e^{-\beta t}}{\sup_{s \in [\tau,t)}L^{{\mu},\lambda}_s}\Big) dt \Big| \Fc_\tau\bigg]
			    ~\mbox{a.s. for all}~
			\tau \in \Tc,
		\end{equation}
		for some $\F$-progressively measurable process $L^{{\mu},\lambda}$ with upper-right continuous paths
		and the consumption plan $C^{{\mu},\lambda} $ for the problem \eqref{eq:Bank_optimal_consumption} such that
		\begin{equation*}
			Y^{C^{{\mu},\lambda}}_t
			~ = ~
			e^{-\beta t} \Big( \eta \vee \frac{- 1}{\sup_{s \in [0,t)}L^{{\mu},\lambda}_s} \Big),
			~
			t \in \R^+,
		\end{equation*}
		is optimal for its cost
		$$
			b^{{\mu},\lambda} =
			\E\bigg[\int_{0}^{+\infty}e^{-\int_0^t r^{{\mu}}_s ds} d C^{{\mu},\lambda}_t\bigg]
			= \E\bigg[\int_{0}^{+\infty}\frac{e^{-\int_0^t r^{{\mu}}_s ds - \beta t}}{ \beta} d\Big( \eta \vee \frac{- 1}{\sup_{s \in [0,t)}L^{{\mu},\lambda}_s}\Big)\bigg].
		$$
	\end{theorem}

	\begin{remark}\label{rmk:optimal_conusmption}
		$\mathrm{(i)}$
		The problem \eqref{eq:Optimal_consumption_representation} reduces to the representation \eqref{eq:BankEK_repres_intro} with
		\begin{equation*}
			Y^{\mu,\lambda}(t,\om)
			~ \vcentcolon = ~
			-\lambda e^{-\beta t -\int_0^t r^{\mu}_s ds},
		\end{equation*}
		and
		\begin{equation*}
			f^{\mu}(t, \om, \ell)
			~ \vcentcolon = ~
			\left\{
			\begin{aligned}
				& -u^\prime(t, \om, \mu, \frac{-e^{-\beta t}}{\ell})
				~ & ~\mbox{if}~ \ell < 0,
				\\
				& \ell & ~\mbox{if}~ \ell \ge 0.
			\end{aligned}
        		        \right.
		\end{equation*}
		Under Assumption \ref{ass:Bank_optimal_consumption}, the classical Bank-El Karoui's representation theorem  \eqref{eq:BankEK_repres_intro}  ensures the existence of a representation process $L^{m,\lambda}$.
		Then it remains to verify that the stopping time
		$\tau_0 \vcentcolon = \inf\{t \geq 0: L^{\mu,\lambda}_t \geq 0\}$ is equal to $+\infty$ a.s.
		The claim holds true if we observe that
		\begin{align*}
			0
			~ & \geq ~
			- \lambda e^{-\beta \tau_0-\int_0^{\tau_0} r^\mu_s ds}
			\\ ~ & = ~
			\E\bigg[\int_{\tau_0}^{T}
				f^\mu\Big(\sup_{s \in [\tau_0,t)}L^{\mu,\lambda}_s\Big) dt \Big| \Fc_{\tau_0}\bigg]
				~  = ~
			\E\bigg[\int_{\tau_0}^{T}
			\sup_{s \in [\tau_0,t)}L^{\mu,\lambda}_s dt \Big| \Fc_{\tau_0}\bigg]
			\geq 0.
		\end{align*}

		\noindent $\mathrm{(ii)}$
		Notice that for different Lagrange multiplier $\lambda \in \R_+$, the induced solution $C^{\mu, \lambda}$ solves an optimal consumption problem with budget $b^{\mu, \lambda}$, which is defined posterior.
		For the constructive results with an a priori given budget, we may refer to Bank and Kauppila \cite{BankKauppila2017} for details.
		We do not adopt the framework in \cite{BankKauppila2017} to formulate our MFG with optimal consumption,
		because it requires good understanding on how the Lagrange multiplier $\lambda_{\mu}$ and the process $\lambda_{\mu} e^{-\beta t -\int_0^t r^{\mu}_s ds}$ change
		when $\mu$ changes.
        \end{remark}

	Let us now define the solution to our MFG of optimal consumption.

	\begin{definition}\label{def:MFE_optimal_consumption}
		
		Let $b \in \R_+$,
		a mean field equilibrium to the MFG of the optimal consumption problem with total budget $b$ is a pair $(\mu^*,C^*)$,
		where $\mu^* \in \L^0_\Gc\big(\Om, \Pc(\D \x \D_-)\big)$, and $C^* \in \Ab(b)$ such that
		$\mu^* := \Lc\big( (r^{\mu^*}, Y^{C^{*}})  \big| \Gc \big)$ and
		$C^*$ is the optimal consumption process in the sense that $U(C^*, \mu^*) =\sup_{C \in \Ab(b)} U(C, \mu^*)$.
        \end{definition}

	First, based on the mean-field representation results in Theorem \ref{thm:continuity}, one can deduce the following existence result for the above MFG of optimal consumption.
	Recall that $\eta > 0$ is the given initial satisfaction level.

	\begin{proposition}\label{prop:optimal_consumption_application}
		Let Assumption \ref{ass:Bank_optimal_consumption} hold true and $\Gc$ be generated by a countable partition of $\Om$ as in \eqref{eq:Gc_countable}. Assume in addition that

		\noindent $\mathrm{(i)}$
		there exists a constant $\etah > \eta$ such that
		\begin{equation}\label{eq:consumption_continuity_restriction}
			u^\prime(\cdot, \mu^1, \ell) = u^\prime(\cdot, \mu^2, \ell),
			~
			r^{\mu^1}(\cdot) = r^{\mu^2}(\cdot),
			~\mbox{for all}~
			\ell \in \R,
		\end{equation}
		whenever
		$$
			\mu^1(A \x B) = \mu^2(A \x B),~\mbox{a.s.}~
			~\mbox{for all}~
			A \in \Bc(\D),
			B \in \Bc(\{D \in \D_- : \eta e^{-\beta \cdot} \le D_\cdot \le \etah e^{-\beta \cdot}\}).
		$$

            \noindent $\mathrm{(ii)}$
            for $\mu^n, \mu^\infty \in \L^0_{\Gc}(\Om;\Pc(\D \x \D_-))$
            with
            $\lim_{n \to \infty} \mu^n = \mu^\infty$ in probability,
            we have for all $\ell \in \R_+\setminus \{0\}$,
            \begin{align*}
    			\lim_{n \to \infty}
    			\E \bigg[
                    d_\D(r^{\mu^n}, r^{\mu^\infty})
                    ~ & + ~
                    \sup_{t \in \R_+}
                        \big| e^{-\int_0^t r^{\mu^n}_s ds}
                            - e^{-\int_0^t r^{\mu^\infty}_s ds} \big|
                    \\ ~ & + ~
                        \int_0^{\infty}
                            \big|
                            u^\prime(t, \om, \mu^n, \ell)
                            - u^\prime(t, \om, \mu^\infty, \ell)
                            \big|
                        ds
                    \bigg]
                ~ = ~
                0.
    		\end{align*}

	\noindent $\mathrm{(iii)}$
	the collection of probability measures $\big\{\Lc(r^\mu|\Gc) : \mu \in \L^0_{\Gc}(\Om;\Pc(\D \x \D_-))\}$ is tight.

	\vspace{0.5em}

	Then there exists a solution $(\mu^*, C^*)$ to the MFG of optimal consumption with budget
	\begin{equation} \label{eq:b_expression}
		b^* ~:=~\E\bigg[\int_{0}^{+\infty}e^{-\int_0^t r^{{\mu^*}}_s ds} d C^*_t\bigg].
	\end{equation}
	\end{proposition}

	We next provide an existence result based on Theorem \ref{thm:monotonicity} under some monotonicity conditions.
	Let $\le_{\D}$ be a partial order on $\D$ such that $(\D, \le_\D)$ is a partially ordered Polish space,
	and define a partial order $\le_{\D_-}$ on $\D_-$ by $\mathbf{l}^1 \le_{\D_-} \le \mathbf{l}^2$ if $\mathbf{l}^1_t \le \mathbf{l}^2_t$ for all $t \in [0,+\infty)$.
        Then both $\D$ and $\D_-$ are partially ordered Polish spaces,
        and there exists a partial order $\le_p$ on $\L^0_{\Gc}(\Om;\Pc(\D \x \D_-))$ defined as \eqref{eq:def_lep}, if we let $E = \D \x \D_-$.

	\begin{proposition}\label{prop:optimal_consumption_application_order}
		Suppose that the functions $r$ and $u$ satisfy Assumption \ref{ass:Bank_optimal_consumption}.
		Assume in addition that

        \noindent $\mathrm{(i)}$
		there exists a bounded increasing map $\varphi: \D \longrightarrow \R$,
		such that for any $\mu^1, \mu^2 \in \L^0_{\Gc}(\Om;\Pc(\D \x \D_-))$,
		\begin{equation}\label{eq:consumption_order_equivalence}
			(\varphi \otimes \mathrm{Id}) \# \mu_1 = (\varphi \otimes \mathrm{Id}) \# \mu_2
			~\Longrightarrow~
			\big( r^{\mu_1}, u^\prime(\cdot, \mu^1, \cdot) \big)
			=
			\big(r^{\mu_2}, u^\prime(\cdot, \mu^2, \cdot) \big).
		\end{equation}

        \noindent $\mathrm{(ii)}$
        for any $\mu^1, \mu^2 \in \L^0_{\Gc}(\Om;\Pc(\D \x \D_-))$,
    		$$
                (\varphi \otimes \mathrm{Id}) \# \mu_1 \le_{p} (\varphi \otimes \mathrm{Id}) \# \mu_2
                \Longrightarrow
    			u^\prime(\cdot, \mu^1, \ell) \le u^\prime(\cdot, \mu^2, \ell),
                ~\mbox{for all}~
                \ell \in \R,
    		$$

        \noindent $\mathrm{(iii)}$
        for any $m \in \L^0_{\Gc}(\Om;\Pc(I_c \x \D_-))$, there exists a $\mu \in \L^0_{\Gc}(\Om;\Pc(\D \x \D_-))$ such that
        $$
            m = (\varphi \otimes \mathrm{Id}) \# \mu.
        $$
		
		\noindent $\mathrm{(iv)}$
        there exists some constant $\etah > \eta$ such that
		\begin{equation}\label{eq:consumption_order_restriction}
			u^\prime(\cdot, \mu^1, \ell) = u^\prime(\cdot, \mu^2, \ell),
			~
			r^{\mu^1}(\cdot) = r^{\mu^2}(\cdot),
			~\mbox{for all}~
			\ell \in \R,
		\end{equation}
		whenever
		$$
			\mu^1(A \x B) = \mu^2(A \x B),~\mbox{a.s.}~
			~\mbox{for all}~
			A \in \Bc(\D),
			B \in \Bc(\{D \in \D_- : \eta e^{-\beta \cdot} \le D_\cdot \le \etah e^{-\beta \cdot}\}).
		$$

		Then there exists a solution $(\mu^*, C^*)$ to the MFG of optimal consumption with budget $b^*$ given by \eqref{eq:b_expression}.
	\end{proposition}

        Finally, we provide an existence result based on the mean-field representation results in Theorem \ref{thm:contraction}.
        In this setting, $r^\mu$ is assumed to be independent of $\mu$ and $\Gc = \{\emptyset,\Om\}$, then all the arguments concerned with $\D \x \D_-$ reduce to those with $\D_-$.
        Let us take $\D_-$ as canonical space with canonical process $L$, for $\mu \in \Pc(\D_-)$, define
    	$$
    		\mu_t(dx)  := \mu \circ L_t^{-1} \in \Pc(\R), ~\mbox{for all}~t \in [0,T).
    	$$

        \begin{proposition}\label{prop:optimal_consumption_contraction}
            Let Assumption \ref{ass:Bank_optimal_consumption} hold true.
            Assume in addition that

            \noindent $\mathrm{(i)}$
            for some function $\ut^\prime: \R^+ \x \Om \x \R \longrightarrow \R$, one has
            $$
                u^\prime(t, \om, \mu, -\frac{e^{-\beta t}}{\ell})
                ~ = ~
                \ut^\prime\bigg(t, \om, -\frac{e^{-\beta t}}{\ell - \mu_t(\phi(-\frac{e^{-\beta t}}{\cdot}))}\bigg),
                ~\mbox{for all}~
                \mu \in \L^0_\Gc(\Om, \Pc(\D_-)),
            $$
            where the function $\phi: (-\infty, 0) \longrightarrow \R$ is bounded and differentiable, satisfying $\phi^\prime(x) \in [0,1)$ for all $x \in  (-\infty, 0)$,

            \noindent $\mathrm{(ii)}$
            for some $\lambda > 0$, the optional process $L$ to the representation \ref{thm:Bank_optimal_consumption} of $\lambda e^{-\beta \cdot-\int_0^\cdot r_s ds}$ w.r.t. $\ut^\prime$ has almost surely non-decreasing paths.

		\vspace{0.5em}

		Then there exists a solution $(\mu^*, C^*)$ to the MFG of optimal consumption with budget $b^*$ given by \eqref{eq:b_expression}.
        \end{proposition}

        \begin{remark}
		In the setting of Proposition \ref{prop:optimal_consumption_contraction},
		we consider a special structure of the dependence of the utility function $u$ in $\mu$
		by letting the utility of the consumer depend not only its own consumption, but also the average consumption of the whole population.
		Namely, its satisfaction increases in its own consumption, but decreases w.r.t. the average consumption of the whole population.

		\vspace{0.5em}

		More concretely, the interaction term is $\mu_t(\phi(-\frac{e^{-\beta t}}{\cdot}))$.
		Notice that, by Theorem \ref{thm:Bank_optimal_consumption}, the optimal consumption plan $C$ of an agent
		is that leads to the satisfaction level
		$$
			Y^{C}_t
    			~ = ~
    			e^{-\beta t} \Big( \eta \vee \frac{- 1}{\sup_{s \in [0,t)}L_s} \Big),
    			~
			t \in \R^+,
		$$
		for a strictly negative process $L$.
		The process
		$$
			\Big(
				\eta \vee \frac{- 1}{\sup_{s \in [0,t)}L_s} \
			\Big)_{t \ge 0}
		$$
		can be interpreted as a consumption intensity process,
		and the interaction term is given by
		$$
			\mu_t \Big(\phi \Big(-\frac{e^{-\beta t}}{\cdot} \Big) \Big)
			~=~
			\E \Big[ \phi \Big (-\frac{1}{\eta} \vee \sup_{s \in [0,t)}L_s \Big) \Big].
		$$
	\end{remark}

\section{A stability analysis on Bank-El Karoui's representation theorem}
\label{sec:stability}

	In this section, we provide a stability result on Bank-El Karoui's representation theorem.
	This plays a crucial technical step in the proof of Theorem \ref{thm:continuity} for our mean-field representation \eqref{eq:MF_representation}.
	At the same time, it would have its own interests and other applications.
	
	\vspace{0.5em}
	
	Recall that $T \in [0,\infty]$, $(\Om, \Fc, \P)$ is a complete probability space, equipped with a filtration satisfying the usual conditions.
	Let all $(f_n)_{ n \ge 0}$ and $f$ be a sequence of functions defined on $[0,T] \x \Om \x \R$, satisfy all the technical conditions in Bank-El Karoui's representation theorem (see Assumption \ref{Assumption:BErepresentation} and Theorem \ref{thm:BErepresentationTheorem}),
	and assumet that $(Y^n)_{n \ge 0}$ and $Y$ are optional processes of class (D) and are u.s.c in expectation (Definition \ref{def:USCE}) such that
	$Y^n_T = Y_T = 0$.
	Then by Bank-El Karoui's representation theorem (Theorem \ref{thm:BErepresentationTheorem}), for a couple $(f_n, Y^n)$ (resp. $(f, Y)$),
	there exists a representation optional process $L^n$ (resp. $L$) such that, for all $\tau \in \Tc$,
	$$
		Y^n_\tau = \E\bigg[\int_{\tau}^T f_n\Big(t, \sup_{\tau \le s < t}L^n_s \Big) dt \Big| \Fc_\tau\bigg]
		~~\bigg( \mbox{resp.}
		Y_\tau = \E\bigg[\int_{\tau}^T f\Big(t, \sup_{\tau \le s < t}L_s \Big) dt \Big| \Fc_\tau\bigg]
		\bigg),
		~\mbox{a.s.}
	$$
	While the representation processes $L^n$ and $L$ may not be unique,
	the corresponding running maximum $\Lh^n$ and $\Lh$ defined below is unique for the given $(f_n, Y^n)$ and $(f,Y)$,
	$$
		\Lh^n_t ~:=~ \sup_{0 \le s < t} L^n_s,
		~~~
		\Lh_t ~:=~ \sup_{0 \le s < t} L_s,
		~~
		t \in [0,T].
	$$

	\begin{theorem} \label{Thm:stability}
		Assume that $(f_n, Y^n)_{n \ge 0}$ converges to $(f, Y)$ in the sense that, for all $\ell \in \R$,
		\begin{equation} \label{eq:fnYn_fY}
			\lim_{n \to \infty}
			\E\bigg[
			\int_{0}^{T}
				|f_n(t,\ell) - f(t,\ell)|
			dt
			+
			\sup_{t \in [0,T]}
				|Y^n_t - Y_t |
			\bigg]
			~ = ~
			0,
		\end{equation}
		and, in addition, $Y$ has almost surely left upper semi-continuous paths.

		\vspace{0.5em}

		Then, with the L\'evy metric $d_L$ on $\V^+$,
		one has for every $\eps > 0$,
		$$
			\lim_{n \to \infty}
			\P \big[ d_L\big(\Lh^n, \Lh\big) \geq \eps   \big]
			= 0.
		$$
		Consequently, with the Lévy–Prokhorov metric $d_{LP}$ on the space $\Pc(\V^+)$
		one has
		$$
			\lim_{n \to \infty}
			\P \big[ d_{LP}\big(\Lc(\Lh^n|\Gc), \Lc(\Lh|\Gc)\big) \geq \eps \big]
			= 0,
			~~~
			\lim_{n \to \infty}
			\E \big[d_{LP}\big(\Lc(\Lh^n|\Gc), \Lc(\Lh|\Gc)\big) \big]
			= 0.
 		$$
	\end{theorem}

	\begin{remark}\label{rmk:stability}
		$\mathrm{(i)}$ In general, one can only expect a stability result on the running maximum process $\Lh$ but not the process $L$ itself.
		Let us illustrate it by  the following counter-example.
		We consider a deterministic setting with $\Fc = \{\Om, \emptyset\}$ and $T = 1$, so that $(Y^n, f_n)_{n \ge 0}$ and $(Y, f)$ are all deterministic.
		Let
		$$
			f_n(t) = f(t) = t,
			~~
			Y^n_t = (t - \frac12) \mathds{1}_{[\frac12, \frac12+\frac1n)}(t),
			~~
			Y_t = 0,
			~~t \in [0,1].
		$$
		We observe that $(f_n, Y^n) \longrightarrow (f,Y)$ uniformly.
		In this deterministic setting, $L^n_t$ can be interpreted as the derivative of the convex envelope of $t \longmapsto - Y^n_t$ on $[t, 1]$ (see \cite[Theorem 2]{BankKaroui2004}).
		By direct computation,
		$$
			L^n_t = - \frac{2}{2 + n(1-2t)} \mathds{1}_{[0, \frac12 + \frac1n)}(t),
			~~
			L_t = 0,
			~~
			t \in [0,1).
		$$
		For $t = \frac12$, we observe that $L^n_{1/2}$ does not converge to $L_{1/2}$.
		But for the running maximum process, we have $\Lh^n_t \longrightarrow \Lh_t$ for all $t \in [0,1]$.

		\vspace{0.5em}

		\noindent $\mathrm{(ii)}$ We next provide a counter-example,
		showing that the convergence $Y^n_t \longrightarrow Y_t$, a.s. for every $t \in [0,T]$ is not enough to ensure the convergence of $\Lh^n \longrightarrow \Lh$.
		We still stay in the deterministic setting, with $T =1$, and
		$$
			f_n(t) = f(t) \equiv t,
			~~
			Y^n_t = n (t - \frac12) \mathds{1}_{[\frac12, \frac12+\frac1n)}(t),
			~~
			Y_t = 0,
			~~t \in [0,1].
		$$
		We observe that $Y^n_t \longrightarrow Y_t$, for all $t \in [0,1]$.
		Similarly, we compute that
		$$
			L^n_t = - \frac{2n}{2 + n(1-2t)} \mathds{1}_{[0, \frac12 + \frac1n)}(t),
			~~
			L_t = 0,
			~~
			t \in [0,1).
		$$
		In particular,
		$$
			\Lh^n_{1/2} = - \frac{2n}{n+2} ~\longrightarrow~ -2 \neq 0 = \Lh_{1/2}.
		$$

	\end{remark}

	\begin{remark}
		In the applications of Bank-El Karoui's representation theorem in the optimal stopping problem, the singular control problem with monotone follower type,
		and the optimal consumption problem,
		one can characterize the optimizers with the corresponding representation process $\Lh$.
		Consequently, the stability result $\Lh^n \longrightarrow \Lh$ should imply the corresponding stability results of the optimizers in concrete applications.

		\vspace{0.5em}
		
		More precisely, Proposition \ref{prop:cvg_st} below provides a stability results for a family of optimal stopping problems.
		Similarly, for a class of singular control problems as in Bank \cite{Bank2004} (see Theorem \ref{thm:Bank_singular_control}), the optimal singular control is directly given by $\Lh$ with truncation.
		The stability result in Theorem \ref{Thm:stability} induces naturally the stability of the corresponding optimal singular controls.
	\end{remark}

	\begin{remark}
		A potential application of the stability result and the corresponding techniques
		is to make some numerical analysis on the approximation methods to compute the process $L$ in the Bank-El Karoui's representation theorem.
		Formally, one can take a time discretization $0 = t_0 < t_1 < \cdots < t_N = T$ of $[0,T]$,
		and then compute $L^N_{t_{N-1}}$ by solving the equation
		$$
			Y_{t_{N-1}} = \E \Big[ f \big( t_{N-1}, L^N_{t_{N-1}} \big) \Delta t \Big| \Fc_{t_{N-1}} \Big],
		$$
		and then, in a backward way, compute $L^N_{t_i}$ by solving the equation
		$$
			Y_{t_i} = \E \Big[f(t_i, L^N_{t_i}) \Delta t +  \sum_{j=i+1}^{N-1} f \Big(t_j, L^N_{t_i} \vee \cdots \vee L^N_{t_j} \Big) \Delta t \Big| \Fc_{t_i} \Big].
		$$
		Under some regularity condition of $Y$ and $f$ in the temporal variable, we can naturally consider it as a perturbation of the continuous time problem,
		and obtain the convergence of the running maximum $\Lh^N$ of $L^N$ to the running maximum $\Lh$ of $L$, as a stability result.
		We will explore this more systematically in our future work.
	\end{remark}

	In preparation of the Proof of Theorem \ref{Thm:stability}, let us define, for each $\ell \in \R$, $n \ge 0$,
	\begin{equation} \label{eq:def_tau_l}
		\tau_{\ell}
		\vcentcolon=
		\inf \big\{ t\geq 0 ~: \Lh_{t} \geq \ell \big\},
		~~~
		\tau'_{\ell}
		\vcentcolon=
		\inf \big\{ t\geq 0 ~: \Lh_{t-} > \ell \big\},
		~~\mbox{and}~~
		\tau^n_\ell
		~\vcentcolon=~
		\inf\{t \ge 0: \Lh^n_t \ge 0\}.
	\end{equation}
	By Bank and El Karoui \cite{BankKaroui2004} (see Theorem \ref{thm:BErepresentationTheorem}), $\tau_{\ell}$ and $\tau'_{\ell}$ are respectively the smallest and the biggest solution to the optimal stopping problem:
	\begin{equation}\label{eq:stability_stopping}
		\sup_{\tau \in \Tc} J_\ell(\tau),
		~\mbox{with}~
		J_\ell(\tau)
		\vcentcolon =
		\E\Big[Y_\tau + \int_{0}^{\tau}f(t,\ell) dt\Big].
	\end{equation}
	Similarly, $\tau^n_{\ell}$ is the smallest solution to the optimal stopping problem
	\begin{equation}\label{eq:stability_stopping_n}
		\sup_{\tau \in \Tc} J^n_\ell(\tau),
		~\mbox{with}~
		J^n_\ell(\tau)
		\vcentcolon =
		\E\Big[Y^n_\tau + \int_{0}^{\tau}f_n(t,\ell) dt\Big].
	\end{equation}
	Further, based on the (abstract) probability space $(\Om, \Fc, \P)$, we introduce an enlarged measurable space $(\Omo, \Fcb)$ together with a canonical element $\Theta: \Omo \longrightarrow [0, T]$, by
	$$
		\Omo := \Om \x [0,T],
		~~
		\Fcb := \Fc \ox \Bc([0,T]),
		~~\mbox{and}~~
		\Theta(\omb) := \theta, ~\mbox{for all}~\omb = (\om, \theta) \in \Omo.
	$$
	We also introduce a filtration $\Fbb = (\Fcb_t)_{t \in [0,T]}$, together with a sub-filtration $\Fbb^0 = (\Fcb^0_t)_{t \in [0,T]}$ and a sub-$\sigma$-field $\Fcb^0$, by
	$$
		\Fcb_t := \Fc_t \ox \sigma\{[0,s], ~s \le t\},
		~~
		\Fcb^0_t := \Fc_t \ox \{ \emptyset, [0,T]\},
		~~
		t \in [0,T],
		~~\mbox{and}~~
		\Fcb^0:= \Fc \ox \{\emptyset, [0,T]\}.
	$$
	Notice that $\Theta$ is clearly a $\Fbb$-stopping time.
	Let $\Pc(\Omb)$ denote the space of all probability measures on $(\Omb, \Fcb)$, we equip $\Pc(\Omb)$ with the stable convergence topology of Jacod and M\'emin \cite{JacodMemin1981},
	that is, the coarsest topology making $\Pb \longmapsto \E^{\Pb}[\xi]$ continuous for all bounded random variable $\xi: \Omb \longrightarrow \R$ such that $\theta \mapsto \xi(\om, \theta)$ is continuous for all $\om \in \Om$ (see Section \ref{sec:stable_cvg} in Appendix for more details).
	Then for each $\ell \in \R$, $n \ge 0$, let us define $\Pb_{\ell, n} \in \Pc(\Omb)$ by
	\begin{equation} \label{eq:def_Pbn}
		\Po_{\ell, n}(A)
		~ \vcentcolon=
		\int_{\Om}
		\mathds{1}_{A} \big(\om, \tau^n_\ell(\om) \big)
		\P(d\om),
		~\mbox{for all}~
		A \in \Fc \ox \Bc([0,T]).
        \end{equation}

	\begin{proposition} \label{prop:cvg_st}
		Let the conditions in Theorem \ref{Thm:stability} hold true. Then there exists a countable set $\L_0 \subset \R$ such that, for all $\ell \in \R \setminus \L_0$,  the following holds true:

		\vspace{0.5em}
		
		\noindent $\mathrm{(i)}$ $\tau_{\ell} = \tau'_{\ell}$, a.s., so that it is the unique solution to the optimal stopping problem \eqref{eq:stability_stopping};

		\vspace{0.5em}

		\noindent $\mathrm{(ii)}$ for all $N > 0$ and $\eps > 0$, one has $\lim_{n \longrightarrow \infty} \P\big[ \big| \tau^n_{\ell} \wedge N - \tau_{\ell} \wedge N \big| > \eps \big] = 0$.
	
	\end{proposition}
	\begin{proof}
	$\mathrm{(i)}$ We first observe that, for a.e. $\om \in \Om$, $\ell \longmapsto \tau_\ell(\om)$ is left-continuous and increasing.
	By Jacod and Shiryaev \cite[Lemma IV.3.12]{JacodShiryaev2013}, there exists a countable subset $\L_0 \subset \R$ such that
	$$
		\P \big[ \tau_{\ell} = \tau_{\ell+} \big] = 1, ~\mbox{for all}~ \ell \in \R \setminus \L_0.
	$$

	Next, for any $h > 0$, we observe that $ \tau_\ell \le \tau'_\ell \le \tau_{\ell + h}$, a.s.
	Letting $h \longrightarrow 0$, it follows that $\tau_\ell \le \tau'_\ell \le \tau_{\ell + }$, a.s.
	Therefore, one has
	$$
		\tau_\ell ~=~ \tau'_\ell ~=~ \tau_{\ell + }, ~~\mbox{a.s., for all}~ \ell \in \R \setminus \L_0,
        $$
        so that $\tau_\ell$ is the unique solution to the optimal stopping problem \eqref{eq:stability_stopping}.

	\vspace{0.5em}

	\noindent $\mathrm{(ii)}$ The assumption that $Y$ has almost surely left upper semi-continuous path, and is u.s.c. in expectation implies that (see e.g. Bismut and Skalli \cite[Theorem II.1]{BismutSkalli1977})
	$Y$ has upper semi-continuous paths, i.e., for a.e. $\om \in \Om$, $t \longmapsto Y_t(\om)$ is u.s.c.
	
	\vspace{0.5em}

	\noindent $\mathrm{(ii.a)}$ Let us first consider the case where $T < \infty$.

	Recall that $\Fcb^0 := \Fc \ox \{\emptyset, [0,T]\}$, by the definition of $\Pb_{\ell, n}$ in \eqref{eq:def_Pbn}, it is clear that, for all $t \in [0,T]$,
	$$
		\Pb_{\ell, n} \big[ \Theta \le t \big| \Fcb^0 \big] (\omb)
		~=~
		\Pb_{\ell, n} \big[ \Theta \le t \big| \Fcb^0_t \big] (\omb)
		~=~
		\mathds{1}_{\{ \tau^n_{\ell}(\om) \le t \}},
		~~\mbox{for}~\Pb_{\ell,n} \mbox{-a.e.}~\omb \in \Omb.
	$$
	Next, we observe that  $T < \infty$ and $\Pb_{\ell, n}|_{\Om} = \P$ for all $n \ge 0$,
	then $(\Pb_{\ell, n})_{n \ge 0}$ is relatively compact in $\Pc(\Omb)$ under the stable convergence topology (see e.g. Theorem \ref{thm:stable_convergence}).
	Then there exists a subsequence $\{n_k\}_{k \ge 0}$ of $\{n\}_{n \ge 0}$, such that $\Po_{\ell, n_k} \longrightarrow \Po_{\ell}$ for some $\Pb_{\ell} \in \Pc(\Omo)$ under the stable convergence topology as $k \longrightarrow +\infty$.
	In particular, by Proposition \ref{prop:H_hypothesis}, one has
	$$
		\Pb_{\ell} \big[ \Theta \le t \big| \Fcb^0 \big] (\omb)
		~=~
		\Pb_{\ell} \big[ \Theta \le t \big| \Fcb^0_t \big] (\omb),
		~~
		\mbox{for}~~
		\Pb_{\ell} \mbox{-a.e.}~\omb \in \Omb.
	$$
	Since $\Fcb^0 = \Fc \otimes \{\emptyset, [0,T] \}$, one can consider $\Pb_{\ell} \big[ \Theta \le t \big| \Fcb^0 \big] $ as a function of $\om \in \Om$.
	Let us then define
	$$
		\overline F_{\om}(s) := \E^{\Pb_{\ell}} \big[\Theta \le s \big| \Fcb^0 \big] (\om),
		~~\mbox{for all}~
		s \in [0,T] \cap \Q,
	$$
	and
	$$
		\overline F_{\om}(t) := \limsup_{\Q \cap[0,T]  \ni s \searrow t} \overline F_{\om}(s),
		~~\mbox{for all}~
		t \in [0,T].
	$$
	Since $t \mapsto \mathds{1}_{\{\Theta \le t \}}$ is right-continuous, by dominated convergence theorem, it follows that,
	\begin{itemize}
		\item for all $t \in [0,T]$, $F_{\om}(t) = \Pb_{\ell} \big[ \Theta \le t \big| \Fcb^0_t \big] (\omb)$, for $\Pb_{\ell}$-a.e. $\omb = (\om, \theta) \in \Omb$;
		\item for $\P$-a.e. $\om \in \Om$, the map $t \mapsto F_{\om}(t)$ is right-continuous, increasing and takes value in $[0,1]$.
	\end{itemize}
	For each $\om \in \Om$, let $\overline{F}^{-1}_\om:[0,1] \longrightarrow [0,T]$ denote the right-continuous inverse function of $t \longmapsto \overline{F}_\om(t)$.
	It follows that, for all $u \in [0,1]$ and $t \in [0,T]$,
	$$
		\{\om \in \Om: \overline{F}^{-1}_\om(u) \le t \}
		~=~
		\{\om \in \Om: u \le \overline{F}_\om(t)\}
		~=~
		\{\om \in \Om: u \le \Po[\tau \le t|\Fcu_t](\om)\}
		~\in~
		\Fc_t.
	$$
	In particular, for all $u \in [0,1]$, $\om \longmapsto F^{-1}_{\om}(u)$ is a stopping time w.r.t. $\F$.

	\vspace{0.5em}
	
	We next claim that, for each $\ell \in \R$,
	\begin{align} \label{eq:claim_stable_cvg}
		\sup_{\tau \in \Tc}J_\ell(\tau)
		&=
		\lim_{k \to \infty}
		\sup_{\tau \in \Tc}J^{n_k}_\ell(\tau)
		=
		\lim_{k \to \infty}
		\E^{\Pb_{\ell,n_k}} \bigg[Y^{n_k}_{\Theta} + \int_{0}^{\Theta} f_{n_k}(t, \ell)dt \bigg] \nonumber \\
		&=
		\E^{\Pb_{\ell}} \bigg[Y_{\Theta} + \int_{0}^{\Theta} f(t, \ell)dt \bigg]
		=
		\int_0^1 \E \bigg[Y_{F^{-1}_{\cdot}(u)} + \int_{0}^{F^{-1}_{\cdot}(u)} f(t, \ell)dt \bigg] du.
	\end{align}
	Notice that $F^{-1}_{\cdot}(u)$ is a $\F$-stopping time, then \eqref{eq:claim_stable_cvg} implies that $F^{-1}_{\cdot}(u)$ is an optimal stopping time to the optimal stopping problem \eqref{eq:stability_stopping} for a.e. $u \in [0,1]$.
	When $\ell \in \R\setminus \L_0$, $\tau_{\ell}$ is the unique solution to the optimal stopping problem \eqref{eq:stability_stopping},
	so that
	$$
		F^{-1}(u) = \tau_{\ell}, ~\mbox{a.s. for a.e.}~
		u \in [0,1].
	$$
	This implies that, for any bounded and $\Fc \ox \Bc([0,T])$-measurable function $\xi: \Omo \longrightarrow \R$, one has
	\begin{equation*}
		\E^{\Pb_{\ell}} [\xi]
		~=~
		\E^{\Pb_{\ell}} \Big[ \E^{\Pb_{\ell}} \big[ \xi \big| \Fc^0 \big] \Big]
		=
		\int_0^1 \int_{\Om} \xi(\om, F^{-1}_{\om}(u)) \P(d \om) du
		=
		\int_{\Om}
		\xi(\om,\tau_\ell(\om))
		\P(d\om).
	\end{equation*}
	Let us consider the bounded random variable $\xi : \Omo \longrightarrow [0,1]$ define by
	\begin{equation*}
		\xi(\om, \theta)
		~\vcentcolon = ~
		|\tau_\ell(\om) - \theta| \wedge 1.
	\end{equation*}
	As $\theta \longmapsto \xi(\om, \theta)$ is continuous, and $\Pb_{\ell,n_k} \longrightarrow \Pb_{\ell}$ under the stable convergence topology,
	it follows that
	\begin{equation*}
		\lim_{k \to \infty} \E \big[ \big|\tau_\ell - \tau^{n_k}_\ell \big| \wedge 1 \big]
		~=~
		\lim_{k \to \infty}\E^{\Pb_{\ell, n_k}}[ \xi ]
		~ = ~
		\E^{\Pb_{\ell}}[\xi]
		~=~
		\E \big[ \big|\tau_\ell - \tau_\ell \big| \wedge 1 \big]
		~ = ~ 0.
	\end{equation*}	
	In fact, the above arguments show that, for all $\ell \in \R \setminus \L_0$,
	any subsequence $\{n_k\}_{k \ge 0}$ has a subsequence $\{n_{k_i}\}_{i \ge 0}$ such that
	$$
		\lim_{i \to \infty} \E \big[ \big|\tau_\ell - \tau^{n_{k_i}}_\ell \big| \wedge 1 \big],
	$$
	which proves the Item $\mathrm{(ii)}$ in the statement when $T < \infty$.
	
	\vspace{0.5em}
	
	To conclude the proof in the case where $T< \infty$, it is enough to prove the claim in \eqref{eq:claim_stable_cvg}.
	For the first equality in \eqref{eq:claim_stable_cvg}, we notice that when $k \longrightarrow +\infty$,
        \begin{equation*}
		\Big |
		\sup_{\tau \in \Tc}J^{n_k}_\ell(\tau) - \sup_{\tau \in \Tc}J_\ell(\tau)
		\Big |
		\le
		\E\bigg[
			\int_{0}^{T}
			|f_{n_k}(t,\ell) - f(t,\ell)|
			dt
			+
			\sup_{t \in [0,T]}
			|Y^{n_k}_t - Y_t |
		\bigg]
		~\longrightarrow~0.
	\end{equation*}
	The second equality  in \eqref{eq:claim_stable_cvg} follows by the definition of $\Pb_{\ell, n_k}$, together with the fact that $\tau^{n_k}_{\ell}$ is an optimal solution to the optimal stopping problem \eqref{eq:stability_stopping_n} with $n = n_k$.
	For the third equality  in \eqref{eq:claim_stable_cvg}, we first notice that
	\begin{align*}
		&\bigg| \E^{\Pb_{\ell,n_k}} \bigg[Y^{n_k}_{\Theta} + \int_{0}^{\Theta} f_{n_k}(t, \ell)dt \bigg] - \E^{\Pb_{\ell,n_k}} \bigg[Y_{\Theta} + \int_{0}^{\Theta} f (t, \ell)dt \bigg] \bigg| \\
		\le ~&
		\E\bigg[
			\int_{0}^{T}
			|f_{n_k}(t,\ell) - f(t,\ell)|
			dt
			+
			\sup_{t \in [0,T]}
			|Y^{n_k}_t - Y_t |
		\bigg]
		~\longrightarrow~0
		~\mbox{as}~
		k \longrightarrow \infty.
	\end{align*}
	Then, whenever the limit exists,
	$$
		\lim_{k \to \infty}
		\E^{\Pb_{\ell,n_k}} \bigg[Y^{n_k}_{\Theta} + \int_{0}^{\Theta} f_{n_k}(t, \ell)dt \bigg]
		~=~
		\lim_{k \to \infty}
		\E^{\Pb_{\ell,n_k}} \bigg[Y_{\Theta} + \int_{0}^{\Theta} f(t, \ell)dt \bigg].
	$$
	
	Further, recall that $t \mapsto Y_t$ is almost surely upper semi-continuous, and $\Pb_{\ell, n_k} \longrightarrow \Pb_{\ell}$ under the stable convergence topology.
	Then for every $K > 0$, one has
	$$
		\lim_{k \to \infty}
		\E^{\Pb_{\ell,n_k}} \bigg[ -K \vee \bigg( Y_{\Theta} + \int_{0}^{\Theta} f(t, \ell)dt \bigg) \wedge K \bigg]
		~\le~
		\E^{\Pb_{\ell}} \bigg[ -K \vee \bigg( Y_{\Theta} + \int_{0}^{\Theta} f (t, \ell)dt \bigg) \wedge K \bigg].
	$$
	Since $Y$ is in class (D), together with the integrability condition of $f$, it follows that when $K \longrightarrow +\infty$
	\begin{align*}
		&\sup_{k \ge 0}
		\bigg| \E^{\Pb_{\ell,n_k}} \bigg[ -K \vee \bigg( Y_{\Theta} + \int_{0}^{\Theta} f(t, \ell)dt \bigg) \wedge K \bigg] - \E^{\Pb_{\ell,n_k}} \bigg[ Y_{\Theta} + \int_{0}^{\Theta} f(t, \ell)dt  \bigg] \bigg| \\
		\le~ &
		\sup_{\tau \in \Tc}
		\E\bigg[
			\bigg(|Y_{\tau}| + \int_{0}^{T}|f(t,\ell)|dt\bigg)
			\mathds{1}_{\{(|Y_{\tau}| + \int_{0}^{T}|f(t,\ell)|dt) > K\}}
		\bigg]
		\longrightarrow 0,
	\end{align*}
	and
	$$
		\bigg| \E^{\Pb_{\ell}} \bigg[ -K \vee \bigg( Y_{\Theta} + \int_{0}^{\Theta} f(t, \ell)dt \bigg) \wedge K \bigg] - \E^{\Pb_{\ell}} \bigg[ Y_{\Theta} + \int_{0}^{\Theta} f(t, \ell)dt  \bigg] \bigg|
		\longrightarrow 0.
	$$
	This leads to
	$$
		\lim_{k \to \infty}
		\E^{\Pb_{\ell,n_k}} \bigg[Y^{n_k}_{\Theta} + \int_{0}^{\Theta} f_{n_k}(t, \ell)dt \bigg]
		~\le~
		\E^{\Pb_{\ell}} \bigg[ Y_{\Theta} + \int_{0}^{\Theta} f(t, \ell)dt  \bigg].
	$$

	Finally, notice that, for all (integrable) random variable $\xi: \Omo \longrightarrow \R$, one has
        \begin{align*}
            ~
            \E^{\Pb_{\ell}} \big[ \xi \big]
            ~ &=  ~
            \E^{\Pb_{\ell}} \Big[\E^{\Po} \big[ \xi \big| \Fcb \big] \Big]
           \\~&=~
            \int_\Om
                \int_{0}^{T}
                    \xi(\om, \theta)\overline{F}_\om(d\theta)
			\P(d\om)
		~=~
		\int_{\Om}
		\int_0^1
		\xi(\om,\overline{F}^{-1}_\om(u)) du \P(d\om).
        \end{align*}
	With
	$$
		\xi(\om, \theta) ~:=~ Y_{\theta}(\om) + \int_0^{\theta} f(t, \om, \ell) dt,
	$$
	it follows that
	$$
		\E^{\Pb_{\ell}} \bigg[Y_{\Theta} + \int_{0}^{\Theta} f(t, \ell)dt \bigg]
		=
		\int_0^1 \E \bigg[Y_{F^{-1}_{\cdot}(u)} + \int_{0}^{F^{-1}_{\cdot}(u)} f(t, \ell)dt \bigg] du.
	$$
	Since $F^{-1}_{\cdot}(u)$ are stopping times w.r.t. $\F$, it follows that
	\begin{align*}
		\sup_{\tau \in \Tc}J_\ell(\tau)
		&~=~
		\lim_{k \to \infty}
		\E^{\Pb_{\ell,n_k}} \bigg[Y^{n_k}_{\Theta} + \int_{0}^{\Theta} f_{n_k}(t, \ell)dt \bigg]
		~\le~
		\E^{\Pb_{\ell}} \bigg[Y_{\Theta} + \int_{0}^{\Theta} f(t, \ell)dt \bigg] \\
		&~=~
		\int_0^1 \E \bigg[Y_{F^{-1}_{\cdot}(u)} + \int_{0}^{F^{-1}_{\cdot}(u)} f(t, \ell)dt \bigg] du
		~\le~
		\sup_{\tau \in \Tc}J_\ell(\tau),
        \end{align*}
	which concludes Claim \eqref{eq:claim_stable_cvg}.

	\vspace{0.5em}
	
	\noindent $\mathrm{(ii.b)}$ Let us now consider the case where $T = +\infty$. We fix an arbitrary constant $N > 0$.
	
	\vspace{0.5em}
	
	For $n \ge 0$, let us define the random variables:
	$$
		Z_N
		\vcentcolon=
		\esssup_{\tau \in \Tc, ~\tau \ge N}
		\E\bigg[ Y_\tau + \int_{N}^{\tau}f(t,\ell)dt\Big|\Fc_N\bigg],
	$$
        $$
		Z^n_N
		\vcentcolon=
		\esssup_{\tau \in \Tc, ~\tau \ge N}
		\E\bigg[ Y^n_\tau + \int_{N}^{\tau}f_n(t,\ell)dt\Big|\Fc_N\bigg],
	$$
	and then processes $\widehat Y$ and $\widehat Y^n: [0,N] \x \Om \longrightarrow \R$ by
	$$
		\Yh_t := Y_t \mathds{1}_{\{t \in [0,N)\}} + Z_N \mathds{1}_{\{t =N\}},
		~~\mbox{and}~
		\Yh^n_t := Y^n \mathds{1}_{\{t \in [0,N)\}}  + Z^n_N \mathds{1}_{\{t = N\}},
		~~t \in [0,N].
	$$
	Under condition \eqref{eq:fnYn_fY}, it is clear that
	$$
		\E \Big[ \sup_{0 \le t \le N} \big| \Yh^n_t - \Yh_t \big| \Big]
		~\longrightarrow~
		0,
		~~\mbox{as}~ n \longrightarrow \infty.
	$$
	Moreover, it is clear that $\Yh$ and $(\Yh^n)_{n \ge 1}$ are all optional processes of Class (D), are u.s.c. in expectation and have left upper semi-continous paths.

	\vspace{0.5em}

	Let us denote by $\Tc_{[0,N]}$ the collection of all $\F$-stopping times $\tau$ taking value in $[0,N]$,
	we consider the following optimal stopping problems:
	\begin{equation}\label{eq:stability_stopping_alternative}
		\sup_{\tau \in \Tc_{[0,N]}} \widehat{J}_\ell(\tau),
		~~\mbox{with}~
		\widehat{J}_\ell(\tau)
		~\vcentcolon = ~
		\E\bigg[\Yh_\tau + \int_{0}^{\tau}f(t,\ell) dt\bigg],
	\end{equation}
	and
	\begin{equation} \label{eq:stability_stopping_alternative_n}
		\sup_{\tau \in \Tc_{[0,N]}} \widehat{J}^n_\ell(\tau),
		~\mbox{with}~
		\widehat{J}^n_\ell(\tau)
		~\vcentcolon = ~
		\E\bigg[\Yh^n_\tau + \int_{0}^{\tau}f_n(t,\ell) dt\bigg],
		~ n \ge 0.
	\end{equation}
	In fact, the above optimal stopping problems have the same Snell envelops as that for \eqref{eq:stability_stopping} and \eqref{eq:stability_stopping_n}.
	More concretely, by the dynamic programming principal of the optimal stopping problem (see e.g. El Karoui \cite{ElKaroui_SF}),
	it follows that, for any $\sigma \in \Tc_{[0,N]}$, one has
	\begin{align*}
		\Zh^{\ell}_{\sigma}
            ~ & := ~
            \esssup_{\tau \in \Tc_{[0,N]}, \tau \ge \sigma}
		\E\bigg[ \Yh_\tau + \int_{0}^{\tau}f(t,\ell)dt\Big|\Fc_\sigma\bigg]
		\\ ~ & = ~
		\esssup_{\tau \in \Tc, \tau \ge \sigma}
		\E\bigg[ Y_\tau + \int_{0}^{\tau}f(t,\ell)dt\Big|\Fc_\sigma\bigg]
		~ =: ~
		Z^{\ell}_{\sigma},
	\end{align*}
	and
	\begin{align*}
		\Zh^{\ell,n}_t
		~ & := ~
		\esssup_{\tau \in \Tc_{[0,N]}, \tau \ge \sigma}
		\E\bigg[ \Yh^n_\tau + \int_{0}^{\tau}f_n(t,\ell)dt\Big|\Fc_\sigma\bigg]
		\\ ~ & = ~
		\esssup_{\tau \in \Tc, \tau \ge \sigma}
		\E\bigg[ \Yh^n_\tau + \int_{0}^{\tau}f_n(t,\ell)dt\Big|\Fc_\sigma\bigg]
		~ =: ~
		Z^{\ell,n}_t.
	\end{align*}
	Recall that $\tau_{\ell}$ (resp. $\tau^n_{\ell}$) is the smallest optimal stopping times to the problems \eqref{eq:stability_stopping} (resp. \eqref{eq:stability_stopping_n}),
	then they can be also given by
	\begin{equation*}
		\tau_{\ell}
		=
		\inf \bigg\{t \ge 0: Y_t + \int_{0}^{t}f_n(s,\ell)ds = Z^{\ell}_t \bigg\},
            ~ \mbox{a.s.},
	\end{equation*}
    and
        \begin{equation*}
		\tau^n_{\ell}
		=
		\inf \bigg\{t \ge 0: Y^n_t + \int_{0}^{t}f_n(s,\ell)ds = Z^{\ell,n}_t \bigg\},
		~ \mbox{a.s.}
        \end{equation*}
	It follows then the smallest optimal stopping times $\tauh_{\ell,N}$ (resp. $\tauh^n_{\ell,N}$) of \eqref{eq:stability_stopping_alternative} (resp. \eqref{eq:stability_stopping_alternative_n}) satisfies
	$$
		\tauh_{\ell,N}
		~\vcentcolon =~
		\inf \bigg\{t \ge 0: \Yh_t + \int_{0}^{t}f(s,\ell)ds = \Zh_t \bigg\}
		~=~
		\tau_{\ell} \wedge N,
		~\mbox{a.s.},
	$$
	and
	$$
		\tauh^n_{\ell,N}
		~\vcentcolon =~
		\inf\bigg\{t \ge 0: \Yh^n_t + \int_{0}^{t}f_n(s,\ell)ds = \Zh^n_t \bigg\}
		~=~
		\tau^n_{\ell} \wedge N,
		~\mbox{a.s.}
	$$
	We thus reduces the problem to the case $T = N < \infty$ by considering the optimal stopping problems \eqref{eq:stability_stopping_alternative} and \eqref{eq:stability_stopping_alternative_n}.
	\end{proof}

	\vspace{0.5em}

	\noindent{\sl Proof of Theorem \ref{Thm:stability}}
	Let us consider an arbitrary subsequence $\{n_k\}_{k \in \N}$ of $\{n\}_{n \in \N}$.
 	By Proposition \ref{prop:cvg_st} and with the corresponding countable set $\L_0 \subset \R$,
	one can find a countable and dense subset $\L_1 \subset \R \setminus \L_0$ and a subsequence $\{n_{k_m}\}_{m \in \N}$ of $\{n_k\}_{k \in \N}$,
	such that
	\begin{equation*}
		\lim_{m \to \infty}\tau^{n_{k_m}}_\ell
		~=~
		\tau_\ell,
		~\mbox{a.s., for all}~
		\ell \in \L_1.
	\end{equation*}
	Since for $\P$-a.e. $\om \in \Om$, the maps $\ell \longmapsto \tau^n_\ell$ and $\ell \longmapsto \tau_\ell$
	are left-continuous and non-decreasing,
	it follows that, for $\P$-a.e. $\om \in \Om$, one has
	\begin{equation*}
		\lim_{m \to \infty}\tau^{n_{k_m}}_\ell(\om)
		~=~
		\tau_\ell(\om),
		~\mbox{whenver}~ \tau_{\ell}(\om) = \tau_{\ell+}(\om).
	\end{equation*}

	Then by Billingsley \cite[Theorem 25.6]{Billingsley2008}, for $\P$-a.e. $\om \in \Om$,
	\begin{equation*}
		\lim_{m \to \infty}\Lh^{n_{k_m}}_t(\om)
		~ = ~
		\Lh_t(\om),
		~\mbox{whenever}~
		\Lh_t(\om) = \Lh_{t+}(\om).
	\end{equation*}
	Hence, by Proposition \ref{prop:ac_[0,T]Polish}, one has
        \begin{equation*}
            \lim_{k \to \infty} d_L(\Lh^{n_{k_m}},\Lh)
            ~ = ~
            0.
        \end{equation*}
        Finally, since the subsequence $\{n_k\}_{k \in \N}$ is arbitrary, one can conclude that, for any $\eps > 0$,
        \begin{equation*}
            \lim_{n \to \infty}
            \P \big[d_L(\Lh^{n},\Lh)\ge \eps \big]
            ~ = ~
            0.
        \end{equation*}
	\qed

\section{Proofs}
\label{sec:proofs}

\subsection{Proofs for Examples \ref{eg:continuity} and \ref{eg:monotonicity}}
\label{subsec:proofs_examples}

	\begin{proposition}\label{prop:tightness}
		In the context of Example \ref{eg:continuity},
		the set $\{\Lc(\Psi(X^m,\Lh)|\Gc): m \in \L^0_\Gc(\Om,\Pc(E)), \Lh \in \L^0_\Fc(\Om,\V^+)\}$ is tight in $\L^0_\Gc(\Om,\Pc(E))$.
	\end{proposition}

	\begin{proof}
	Let us consider the setting $\Gc = \{\emptyset,\Om\}$, where it is sufficient to check that
        $$
            \big\{\Lc(L^- \vee \Lh \wedge L^+)|\Gc): \Lh \in \L^0_\Fc(\Om,\V^+)\big\}
        $$
        is tight in $\Pc(\V^+)$.
        In details, for all $n \in \N$,
            we first define a subset $\V^+_{-n, n}$ of $\V^+$ by
        $$
            \V^+_{-n, n}
            ~ \vcentcolon= ~
            \{\mathbf{l} \in \V^+:
            - n \leq \mathbf{l}_t \leq n,
            ~\mbox{for all}~
            t \in (0,T).
            \}.
        $$
        Clearly for all $n \in \N$,
            $\V^+_{-n, n}$ is compact
            and $\V^+ = \cup_{n = 1}^{+\infty} \V^+_{-n, n}$.
        Then for any $\eps > 0$,
            there exists some $N_\eps \in \N$,
            such that
        $$
            \P \circ (L^-)^{-1}\big(\V^+_{-N_\eps,N_\eps}\big),
            ~
            \P \circ (L^+)^{-1}\big(\V^+_{-N_\eps,N_\eps}\big)
            \geq
            1 - \frac{\eps}{2}.
        $$
        Thus, for any $\P \circ (L)^{-1}$
            with $L \in \L^0_{\Fc}(\Om, \V^+)$,
            we have
        \begin{align*}
            & ~
            \P \circ (L)^{-1}\big(\V^+_{-N_\eps,N_\eps}\big)
            ~ = ~
            \P\big(\{-N_\eps \leq L \leq N_\eps\}\big)
            \\ ~ \geq & ~
            \P\big(\{-N_\eps \leq L^- \leq N_\eps\}
                \cap \{-N_\eps \leq L^+ \leq N_\eps\}\big)
            \\ ~ = & ~
            1 - (1 - \P\big(\{-N_\eps \leq L^- \leq N_\eps\}\big))
              - (1 - \P\big(\{-N_\eps \leq L^+ \leq N_\eps\}\big))
            \\ ~ = & ~
            1 - (1 - \P \circ (L^-)^{-1}\big(\V^+_{-N_\eps,N_\eps}\big))
              - (1 - \P \circ (L^+)^{-1}\big(\V^+_{-N_\eps,N_\eps}\big))
            \\ ~ \geq & ~
            1 - \eps.
        \end{align*}
        Finally, when $\Gc$ is generated by countably many disjoint sets $\{A_i\}_{i \in \N}$ with $\P(A_i) > 0$, $i \in \N$ and $\Om = \cup_{i = 1}^\infty A_i$, $\L^0_\Gc(\Om,\Pc(E))$ can be identified as the space $(\Pc(E))^\N$ by the map $\psi:\L^0_\Gc(\Om,\Pc(E)) \longrightarrow (\Pc(E))^\N$ defined as
        $$
            \psi(m)
            ~ = ~
            (m(\om_i))_{i \in \N},
        $$
        where $\om_i \in A_i$ is arbitrarily chosen.

	\vspace{0.5em}

        Then, since for each $i \in \N$, on the probability space $(\Om, \Fc, \P_i )$, where $\P_i \vcentcolon = \frac{1}{\P(A_i)}\P(\cdot \cap A_i)$, we can argue that
        $$
            \big\{ \P_i \circ (\Psi(X^m, \Lh))^{-1} : m \in \L^0_\Gc(\Om,\Pc(E)), \Lh \in \L^0_\Fc(\Om,\V^+) \big\}
        $$
        is tight in $\Pc(E)$.

	\vspace{0.5em}

        On the other hand,
        it holds that $\psi(\Lc(\Psi(X^m,\Lh)|\Gc)) = (\P_i \circ (\Psi(X^m, \Lh))^{-1})_{i \in \N}$,
        for all $m \in \L^0_\Gc(\Om,\Pc(E)), \Lh \in \L^0_\Fc(\Om,\V^+)$.
        Hence we know that the set $\psi\big(\{\Lc(\Psi(X^m,\Lh)|\Gc): m \in \L^0_\Gc(\Om,\Pc(E)), \Lh \in \L^0_\Fc(\Om,\V^+)\}\big)$ is tight in $(\Pc(E))^\N$.
        Thus, we conclude our proof since $\psi$ is a homeomorphism.
    \end{proof}

	\begin{proposition}\label{prop:complete_lattice}
		In the setting of Example \ref{eg:monotonicity}, let us consider two fixed stochastic processes $L^-, L^+ \in \L^0_\Fc(\Om, \V^+)$ such that $L^- \le_l \L^+$.
		Let
		$$
			\L_0
			\vcentcolon=
			\{L \in \L^0_\Fc(\Om,\V^+):L^- \le_l L \le_l L^+,\},
			~~\mbox{and}~
			K
			\vcentcolon=
			\{\Lc(\Lh|\Gc): \Lh \in \L_0\}.
		$$
		Then both of $(\L_0,\le_l)$ and $(K,\le_p)$, where $\le_p$ is a partial order induced by $(E = \V^+,\le_v)$ as in \eqref{eq:def_lep}, are complete lattices.
	\end{proposition}
	\begin{proof}
        For any subset $\Gamma$ of $\L_0$, we define a process as follows, for any $t \in [0,T]$,
        $$
            \Lt_t \vcentcolon= \esssup_{L \in \Gamma}L_t,
        $$
        and the corresponding process $\Lo$ with $\Lo_t \vcentcolon = \lim_{s \in \Q \rightarrow t-}\Lt_s$ for all $t \in (0,T]$ with $\Lo_0 \vcentcolon= -\infty$.

	\vspace{0.5em}

        Then $\Lo$ has left-continuous paths by its definition, and for any $t \in [0,T]$, $\Lo_t$ is $\Fc_t$-measurable, $\Lo_t \leq \Lt_t$ a.s.
        In fact, for any $s, t \in (0,T]$ with $s < t$,
            there exist two increasing sequences $(s_n)_{n \in \N}$, $(t_n)_{n \in \N}$ of constants in $[0,T] \in \Q$ with $s_n < t_n$ for all $n \in \N$
            such that $\lim_{n \to \infty}s_n = s$, $\lim_{n \to \infty}t_n = t$.
        Moreover, there exist a $\P$-null set $N$ such that on $N^C$, $\Lt_{s_n} \leq \Lt_{t_n}$ for all $n \in \N$.
        Letting $n$ tends to $\infty$, we have that $\Lo_s \leq \Lo_t$.
        In other words, $\Lo$ admits increasing paths almost surely.
        Thus we have $\Lo \in \L_0$.

	\vspace{0.5em}

        On the other hand, we fix some $t \in [0,T] \setminus \Q$, for any $s \in \Q$ with $s < t$ and $L \in \Gamma$, the inequality $L_s \leq \Lt_s$ holds a.s.
        By the left-continuity of $L$ and definition of $\Lo$, we have $L_t \leq \Lo_t$ holds a.s.
        Thus the definition of $\Lt$ and arbitrariness of $L$,
        imply that $\Lt_t \leq \Lo_t$ a.s.
        and we can conclude that for any $t\in [0,T]$, $\Lt_t =\Lo_t$ a.s., $\Lo$ is the least upper bound of $\Gamma$.

	\vspace{0.5em}

        Similarly we can prove there exists some $L \in \L_0$ is the greatest lower bound of $\Gamma$.

	\vspace{0.5em}

        Then we can conclude that $(\L_0,\le_l)$ is a complete lattice.

        Finally, we observe that $\xi \longmapsto \Lc(\xi|\Gc)$ can be seen as an increasing map from $\L^0$ to $K$, then the partially ordered set $(K,\le_p)$ is also a complete lattice.
	\end{proof}

    \begin{corollary}
        Let $E := \{ \mathbf{l} \in \V^+ ~: |\mathbf{l}(t)| \le 2C, ~t \in [0,T)\}$, we consider the partial order $\le_E$ by
		$$
			\mathbf{l}^1 \le_E \mathbf{l}^2 ~~\mbox{if}~~ \mathbf{l}^1_t \le \mathbf{l}^2_t, ~\mbox{for all}~t \in [0,T).
		$$
		Then $(E, \le_E)$ is a complete lattice and the corresponding space $\L^0_{\Gc}(\Om, \Pc(E))$ with the induced order $\le_p$ is also a complete lattice.
     \end{corollary}

    \subsection{Proofs of the main theorems}\label{subsec:proof_theorems}

    \noindent\textsl{Proof of Theorem \ref{thm:continuity}.}
        First, let us define a map $\psi: K \to K$ by
        $$
        	\psi(m) = \Lc( \psi(X^{m}, \Lh^{m}) | \Gc),
	$$
        where $\Lh^{m} \vcentcolon= \sup_{t \in [0,\cdot)}L^{m}_t$
        and $L^{m}$ is the solution of the representation theorem \eqref{eq:BankEK_repres_intro} of $Y^{m}$ w.r.t. $f^{m}$.

	\vspace{0.5em}

	Let $\{m^n\}_{n \in \N} \subset K$ be such that
	$\lim_{n \to \infty}m^n = m^{\infty}$ in probability for some $m^{\infty} \in K$.
	By the conditions in the theorem, one has
		$$
			\lim_{n \to \infty}
			\E \bigg[ d_\D(X^{m^n}, X^{m^{\infty}})
                + \sup_{t \in [0,T]}
                    \big| Y^{m^n}_{t}
                         - Y^{m^{\infty}}_{t} \big|
                + \sup_{\ell \in \R}
                    \int_0^{T}
                        \big|f^{m^n}(s,\ell) - f^{m^{+\infty}}(s,\ell)\big|
                    ds
                \bigg]
            =
            0.
		$$
      	Then it follows by Theorem \ref{Thm:stability} that
        $$
            \lim_{n \to \infty}
            \P
            \big\{ d_\D(X^{m^n}, X^{m^{\infty}}) \geq \eps \big\}
                    +
            \P
            \big\{d_L\big(\Lh^{m^n}, \Lh^{m^{\infty}}\big) \geq \eps \big\}
            ~ = ~
            0,
        $$
        or equivalently,
        $$
            \lim_{n \to \infty}
            \P
            \Big\{\sqrt{
                    d_\D(X^{m^n}, X^{m^{\infty}})^2
                    + d_L\big(\Lh^{m^n}, \Lh^{m^{\infty}}\big)^2
                    }
            \geq \eps \Big\}
            ~ = ~
            0,
        $$
	The above implies the continuity of $m \mapsto (X^{m}, \Lh^{m})$ from $K$ to $\L^0_{\Fc}(\Om, \D \x \V^+)$.
	By the continuity of the map $\Psi$, one obtains the continuity of $\Lc(\Psi(\cdot)|\Gc): \L^0_{\Fc}(\Om, \D \x \V^+) \longrightarrow K$, which leads to the continuity of $\psi$.
	We can then conclude with Schauder fixed-point theorem for the existence of a fixed point $m^* \in K$ such that $m^* = \psi(m^*)$.

	\vspace{0.5em}

        In particular, the couple $\big(\Lc( \Psi(X^{m^*}, \Lh^{m^*}) | \Gc), L^{m^*}\big)$ gives a solution to the mean-field representation in \eqref{eq:MF_representation}.
	\qed
	\endproof

	\vspace{0.5em}

    \noindent \textsl{Proof of Theorem \ref{thm:monotonicity}.}
	Let us define  $\psi_1: K \longrightarrow \L^0_{\Fc}\big(\Om, \D \x \V^+\big)$, and $\psi_2: \L^0_{\Fc}\big(\Om, \D \x \V^+\big) \longrightarrow K$ by
	\begin{equation}
            \psi_1(m) := (X^m, \Lh^{m}),
            ~~
            \psi_2(X, \Lh) := \Lc(\Psi(X,\Lh)|\Gc).
        \end{equation}
	Then $\psi_2$ is order-preserving by order-preserving property of $\Psi$.
	We claim that $\psi_1$ is also order-preserving.

	\vspace{0.5em}

	Indeed, by Theorem \ref{thm:BErepresentationTheorem} (ii),
        for any $m \in K$, we have
        \begin{equation} \label{eq:defLm}
            L^{m}_t ~ = ~
            \essinf_{\sigma \in \Tc_{t,+}} \ell^{m}_{t,\sigma},
        \end{equation}
        where $\ell^{m}_{t,\sigma}$ is the unique $\Fc_t$-measurable random variable such that
        \begin{equation*}
            \E\bigg[
                \int_{t}^{\sigma}
                    f^{m}(s, \ell^{m}_{t,\sigma})
                ds
            \Big| \Fc_t\bigg]
            ~ = ~
            \E[Y^{m}(t) - Y^{m}(\sigma) | \Fc_t].
        \end{equation*}
        For any $m^1, m^2 \in K$ with $m_1 \leq_p m_2$, the conditions in the theorem implies that the process $Y^{m_1}(\cdot) - Y^{m^2}(\cdot)$ is a $\F$-supermartingale,
        then for any $\sigma \in \Tc_{t,+}$, we have
        \begin{equation*}
            Y^{m^1}(t) - Y^{m^2}(t)
            ~\geq ~
            \E[Y^{m^1}(\sigma) - Y^{m^2}(\sigma) | \Fc_t],
        \end{equation*}
        i.e.
        \begin{equation*}
            \E[Y^{m^1}(t) - Y^{m^1}(\sigma) | \Fc_t]
            ~ \geq ~
            \E[Y^{m^2}(t) - Y^{m^2}(\sigma) | \Fc_t].
        \end{equation*}
        Then since $f$ is order-reversing w.r.t. $m$ (see \eqref{eq:XYf_order}), we get the estimate
        \begin{align*}
            & ~
            \E\bigg[
                \int_{t}^{\sigma}
                    f^{m^1}(s, \ell^{m^1}_{t,\sigma})
                ds
            \Big| \Fc_t\bigg]
            =  ~
            \E[Y^{m^1}(t) - Y^{m^1}(\sigma) | \Fc_t]
            \\
            \geq & ~
            \E[\Yt^{m^2}(t) - Y^{m^2}(\sigma) | \Fc_t]
            =  ~
            \E\bigg[
                \int_{t}^{\sigma}
                    f^{m^2}(s, \ell^{m^2}_{t,\sigma})
                ds
            \Big| \Fc_t\bigg]
            \\
            \geq & ~
            \E\bigg[
                \int_{t}^{\sigma}
                    f^{m^1}(s, \ell^{m^2}_{t,\sigma})
                ds
            \Big| \Fc_t\bigg]
            ~\mbox{a.s.}
        \end{align*}
        Since $f^{m^1}(\cdot, \ell)$ is strictly increasing in $\ell$, one can then conclude that $\ell^{m^1}_{t,\sigma} \geq \ell^{m^2}_{t,\sigma}$,
        hence $\Lh^{m^1}_t \geq \Lh^{m^2}_t$ a.s. for all $t \in [0,T)$ by \eqref{eq:defLm}.
        This means that $\psi_1$ is also order-preserving.

  	\vspace{0.5em}

        Consequently, $\psi_2 \circ \psi_1 : K \to K$ is order-preserving.

	\vspace{0.5em}

        Finally, by the Tarski's fixed point theorem, there exists some $m^* \in K$ such that $\psi_2 \circ \psi_1 (m^*) = m^*$.
        Let $L^*$ be the solution of the representation theorem \eqref{eq:BankEK_repres_intro} of $Y^{m^*}$ w.r.t. $f^{m^*}$, $\Lh^* \vcentcolon= \psi_1(m^*)$, $m^* = \Lc\big(\Psi(X^{m^*},\Lh^*)|\Gc\big)$, then $(L^*, m^*)$ is a solution to the mean-field representation problem \eqref{eq:MF_representation}.
    \qed
    \endproof

    \vspace{0.5em}

	\noindent \textsl{Proof of Theorem \ref{thm:contraction}.}
	$\mathrm{(i)}$ Let $L$ be the unique solution to the representation \eqref{eq:BankEK_repres_intro} problem with $\Yt$ and $\ft$.
	For each $m \in \Pc(\V^+)$, let us define the left-continuous increasing process $L^m$ by
	\begin{equation} \label{eq:LmLm}
            L^m_{\cdot} := L_\cdot +
               m_{\cdot}(\phi).
	\end{equation}
	We observe that $L^m$ is the unique solution to the representation problem \eqref{eq:BankEK_repres_intro} with $\Yt$ and generator $f^m$.
	Then, in the setting of Theorem \ref{thm:contraction}, the mean-field representation formula \eqref{eq:MF_representation} reduces to
        \begin{equation}\label{eq:fixedpoint_uniqueness}
            m
            ~ = ~
            \P \circ \big(L + m_{\cdot}(\phi) \big)^{-1}.
        \end{equation}

        \noindent $\mathrm{(ii)}$ Given a solution $m$ to the fixed-point problem \eqref{eq:fixedpoint_uniqueness},
	by integrating $\phi$ w.r.t. the marginal law of both sides in \eqref{eq:fixedpoint_uniqueness} at any time $t \in [0,T]$, one has that
        \begin{equation*}
           	m_t(\phi)
            ~ = ~
            \E\big[\phi\big(L_t + m_t(\phi) \big)\big],
            ~\mbox{for all}~
            t \in [0,T].
        \end{equation*}
        In other words, for each $t \in [0,T]$, $m_t(\phi)$ is a solution to the fixed point problem
        \begin{equation}\label{eq:fixed_point_contraction}
            \Phi(t,y) := \E[\phi(L_t + y)] = y, ~ y \in \R.
        \end{equation}

	\noindent $\mathrm{(iii)}$ Let $(y^*_t)_{t \in [0,T]}$ be a solution to the fixed point problem \eqref{eq:fixed_point_contraction},
	then
	\begin{equation*}
		L + y^*
		~=~
		L + \E[\phi(L_\cdot+ y^*_\cdot)]
		~ = ~
		L + \int_{\R}\phi(x)\P \circ (L + y^*)^{-1}_{\cdot}(dx)
		~=~
		L^{\P \circ (L + y^*)^{-1}},
	\end{equation*}
        where the last equality follows by the fact that $L^m$ in \eqref{eq:LmLm} provides the unique solution solution to the representation problem \eqref{eq:BankEK_repres_intro}.
        Consequently, $(\P \circ (L + y^*)^{-1}, L + y^*)$ is a solution to the fixed point problem \eqref{eq:fixedpoint_uniqueness}.

        \vspace{0.5em}

        \noindent $\mathrm{(iv)}$ Finally, we prove that, for each $t \in [0,T]$, there exists a unique solution to \eqref{eq:fixed_point_contraction}.
	In fact, since $ \phi^\prime \in [0,1)$,
	one has
	$$
            \frac{\partial\Phi}{\partial y}(t,y)  = \E[\phi^\prime(L_t + y)] \in [0,1),
            ~\mbox{for all}~
            (t,y) \in [0,T] \x \R.
        $$
	Together with the boundedness of $\phi$, one obtains that the fixed point problem \eqref{eq:fixed_point_contraction} has a unique solution $y^*_t \in \R$ for each $t \in [0,T]$.
	Moreover, as $L_t$ is increasing in $t$, this implies that $\Phi(t,y) = \E[ \phi(L_t + y) ]$ is also increasing in $t$, and so is $y^*_t$ in $t$.
	\qed
	\endproof

\subsection{Proofs of the results in Section \ref{subsec:applications}}
\label{sec:proofs_applications}

Let us first provide

\vspace{0.5em}

	\noindent \textsl{Proof of Proposition \ref{prop:optimal_stopping_application}.}

	\vspace{0.5em}

	\noindent \textbf{Step 1. Define the Polish space $E$ and the tuple $(X^m, Y^m, f^m)$.}

        Let us fix $\varepsilon > 0$.
        Recall that
		there exists a $\delta_\varepsilon > 0$ independent of $\om$ and $\mu$
		such that
		$$
			|g_{\ell_1}(t, \om, \mu) - g_{\ell_2}(t, \om, \mu)| ~ \leq ~ \varepsilon,
			~~\mbox{whenever}~
			|\ell_1 - \ell_2| \leq \delta_\varepsilon,
		$$
        and denote by $\Et_\eps$ the collection of elements $(t_\ell)_{\ell \in \R}$ in $[0,T]^\R$ such that
        the map $\ell \in \R \longrightarrow t_\ell \in [0,T]$ is increasing and Lipschitz continuous with Lipschitz constant $\frac{1}{\delta_{\eps/(3T)}}$.

        Then we equip $\Et_\eps$ with the L\'evy metric $d_L$ so that $\Et_\eps$ is a Polish space and its topology induced by $d_L$ coincides with the product topology of $[0,T]^\R$ restricted on $\Et_\eps$.

        Let $E = \D \x \Et_\eps$, for each $m \in \L^0_{\Gc}(\Om, \Pc(E))$, we can define $X^m: [0,T] \x \Om \longrightarrow \R$, $Y^m: [0,T] \x \Om \longrightarrow \R$ and $f^m: [0,T] \x \Om \x \R \longrightarrow \R$ by
        $$
            X^m(t, \om)
            \vcentcolon=
            X^{m}(t, \om),
            ~~
            Y^m(t, \om)
            \vcentcolon=
            G(t, \om, m),
            ~~
            f^m(t, \om, \ell)
            \vcentcolon=
            g_\ell(t, \om, m).
        $$
        Then it is easy to verify that $\{(X^m, Y^m, f^m)\}_{m \in \L^0_{\Gc}(\Om, \Pc(E))}$ satisfy Assumption \ref{assum:Yfm}.

        \vspace{0,5em}

	\noindent \textbf{Step 2. Define the couple $(K, \Psi)$ and verify their properties.}

        Let $\Psi :\Om \x \D \x \V^+ \to \D \x [0,T]^\R$ be defined by
        $$
            \Psi(\om,\mathbf{x},\mathbf{l})
            \vcentcolon =
            \bigg(\mathbf{x}, \Big(\tau^{\mathbf{l} + \delta_{\eps/(3T)}I}_\ell\Big)_{\ell \in \R}\bigg),
        $$
        where 
	$$
		\tau^\mathbf{l}_\ell \vcentcolon= \inf\{t \geq 0; \mathbf{l}_{t} \geq \ell\}, 
		~\mbox{for all}~
		\ell \in \R 
		~\mbox{and}~
		\mathbf{l} \in \V^+,
	$$
	and $I$ is the identity function on $[0,T]$.
        It is easy to see that $\Psi$ takes value in $E$.

        \vspace{0,5em}

        Further, let the space $K$ be defined by
        \begin{align*}
            ~
            K
            ~  & \vcentcolon=  ~
            \overline{
            conv\big(\{
            \Lc((X^m, (\tau_\ell)_{\ell \in \R})|\Gc)
            :
            m \in \L^0_{\Gc}(\Om, \Pc(E)),
            \tau_\ell \in \Tc,
            \ell \in \R,
            (\tau_\ell)_{\ell \in \R} \in \Et_\eps
            \}\big)},
        \end{align*}
        where $conv(P)$ denote the convex hull of the set $P$ in $\L^0_\Gc(\Om, \Pc(E))$, i.e. the intersection of all convex sets in $\L^0_\Gc(\Om, \Pc(E))$ that contain $P$.
	
        Then the space $K$ is obviously nonempty, convex and closed.
        Tychonoff's theorem implies the compactness of $[0,T]^\R$, endowed with product metric.
        By the tightness of
            $\{\Lc(X^m|\Gc): m \in \L^0_{\Gc}(\Om, \Pc(\D \x [0,T]^\R))\}$
            and
            $\{\Lc((\tau_\ell)_{\ell \in \R}|\Gc) : \tau_\ell \in \Tc, \ell \in \R\}$,
        one has that $K$ is compact.
        Meanwhile, the space $\L^0_\Gc(\Om, \Pc(E))$ is a Hausdorff locally convex topological vector space.

	Next, for $\mu^n, \mu^\infty \in K$, $n \in \N$
	with $\lim_{n \to \infty}m^n = m^\infty$ in probability,
	since $\Et_\eps$ is a closed subset of $[0,T]^\R$,
        it holds that
        $\lim_{n \to \infty}m^n = m^\infty$ almost surely.
        By assumption (iii) in the statement, we have for any $\ell \in \R$,
        $$
    		\lim_{n \to \infty}
            \E \bigg[ d_\D(X^{m^n}, X^{m^{\infty}})
                + \sup_{t \in [0,T]}\big| Y^{m^n}_{t} - Y^{m^{\infty}}_{t} \big|
                + \int_0^{T} |f^{m^n}(t,\ell) - f^{m^\infty}(t,\ell)| dt
            \bigg]
            ~ = ~
            0.
        $$

        Then it remains to verify that the map $\Psi$ is continuous from $\D \x \V^+$ to $E$, for all $\om \in \Om$.
        We claim that if the elements $\{\mathbf{l}^n\}_{n \in \N}, \mathbf{l}^\infty$ in $\V^+$ are such that $\lim_{n \to \infty}d_L(\mathbf{l}^n,\mathbf{l}^\infty) = 0$ and $\mathbf{l}^n$, $\mathbf{l}^\infty$ are strictly increasing, then
	\begin{equation} \label{eq:claim_timing}
		\lim_{n \to \infty}\tau^{\mathbf{l}^n}_\ell = \tau^{\mathbf{l}^\infty}_\ell,
		~~\mbox{a.s., for all}~
		\ell \in \R.
	\end{equation}

	Indeed, since $\lim_{n \to \infty}d_L(\mathbf{l}^n,\mathbf{l}^\infty) = 0$,
	by Proposition \ref{prop:ac_[0,T]Polish},
        we have
        $\mathbf{l}^n_\cdot$ converges to $\mathbf{l}^\infty_\cdot$ on every continuity points of $L^\infty_\cdot$,
        and then $ \tau^{\mathbf{l}^n}_\ell$ converges to $ \tau^{\mathbf{l}^\infty}_\ell$ on every continuity points of $\ell \longmapsto \tau^{\mathbf{l}^\infty}_\ell$, as $n$ tends to $\infty$.
        At the same time, as inverse function of the strictly increasing function $t \longmapsto \mathbf{l}^{\infty}_t$, $\ell \longmapsto \tau^{\mathbf{l}^{\infty}}_\ell$ is continuous for all $\ell \in \R$.
	This proves the claim in \eqref{eq:claim_timing}.

        \vspace{0,5em}

	Further, by Proposition \ref{prop:ac_[0,T]Polish}, $\Psi$ is clearly continuous.
	Hence, we can apply Theorem \ref{thm:continuity} to obtain the existence of some $m^*_\varepsilon \in \L^0_{\Gc}(\Om,\Pc(E))$ such that
	$m^*_\varepsilon = \Lc(\Psi(X^{m^*_\varepsilon},\Lh^{m^*_\varepsilon})|\Gc) $,
	where $\Lh^m$ is the running maximum process defined by $\Lh^m_t \vcentcolon= \sup_{s \in [0,t)}L^m_s$ with the convention that $\sup \emptyset = -\infty$ and $L^m$ is the solution to the representation theorem \eqref{eq:BankEK_repres_intro} of $Y^m$ w.r.t. $f^m$.

	\vspace{0.5em}

	\noindent \textbf{Step 3. Existence of the $\eps$-mean field equilibrium.}

        We claim that the pair
        $$
        \bigg(
        m^*_\eps,
        \Big(\tau^{\Lh^{m^*_\varepsilon} + \delta_{\eps/(3T)}I}_\ell\Big)_{\ell \in \R}\bigg)
        $$
        is a $\varepsilon$-mean field equilibrium.
        The stopping time $\tau^{\Lh^{m^*_\varepsilon} + \delta_{\eps/(3T)}I}_\ell$ is hereinafter abbreviated as $\tau^{*,\varepsilon}_\ell$ and one may define a map $\psi: \V^+ \longrightarrow \V^+$ by
        $$
            \psi(\mathbf{l})
            \vcentcolon =
            \mathbf{l} + \delta_{\eps/(3T)}I.
        $$

        First, the consistency condition
        $$
            m^*_\varepsilon
            ~ = ~
            \Lc\bigg(\bigg(X^{m^*_\varepsilon}, \Big(\tau^{\Lh^{m^*_\varepsilon} + \delta_{\eps/(3T)}I}_\ell\Big)_{\ell \in \R}\bigg)\bigg|\Gc\bigg),
        $$
        holds by the definition of $\Psi$.

	\vspace{0.5em}

	Then it is sufficient to verify that, for any $\tau \in \Tc$, one has
	$$
		J_\ell\big(\tau^{*,\varepsilon}_\ell, \Psi(m^*_\varepsilon)\big)
		~ \geq ~
		J_\ell\big(\tau,\Psi(m^*_\varepsilon)\big) - \varepsilon.
	$$
        By Theorem \ref{thm:BErepresentationTheorem} (iii), we know that $\tau^{\Lh^{m^*_\varepsilon}}_\ell$(hereinafter abbreviated as $\tau^*_\ell$) is the smallest optimal stopping time that maximizes the objective functional $J_\ell$ w.r.t. $m^*_\varepsilon$,
        i.e. for any $\tau \in \Tc$, we have
        $$
            J_\ell\big(\tau^*_\ell,m^*_\varepsilon\big)
            ~ \geq ~
            J_\ell\big(\tau, m^*_\varepsilon\big).
        $$
        Then we have the estimation
        \begin{align*}
            & ~
            \Big|
            J_{\ell}
            \big(\tau^{*,\varepsilon}_{\ell}, m^*_\varepsilon\big)
            -
            J_\ell\big(\tau^*_\ell, m^*_\varepsilon\big)
            \Big|
        \\ = & ~
            \bigg|
            \E\bigg[
                \int_{0}^{\tau^{*,\varepsilon}_{\ell}}
                g_\ell\big(t, m^*_\varepsilon\big) dt
            +
                G\big(\tau^{*,\varepsilon}_{\ell}, m^*_\varepsilon\big)
            -
                \int_{0}^{\tau^{*}_{\ell}}
                g_\ell\big(t,\Psi(m^*_\varepsilon)\big) dt
            -
                G\big(\tau^{*}_{\ell}, m^*_\varepsilon\big)
            \bigg]
            \bigg|
        \\ = & ~
            \bigg|
            \E\bigg[
            \int_{0}^{\tau^{*,\varepsilon}_{\ell}}
                f^{m^*_\varepsilon}(t,\ell) dt
                + Y^{m^*_\varepsilon}(\tau^{*,\varepsilon}_{\ell})
            -
            \int_{0}^{\tau^{*}_{\ell}}
                f^{m^*_\varepsilon}(t,\ell) dt
                - Y^{m^*_\varepsilon}(\tau^{*}_{\ell})
            \bigg]
            \bigg|
        \\ = & ~
            \bigg|
            \E\bigg[
            \int_{0}^{\tau^{*,\varepsilon}_{\ell}}
                f^{m^*_\varepsilon}(t,\ell) dt
                +
            \int_{\tau^{*,\varepsilon}_{\ell}}^{T}
                f^{m^*_\varepsilon}\big(t,\sup_{r \in [\tau^{*,\varepsilon}_{\ell},t)}
                    L^{m^*_\varepsilon}_r\big) dt
            \\ & ~~~~~~~~~~~~~~~~~~~~~~~~~~~~~~~~~~~~~~~~~-
            \int_{0}^{\tau^{*}_{\ell}}
                f^{m^*_\varepsilon}(t,\ell) dt
                -
            \int_{\tau^{*}_{\ell}}^{T}
                f^{m^*_\varepsilon}\big(t,\Lh^{m^*_\varepsilon}_t\big) dt
            \bigg]
            \bigg|
        \\ = & ~
            \bigg|
            \E\bigg[
            \int_{0}^{T}
                f^{m^*_\varepsilon}\Big(t,\sup_{r \in [\tau^{*,\varepsilon}_{\ell},t)}
                    L^{m^*_\varepsilon}_r\mathds{1}_{\{\psi(\Lh^{m^*_\varepsilon})_t \geq \ell\}} + \ell\mathds{1}_{\{\psi(\Lh^{m^*_\varepsilon})_t < \ell\}}\Big) dt
            \\ &~~~~~~~~~~~~~~~~~~~~~~~~~~~~~~~~~~~~~~~~~-
            \int_{0}^{T}
                f^{m^*_\varepsilon}\big(t,\Lh^{m^*_\varepsilon}_t \vee \ell
                \big) dt
            \bigg]
            \bigg|
        \\ \leq & ~
            \bigg|
            \E\bigg[
            \int_{0}^{T}
                f^{m^*_\varepsilon}\Big(t,\psi(\Lh^{m^*_\varepsilon})_t \vee \ell \Big) dt
            -
            \int_{0}^{T}
                f^{m^*_\varepsilon}\big(t,\Lh^{m^*_\varepsilon}_t \vee \ell
                \big) dt
            \bigg]
            \bigg|
            \\ & ~ + ~
            \bigg|
            \E\bigg[
            \int_{0}^{T}
                f^{m^*_\varepsilon}\Big(t,\Lh^{m^*_\varepsilon}_t\mathds{1}_{\{\psi(\Lh^{m^*_\varepsilon})_t \geq \ell\}} + \ell\mathds{1}_{\{\psi(\Lh^{m^*_\varepsilon})_t < \ell\}}\Big) dt
                \\ & ~~~~~~~~~~~~~~~~~~~~~~~~~~~~~~~~~~~~~~~~~-
            \int_{0}^{T}
                f^{m^*_\varepsilon}\Big(t,\psi(\Lh^{m^*_\varepsilon})_t\vee \ell \Big) dt
            \bigg]
            \bigg|
            \\ & ~ + ~
            \bigg|
            \E\bigg[
            \int_{0}^{T}
                f^{m^*_\varepsilon}\Big(t,\sup_{r \in [\tau^{*,\varepsilon}_{\ell},t)}
                    L^{m^*_\varepsilon}_r\mathds{1}_{\{\psi(\Lh^{m^*_\varepsilon})_t \geq \ell\}} + \ell\mathds{1}_{\{\psi(\Lh^{m^*_\varepsilon})_t < \ell\}}\Big) dt
                \\ & ~~~~~~~~~~~~~~~~~~~~~~~~~~~~~~~~~~~~~~~~~-
                \int_{0}^{T}
                    f^{m^*_\varepsilon}\Big(t,\Lh^{m^*_\varepsilon}_t\mathds{1}_{\{\psi(\Lh^{m^*_\varepsilon})_t \geq \ell\}} + \ell\mathds{1}_{\{\psi(\Lh^{m^*_\varepsilon})_t < \ell\}}\Big) dt
                \bigg]
            \bigg|
        \\ \leq & ~
        \varepsilon.
        \end{align*}
        In above, the first two equalities follow by the definition of $J_\ell$, $f^m$ and $Y^m$,
        the third equality holds by Theorem \ref{thm:BErepresentationTheorem} and that $\tau^{*}_\ell$, is a hitting time of $\Lh^{m^*_\varepsilon}$,
        the fourth equality holds by the strictly increasing property of $\ell \longmapsto f^m(t, \om, \ell)$, the first inequality is trivial.
        For the last inequality, it is enough to notice that on $\{s < \tau^{*,\varepsilon}_{\ell}\}$, $L^{m^*_\varepsilon}_s < \ell - \delta_{\eps/(3T)}s$
        and for $t > \tau^{*,\varepsilon}_{\ell}$, $\sup_{r \in [\tau^{*,\varepsilon}_{\ell},t)}
                    L^{m^*_\varepsilon}_r \geq \ell - \delta_{\varepsilon/(3T)}T
                    \geq
                    L^{m^*_\varepsilon}_s - \delta_{\varepsilon/(3T)}$T,
        and hence $\sup_{r \in [\tau^{*,\varepsilon}_{\ell},t)}
                    L^{m^*_\varepsilon}_r \geq \Lh^{m^*_\varepsilon}_t - \delta_{\varepsilon/(3T)}T \geq \sup_{r \in [\tau^{*,\varepsilon}_{\ell},t)}
                    L^{m^*_\varepsilon}_r - \delta_{\varepsilon/(3T)}$T.

	\vspace{0.5em}

	The above estimation implies the $\varepsilon$-optimality of $\tau^{*,\varepsilon}_\ell$,
	which are defined as hitting times of the process $\Lh^{m^*_\varepsilon} + \delta_{\eps/(3T)}I$.
	We hence conclude the proof.
	
	\qed
	\endproof

	\vspace{0.5em}

	\noindent \textsl{Proof of Proposition \ref{prop:optimal_stopping_application_order}.}

	\vspace{0.5em}

	\noindent \textbf{Step 1. Define the space $E$ and the tuple $(X^m, Y^m, f^m)_m$.}

	Let $E = I_c \x [0,T]^\R$.
	Then for each $m \in \L^0_{\Gc}(\Om, \Pc(I_c \x [0,T]^\R))$, we define $X^m: [0,T] \x \Om \longrightarrow \R$, $Y^m: [0,T] \x \Om \longrightarrow \R$ and $f^m: [0,T] \x \Om \x \R \longrightarrow \R$ by
	$$
            X^m(t, \om)
            ~ \vcentcolon = ~
            X^\mu(t, \om),
            ~
            Y^m(t, \om)
            ~ \vcentcolon = ~
            G(t, \om, \mu),
            ~
            f^m(t, \om, \ell)
            ~ \vcentcolon = ~
            g_\ell(t, \om, \mu),
	$$
	where $\mu \in \L^0_{\Gc}(\Om, \Pc(I_c \x [0,T]^\R))$ is such that
	$(\varphi,\mathrm{Id})\#\mu = m$.
	Then it is easy to verify that $\{(X^m, Y^m, f^m)\}_{m \in \L^0_{\Gc}(\Om, \Pc(I_c \x [0,T]^\R))}$ satisfy Assumption \ref{assum:Yfm}.

	\vspace{0,5em}

	\noindent \textbf{Step 2. Define the set $K$, the map $\Psi$ and verify their properties.}

        Let us define $K$ and $\Psi$ as follows,
        \begin{align*}
             ~
            K           ~ &  \vcentcolon = ~
            \big\{ \Lc \big( (X,(\tau_\ell)_{\ell \in \R}) \big | \Gc \big) :
            X \in \L^0_\Fc(\Om, I_c), \tau_\ell \in \Tc, \ell \in \R \big \},
            \\
            \Psi
            ~ & : \Om \x \D \x \V^+ \longrightarrow I_c \x [0,T]^\R
            \\
            & ~
            (\om,\mathbf{x},\mathbf{l}) \longmapsto (\varphi(\mathbf{x}), (\tau^\mathbf{l}_\ell)_{\ell \in \R}),
        \end{align*}
        where $\tau^\mathbf{l}_\ell \vcentcolon= \inf\{t \geq 0; \mathbf{l}_{t} \geq \ell\}$.
        Since the space $\Tc$ equipped with the almost sure partial order is a complete lattice, when $T < +\infty$, the product space $\Tc^{\R}$ equipped with the product order is still a complete lattice.
        On the other hand, since $I_c$ is a closed interval on $\R$, by the definition of essential supremum and essential infimum, $\L^0_\Fc(\Om, I_c)$ is a complete lattice with almost sure order.
        Then by the order-preserving property of the conditional expectation map $\Lc(\cdot|\Gc)$, one has that $K$ is a complete lattice.

\vspace{0.5em}

        The map $\Psi$ is increasing, for all $\om \in \Om$, in the sense that
        for any $(\mathbf{x}_1,\mathbf{l}_1), (\mathbf{x}_2,\mathbf{l}_2) \in \D \x \V^+$, with
        $\mathbf{l}_1 \leq_v \mathbf{l}_2$ and $\mathbf{x}_1 \leq_\D \mathbf{x}_2$,
        one has that $\tau^{\mathbf{l}_1}_\ell \leq \tau^{\mathbf{l}_2}_\ell$ for all $\ell \in \R$ and $\varphi(\mathbf{x}_1) \le \varphi(\mathbf{x}_2)$.

        \vspace{0.5em}

        Now by Theorem \ref{thm:monotonicity}, one obtains the existence of some $m^* \in \L^0_\Gc(\Om,\Pc(I_c \x [0,T]^\R))$ such that
        $m^* = \Lc(\Psi(X^{m^*}, \Lh^{m^*})|\Gc)$,
        where $\Lh^m$ is the running maximum process defined by $\Lh^m_t \vcentcolon= \sup_{s \in [0,t)}L^m_s$ with the convention that $\sup \emptyset = -\infty$ and $L^m$ is the solution to the representation theorem \eqref{eq:BankEK_repres_intro} of $Y^m$ w.r.t. $f^m$.

        \vspace{0.5em}

	\noindent \textbf{Step 3. Existence of the mean field equilibrium.}

	Finally, we claim that $\Big(\Lc((X^{m^*},  (\tau^{\Lh^{m^*}}_\ell)_{\ell \in \R})|\Gc), (\tau^{\Lh^{m^*}}_\ell)_{\ell \in \R}\Big)$ is a mean field equilibrium.

	\vspace{0.5em}
	
        In fact, by Theorem \ref{thm:BErepresentationTheorem} (iii), \eqref{eq:assumption_timing_order} and
        $$
            m^* = \Lc \big(  \big( \varphi(X^{m^*}), (\tau^{\Lh^{m^*}}_\ell)_{\ell \in \R}) \big| \Gc \big),
        $$
        one has that $\tau^{\Lh^{m^*}}_\ell$ is the smallest optimal stopping time that maximizes the objective functional $J_\ell$ w.r.t. $\Lc((X^{m^*},  (\tau^{\Lh^{m^*}}_\ell)_{\ell \in \R})|\Gc)$,
        i.e. for any $\tau \in \Tc$, it holds that
        $$
            J_\ell\big(\tau^{\Lh^{m^*}}_\ell,\Lc((X^{m^*},  (\tau^{\Lh^{m^*}}_\ell)_{\ell \in \R})|\Gc)\big)
            ~ \geq ~
            J_\ell\big(\tau,\Lc((X^{m^*},  (\tau^{\Lh^{m^*}}_\ell)_{\ell \in \R})|\Gc)\big),
        $$
        and this concludes the proof.
        \qed
	\endproof

	\vspace{0.5em}

	\noindent \textsl{Proof of Proposition \ref{prop:singular_control_application}.}

	\vspace{0.5em}

	\noindent \textbf{Step 1. Define the tuple $(X^m, Y^m, f^m)_m$.}

        For each $m \in \L^0_{\Gc}(\Om, \Pc(\D \x \Pi_{i \in \N}\V^+_{\thetau_i}))$, we define $X^m: [0,T] \x \Om \longrightarrow \R$, $Y^m: [0,T] \x \Om \longrightarrow \R$ and $f^m: [0,T] \x \Om \x \R \longrightarrow \R$ by
        $$
            X^m(t, \om)
            ~ \vcentcolon= ~
            X^{m}(t, \om),
            ~
            Y^m(t, \om)
            ~ \vcentcolon= ~
            k(t, \om, m),
            ~
            f^m(t, \om, \ell)
            ~ \vcentcolon= ~
            c^\prime(t, \om, m),
            ~ \ell \in \R.
        $$
	Then it is easy to verify that $\{(X^m, Y^m, f^m)\}_{m \in \L^0_{\Gc}(\Om, \Pc(\D \x \Pi_{i \in \N}\V^+_{\thetau_i}))}$ satisfy Assumption \ref{assum:Yfm}.

	\vspace{0.5em}

	\noindent \textbf{Step 2. Define the tuple $(E, K, \Psi)$ and verify their properties.}

	We specify the tuple $(E, K, \Psi)$ as follows:
	\begin{align*}
            E
            ~  & \vcentcolon =  ~
            \D \x \Pi_{i \in \N}\V^+_{\thetau_i}
            \\
            ~
            K
            ~  & \vcentcolon =  ~
            \overline{
            conv\{
            \Lc((X^m, (L^i)_{i \in \N})|\Gc):
            m \in \L^0_{\Gc}(\Om, \Pc(\D \x \Pi_{i \in \N}\V^+_{\thetau_i})),
            ~
            L^i \in \Ab_i,
            ~
            i \in \N
            \}},
            \\ ~
            \Psi
            ~ & : \Om \x \D \x \V^+ \longrightarrow \D \x \Pi_{i \in \N}\V^+_{\thetau_i})
            \\
            & ~
            (\om, \mathbf{x}, \mathbf{l}) \longmapsto (\mathbf{x}, (\thetau_i \vee \mathbf{l} \wedge \Thetah_i(\om))_{i \in \N}),
        \end{align*}
        where $K$ is endowed with the topology of weak convergence,
        $K \subset \L^0_\Gc(\Om, \Pc(\D \x \Pi_{i \in \N}\V^+_{\thetau_i}))$, $\D \x \Pi_{i \in \N}\V^+_{\thetau_i}$ is endowed with the product metric.

	\vspace{0.5em}

	The space $K$ is obviously nonempty, convex and closed.
	Then the tightness of
          $\{\Lc(X^m|\Gc): m \in \L^0_{\Gc}(\Om, \Pc(\D \x \Pi_{i \in \N}\V^+_{\thetau_i}))\}$
            and
          $\{\Lc((L^i)_{i \in \N}|\Gc) : L^i \in \Ab_i,~ i \in \N\}$
	implies the compactness of $K$.
	Besides, $K$ is a subset of the Hausdorff locally convex topological vector space $\L^0_\Gc(\Om, \Pc(\D \x \Pi_{i \in \N}\V^+_{\thetau_i}))$ endowed with the topology of convergence in probability.

	\vspace{0.5em}

	Further, the map $\Psi$ is clearly continuous, for all $\om \in \Om$.
	Hence, one can apply Theorem \ref{thm:continuity} to obtain the existence of some $m^* \in \L^0_{\Gc}(\Om,\Pc(\D \x \Pi_{i \in \N}\V^+_{\thetau_i}))$ such that
        $m^* = \Lc(\Psi(X^{m^*},\Lh^{m^*})|\Gc)$,
        where $\Lh^m_t \vcentcolon= \sup_{s \in [0,t)}L^m_s$ with the convention that $\sup \emptyset = -\infty$ and $L^m$ is the solution to the representation theorem \eqref{eq:BankEK_repres_intro} of $Y^m$ w.r.t. $f^m$.

        \vspace{0.5em}

	\noindent \textbf{Step 3. Existence of the mean field equilibrium.}

        Finally, we claim that $(m^*, (\thetau_i \vee \Lh^{m^*} \wedge \Thetah_i)_{i \in \N})$ is a mean field equilibrium.

	\vspace{0.5em}

        Indeed, the consistency condition clearly holds true because
        $$
            m^*
            ~ = ~
            \Lc(\Psi(X^{m^*},\Lh^{m^*})|\Gc)
            ~ = ~
            \Lc((X^{m^*},(\thetau_i \vee \Lh^{m^*} \wedge \Thetah_i)_{i \in \N})|\Gc) .
        $$
	Moreover, Theorem \ref{thm:Bank_singular_control} in Bank \cite{Bank2004} implies the required optimality of $\thetau_i \vee \Lh^{m^*} \wedge \Thetah_i$  in the definition of MFG in Definition \ref{def:MFE_singular_control}.
	Thus, the proof is concluded.
    \qed
    \endproof

    \vspace{0.5em}

	\noindent \textsl{Proof of Proposition \ref{prop:singular_control_application_order}.}

	\vspace{0.5em}

	\noindent\textbf{Step 1. Define the space $E$ and the tuple $(X^m, Y^m, f^m)_m$.}

	\vspace{0.5em}

	Let $E = I_c \x \Pi_{i \in \N}\V^+_{\thetau_i}$.

	\vspace{0.5em}

	Then for each $m \in \L^0_{\Gc}(\Om, \Pc(I_c\x \Pi_{i \in \N}\V^+_{\thetau_i}))$,
	we define $X^m: [0,T] \x \Om \longrightarrow \R$, $Y^m: [0,T] \x \Om \longrightarrow \R$ and $f^m: [0,T] \x \Om \x \R \longrightarrow \R$ by
	$$
            X^m(t, \om)
            ~ \vcentcolon = ~
            X^\mu(t, \om),
            ~
            Y^m(t, \om)
            ~ \vcentcolon = ~
            G(t, \om, \mu),
            ~
            f^m(t, \om, \ell)
            ~ \vcentcolon = ~
            g_\ell(t, \om, \mu),
        $$
        where $\mu \in \L^0_{\Gc}(\Om, \Pc(\D \x \Pi_{i \in \N}\V^+_{\thetau_i}))$ is such that
        $(\varphi \ox \mathrm{Id})\#\mu = m$.
        Then it is easy to verify that $\{(X^m, Y^m, f^m)\}_{m \in \L^0_{\Gc}(\Om, \Pc(I_c \x \Pi_{i \in \N}\V^+_{\thetau_i}))}$ satisfy Assumption \ref{assum:Yfm}.

        \vspace{0.5em}

	\noindent \textbf{Step 2. Define the set $K$, the map $\Psi$ and verify their properties.}

	\vspace{0.5em}

        Now one may specify the couple $(K, \Psi)$ as follows:
        \begin{align*}
            ~
            K
            ~ &  \vcentcolon = ~
            \{
            \Lc((X, (L^i)_{i \in \N})|\Gc):
            X \in \L^0_\Fc(\Om, I_c),
            L^i \in \Ab_i,
            ~
            i \in \N
            \},
            \\ ~
            \Psi
            ~ & : \Om \x \D \x \V^+ \longrightarrow I_c \x \Pi_{i \in \N}\V^+_{\thetau_i}
            \\
            & ~
            (\om, \mathbf{x}, \mathbf{l}) \longmapsto (\varphi(\mathbf{x}),(\thetau_i \vee \mathbf{l} \wedge \Thetah_i(\om))_{i \in \N}).
        \end{align*}

        In Proposition \ref{prop:complete_lattice}, we proved that given two stochastic processes $L^-, L^+ \in \L^0_\Fc(\Om, \V^+)$ such that $L^- \le_l L^+$, the set
        $\{L \in \L^0_\Fc(\Om,\V^+): L^- \le L \le L^+~\mbox{a.s.}\}$ is a complete lattice endowed with the order $\le_{l}$.
        Then the partially ordered set
        $
            \{(L^i)_{i \in \N}: L^i \in \Ab_i, i \in \N\}
        $
        is still a complete lattice endowed with the product order.
        On the other hand, since $I_c$ is a closed interval on $\R$, by the definition of essential supremum and essential infimum, $\L^0_\Fc(\Om, I_c)$ is a complete lattice with almost sure order.
        Then by the order-preserving property of the conditional expectation map $\Lc(\cdot|\Gc)$, we have $K$ is a complete lattice.

        \vspace{0.5em}

        Then by Theorem \ref{thm:monotonicity}, one obtains the existence of some $m^* \in \L^0_\Gc(\Om,\Pc(I_c \x \Pi_{i \in \N}\V^+_{\thetau_i}))$ such that
        $m^* = \Lc(\Psi(X^{m^*}, \Lh^{m^*})|\Gc)$,
        where $\Lh^m$ is the running maximum process of $L^m$, with $L^m$ being the solution to the representation theorem \eqref{eq:BankEK_repres_intro} of $Y^m$ w.r.t. $f^m$.

        \vspace{0.5em}

	\noindent \textbf{Step 3. Existence of the mean field equilibrium.}

        Finally, we claim that
        $\Big(\Lc((X^{m^*}, (\thetau_i \vee \Lh^{m^*} \wedge \Thetah_i)_{i \in \N}|\Gc), (\thetau_i \vee \Lh^{m^*} \wedge \Thetah_i)_{i \in \N}\Big)$
        is a mean field equilibrium.
	In fact, by Theorem \ref{thm:Bank_singular_control} in Bank \cite{Bank2004},  \eqref{eq:assumption_singular_control_order} and
        $$
            m^* = \Lc(( \varphi(X^{m^*}),  (\thetau_i \vee \Lh^{m^*} \wedge \Thetah_i)_{i \in \N})|\Gc),
        $$
        one has that $\thetau_i \vee \Lh^{m^*} \wedge \Thetah_i$ maximize $J(\cdot, \Lc((X^{m^*}, (\thetau_i \vee \Lh^{m^*} \wedge \Thetah_i)_{i \in \N}|\Gc))$ over $\Ab_i$, for $i \in \N$, thus one can conclude our proof.
	\qed
	\endproof

	\vspace{0.5em}

	\noindent \textsl{Proof of Proposition \ref{prop:optimal_consumption_application}.}

	\vspace{0.5em}

	\noindent \textbf{Step 1. Define the tuple $(X^m, Y^m, f^m)_m$.}

	\vspace{0.5em}

        For each $m \in \L^0_{\Gc}(\Om, \Pc(\D \x \D_-))$, we define $X^m: [0,T] \x \Om \longrightarrow \R$, $Y^m: [0,T] \x \Om \longrightarrow \R$ and $f^m: [0,T] \x \Om \x \R \longrightarrow \R$ by
        $$
            X^m(t, \om)
            ~ \vcentcolon= ~
            r^{m}(t, \om),
            ~
            Y^m(t, \om)
            ~ \vcentcolon= ~
            -\lambda e^{-\beta t}e^{-\int_0^t r^m_s ds},
        $$
        \begin{equation*}
            f^m(t, \om, \ell)
            ~ \vcentcolon = ~
            \left\{
                \begin{aligned}
                    & -u^\prime(t, \om, m, \frac{-e^{-\beta t}}{\ell})
                       ~ & ~\mbox{if}~ \ell < 0,
                    \\
                    & \ell & ~\mbox{if}~ \ell \ge 0.
                \end{aligned}
            \right.
        \end{equation*}
        Then it is easy to verify that $\{(X^m, Y^m, f^m)\}_{m \in \L^0_{\Gc}(\Om, \Pc(\D \x \D_-))}$ satisfy Assumption \ref{assum:Yfm}.

        \vspace{0.5em}

	\noindent\textbf{Step 2. Define the space $E$, the set $K$, the map $\Psi$  and verify their properties.}

	\vspace{0.5em}

        Let us define two maps $\psi_1$ and $\psi_2$ as follows:
        \begin{align*}
            \psi_1 : \V^+ & \longrightarrow \V^+
            ~&~
            \psi_2 : \V^+_{-\frac{1}{\eta},-\frac{1}{\etah}} &  \longrightarrow \D_-
            \\
            \mathbf{l} & \longmapsto -\frac{1}{\eta} \vee \mathbf{l} \wedge -\frac{1}{\etah},
            ~&~
            \mathbf{l} & \longmapsto \frac{-e^{-\beta \cdot}}{\mathbf{l}}.
        \end{align*}
        Then one may specify the space $E$, the set $K$ and the map $\Psi$ as follows:
        \begin{align*}
            ~ E
            & =
            \D \x \D_-,
            \\
            ~\Psi
             & : \Om, \D \x \V^+ \longrightarrow \D \x \D_-
            \\
            & ~
            (\om, \mathbf{x}, \mathbf{l}) \longmapsto (\mathbf{x},\psi_2 \circ \psi_1(\mathbf{l})),
            \\
            K
             & \vcentcolon =   ~
            \overline{
            conv\{
            \Lc(\Psi(X^m, L) |\Gc):
            m \in \L^0_{\Gc}(\Om, \Pc(\D \x \D_-)),
            ~
            L \in \L^0_\Fc(\Om, \V^+)
            \}},
        \end{align*}
        where $K \subset \L^0_\Gc(\Om, \Pc(\D \x \D_-))$ is endowed with the topology of weak convergence.

	\vspace{0.5em}

	The space $K$ is obviously nonempty, convex and closed.
        The tightness of $\{\Lc(\psi_1(L) |\Gc): L \in \L^0_\Fc(\Om,\V^+)\}$ is given by Proposition \ref{prop:tightness}.
        Then by the definition of $\Psi$, continuity, injectivity of $\psi_2$, one has the tightness of $\{\Lc(\psi(L) |\Gc) : L \in \L^0_\Fc(\Om,\V^+)\}$.
        Thus, together with the tightness of
        $\{\Lc(X^m |\Gc) : m \in \L^0_{\Gc}(\Om, \Pc(\D \x \D_-))\}$,
        the set $K$ is compact.

	\vspace{0.5em}

	Further, the map $\Psi$ is clearly continuous, for all $\om \in \Om$.
	
	\vspace{0.5em}
	
	Hence, one can apply Theorem \ref{thm:continuity} to obtain the existence of some $m^* \in \L^0_{\Gc}(\Om,\Pc(\D \x \D_-))$ such that
        $m^* = \Lc(\Psi(X^{m^*},\Lh^{m^*})|\Gc)$,
        where $\Lh^m_t \vcentcolon= \sup_{s \in [0,t)}L^m_s$ with the convention that $\sup \emptyset = -\infty$ with $L^m$ being the solution in \eqref{eq:BankEK_repres_intro} of $Y^m$ w.r.t. $f^m$.

        \vspace{0.5em}

	\noindent \textbf{Step 3. Existence of the mean field equilibrium.}

	\vspace{0.5em}

        Finally, we claim that $(\Lc((r^{m^*}, Y^{C^*})|\Gc), C^*)$ is a mean field equilibrium, where $C^*$ is the consumption plan such that
        $Y^{C^*} = e^{-\beta t} \Big( \eta \vee \frac{- 1}{\Lh^{m^*}_t} \Big)$, and for simplicity, $\Lc((r^{m^*}, Y^{C^*})|\Gc)$ is abbreviated as $\mt^*$.
	
        In fact, Theorem \ref{thm:Bank_optimal_consumption} implies the consumption plan $C^*$ is optimal for $U(C, m^*)$ at the cost $b^* = \E[\int_{0}^{+\infty}e^{-\int_0^\cdot r^{m^*}_s ds} d C^*_t]$,
        \eqref{eq:assumption_singular_control_order} implies that
        $U(C,m^*) = U(C, \mt^*)$ and $C^*$ is optimal for $U(C, \mt^*)$ at the cost
        $\E[\int_{0}^{+\infty}e^{-\int_0^\cdot r^{\mt^*}_s ds} d C^*_t]$, thus one can conclude the proof.
	\qed
	\endproof

	\vspace{0.5em}

	\noindent \textsl{Proof of Proposition \ref{prop:optimal_consumption_application_order}.}

	\vspace{0.5em}

	\noindent \textbf{Step 1. Define the space $E$ and the tuple $(X^m, Y^m, f^m)_m$.}

	\vspace{0.5em}

	Let $E \vcentcolon = I_c \x \D_-$ so that $E$ is a Polish space.

	\vspace{0.5em}

	For each $m \in \L^0_{\Gc}(\Om, \Pc(I_c \x \D_-))$, we define $X^m: [0,T] \x \Om \longrightarrow \R$, $Y^m: [0,T] \x \Om \longrightarrow \R$ and $f^m: [0,T] \x \Om \x \R \longrightarrow \R$ by
        $$
            X^m(t, \om)
            ~ \vcentcolon= ~
            r^{\mu}(t, \om),
            ~
            Y^m(t, \om)
            ~ \vcentcolon= ~
            -\lambda e^{-\beta t}e^{-\int_0^t r^\mu_s ds},
        $$
        \begin{equation*}
            f^m(t, \om, \ell)
            ~ \vcentcolon = ~
            \left\{
                \begin{aligned}
                    & -u^\prime(t, \om, \mu, \frac{-e^{-\beta t}}{\ell})
                       ~ & ~\mbox{if}~ \ell < 0,
                    \\
                    & \ell & ~\mbox{if}~ \ell \ge 0,
                \end{aligned}
            \right.
        \end{equation*}
        where $\mu \in \L^0_{\Gc}(\Om, \Pc(\D \x \Pi_{i \in \N}\V^+_{\thetau_i}))$ is such that
        $(\varphi \ox \mathrm{Id})\#\mu = m$.
        Then it is easy to verify that $\{(X^m, Y^m, f^m)\}_{m \in \L^0_{\Gc}(\Om, \Pc(I_c \x \D_-))}$ satisfy Assumption \ref{assum:Yfm}.

	\vspace{0.5em}

	\noindent \textbf{Step 2. Define the set $K$, the map $\Psi$ and verify their properties.}

	\vspace{0.5em}

        We define two maps $\psi_1$ and $\psi_2$ as follows:
        \begin{align*}
            \psi_1 : \V^+ & \longrightarrow \V^+
            ~&~
            \psi_2 : \V^+_{-\frac{1}{\eta},-\frac{1}{\etah}} &  \longrightarrow \D_-
            \\
            \mathbf{l} & \longmapsto -\frac{1}{\eta} \vee \mathbf{l} \wedge -\frac{1}{\etah},
            ~&~
            \mathbf{l} & \longmapsto \frac{-e^{-\beta \cdot}}{\mathbf{l}}.
        \end{align*}
        Then one may specify the space $E$, the set $K$ and the map $\Psi$ as follows:
        \begin{align*}
            ~ E
            & =
            \D \x \D_-,
            \\
            ~\Psi
             & : \Om, \D \x \V^+ \longrightarrow \D \x \D_-
            \\
            & ~
            (\om, \mathbf{x}, \mathbf{l}) \longmapsto (\varphi(\mathbf{x}),\psi_2 \circ \psi_1(\mathbf{l})),
            \\
            K
             & \vcentcolon =   ~
            \{
            \Lc((X, \psi_2 \circ \psi_1(L))|\Gc):
            X \in \L^0_{\Fc}(\Om, I_c),
            ~
            L \in \L^0_\Fc(\Om, \V^+)
            \}.
        \end{align*}

        In Proposition \ref{prop:complete_lattice}, we proved that
        $\L^0 \vcentcolon= \{L \in \L^0_\Fc(\Om,\V^+): L_- \le L \le L^+~\mbox{a.s.}\}$ is a complete lattice endowed with the order $\le_{l}$, for some $L^-, L^+ \in \L^0_\Fc(\Om, \V^+)$.
        As a special case of $\L^0$,
        $
            \{\psi_1(L): L \in \L^0_\Fc(\Om, \V^+)\}
        $
        is a complete lattice.
        Then by the order-preserving property of $\psi_2$, the partially ordered set
        $
            \{\psi(L): L \in \L^0_\Fc(\Om, \V^+)\}
        $
        is still a complete lattice endowed with the almost sure order.
        On the other hand, since $I_c$ is a closed interval on $\R$, by the definition of essential supremum and essential infimum, $\L^0_\Fc(\Om, I_c)$ is a complete lattice with the almost sure order.
        Thus, by the order-preserving property of the conditional expectation map $\Lc(\cdot|\Gc)$, it holds that $K$ is a complete lattice.

	\vspace{0.5em}

        Then by Theorem \ref{thm:monotonicity}, one obtains the existence of some $m^* \in \L^0_{\Gc}(\Om,\Pc(I_c \x \D_-))$ such that
        $m^* = \Lc(\Psi(X^{m^*},\Lh^{m^*})|\Gc)$,
        where $\Lh^m_t \vcentcolon= \sup_{s \in [0,t)}L^m_s$ with the convention that $\sup \emptyset = -\infty$ and $L^m$ being the solution to \eqref{eq:BankEK_repres_intro} of $Y^m$ w.r.t. $f^m$.

	\vspace{0.5em}

	\noindent \textbf{Step 3. Existence of the mean field equilibrium.}

	\vspace{0.5em}

        Finally, we claim that $(\Lc((r^{m^*}, Y^{C^*})|\Gc), C^*)$ is a mean field equilibrium, where $C^*$ is the consumption plan such that
        $Y^{C^*} = e^{-\beta t} \Big( \eta \vee \frac{- 1}{\Lh^{m^*}_t} \Big)$, and for simplicity, $\Lc((r^{m^*}, Y^{C^*})|\Gc)$ is abbreviated as $\mt^*$.
	
        In fact, by Theorem \ref{thm:Bank_optimal_consumption} and \eqref{eq:consumption_order_equivalence},
        one has that the consumption plan $C^*$ is optimal for the utility $U(C, \Lc(r^{m^*}, \psi_1(Y^{C^*}))|\Gc))$ at the cost $b^* = \E[\int_{0}^{+\infty}e^{-\int_0^\cdot r_s ds} d C^*_t]$, and
        \eqref{eq:consumption_order_restriction} implies that
        $U(C,\Lc((r^{m^*}, \psi_1(Y^{C^*}))|\Gc)) = U(C, \mt^*)$ and $C^*$ is optimal for $U(C, \mt^*)$ at the cost
        $b^*$, thus one can conclude the proof.
    \qed
    \endproof

    \vspace{0.5em}

    \noindent \textsl{Proof of Proposition \ref{prop:optimal_consumption_contraction}.}

	\vspace{0.5em}

	\noindent \textbf{Step 1. Define the tuple $(Y^\mu, f^\mu)_\mu$.}

	\vspace{0.5em}

	There exists a constant $\lambda \in \R$, such that, for one $\mu \in \Pc(\V^+_{-\frac{1}{\eta}})$ (and hence for all $\mu \in \Pc(\V^+_{-\frac{1}{\eta}})$),
	with $Y: [0,T] \x \Om \longrightarrow \R$ and $f^\mu: [0,T] \x \Om \x \R \longrightarrow \R$ defined by
        $$
            Y(t, \om)
            ~ \vcentcolon= ~
            -\lambda e^{-\beta t}e^{-\int_0^t r_s ds},
        $$
        \begin{equation*}
            f^\mu(t, \om, \ell)
            ~ \vcentcolon = ~
            \left\{
                \begin{aligned}
                    & -\ut^\prime \Big(t, \om, \frac{-e^{-\beta t}}{\ell - \mu_t(\phi)} \Big)
                       ~ & ~\mbox{if}~ \ell < 0,
                    \\
                    & \ell & ~\mbox{if}~ \ell \ge 0,
                \end{aligned}
            \right.
        \end{equation*}
        the corresponding optional process $L^{\mu}$ in  the representation \eqref{eq:BankEK_repres_intro} is almost surely non-decreasing.
        Then it is easy to verify that $\{(Y^\mu, f^\mu)\}_{\mu \in \Pc(\D_-)}$ satisfy Assumption \ref{assum:Yfm}.

        \vspace{0.5em}

	\noindent \textbf{Step 2. Define the function $\Psi$ and verify its property.}

	\vspace{0.5em}

        We define the map $\Psi$ as the following:
        \begin{align*}
            \Psi : &  \Om \x \V^+ \longrightarrow \V^+_{-\frac{1}{\eta}}
            \\
            & ~ (\om, \mathbf{l}) \longmapsto \frac{-1}{\eta} \vee \mathbf{l}
        \end{align*}
        Then by Theorem \ref{thm:contraction}, one obtains some $\mu^* \in \Pc(\V^+_{-\frac{1}{\eta}})$ such that
        $\mu^* = \P \circ (\frac{-1}{\eta} \vee \Lh^{\mu^*})^{-1}$,
        where $\Lh^{\mu^*}_t \vcentcolon= \sup_{s \in [0,t)}L^{\mu^*}_s$ with the convention that $\sup \emptyset = -\infty$ and $L^{\mu^*}$ is solution to the representation theorem \eqref{eq:BankEK_repres_intro} of $Y^{\mu^*}$ with generator $f^{\mu^*}$.

	\vspace{0.5em}

	\noindent \textbf{Step 3. Existence of the mean field equilibrium.}

	\vspace{0.5em}

        Finally, we claim that $(\P \circ (e^{-\beta \cdot}(\eta \vee \frac{-1}{ \Lh^{\mu^*_\cdot}})^{-1}, C^* ) )$ is a mean field equilibrium, where $C^*$ is the consumption plan such that
        $Y^{C^*} = e^{-\beta \cdot}(\eta \vee \frac{-1}{ \Lh^{\mu^*_\cdot}})$.
	Indeed, by Theorem \ref{thm:Bank_optimal_consumption}, one has that the consumption plan $C^*$ is optimal at the cost $\E[\int_{0}^{+\infty}e^{-\int_0^t r_s ds} d C^*_t]$.
    \qed
    \endproof

\begin{appendix}
\section{Bank-El Karoui's representation theorem}

	We recall here Bank-El Karoui's representation theorem as well as some first properties of the representation processes.

	\begin{definition} \label{def:USCE}
		An optional process $Y$ of class $(D)$ is said to be upper-semicontinuous in expectation if, for any $\tau \in \Tc$ and any sequence $(\tau_n)_{n \ge 1} \subset \Tc$ satisfying either
		$$
			(\tau_n)_{n\ge 1} ~\mbox{is non-decreasing,}~
    			\tau_n < \tau ~\mbox{on}~\{\tau > 0\},~\mbox{for each}~n\ge 1,
    			~\lim_{n \to \infty} \tau_n = \tau,
    		$$
    		or
    		$$
                \tau_n \ge \tau,~\mbox{a.s., for each}~ n \ge 1,~\lim_{n \to \infty} \tau_n = \tau,
    		$$
    		one has
    		$$
    			\E[Y_{\tau}] ~\geq~ \limsup_{n\to \infty}\E[Y_{\tau_n}].
    		$$
    	\end{definition}

	We are given a function $f: [0,T] \x \Om \x \R \longrightarrow \R$, which satisfies the following condition.

	\begin{assumption}\label{Assumption:BErepresentation}
		For all $(t, \om) \in [0,T] \x \Om$, the map $\ell \longmapsto f(t,\om, \ell)$ is continuous and strictly increasing from $-\infty$ to $+\infty$.
		Moreover, for each $\ell \in \R$, the process $(t, \om) \longmapsto f(t, \om, \ell)$ is progressively measurable, and satisfies
		$$
			\E \bigg[\int_{0}^{T} \big| f(t, \ell) \big| dt \bigg] < \infty.
		$$
	\end{assumption}

	\begin{theorem}[Bank-El Karoui's representation, \cite{BankKaroui2004, BankFollmer}] \label{thm:BErepresentationTheorem}
		Let  $f: [0,T] \x \Om \x \R \longrightarrow \R$ satisfy Assumption \ref{Assumption:BErepresentation}.
		Then for every optional process $Y$ of class (D) and u.s.c. in expectation such that $Y_T = 0$,
		the following holds true.

    		\vspace{0.5em}	
    		
		\noindent $\mathrm{(i)}$ There exists an optional process $L : [0,T] \x \Om \longrightarrow \R$, such that
		\begin{equation*}
			\E\bigg[\int_{\tau}^{T} \Big|f\Big( t, \sup_{v\in [\tau,t)}L_v\Big)\Big| dt \bigg] < \infty,
    			~~\mbox{for all}~ \tau \in \Tc,
    		\end{equation*}
		and
		\begin{equation}\label{eq:BErepresentation}
			Y_\tau = \E\bigg[\int_{\tau}^{T} f\Big(t, \sup_{s \in [\tau,t)}L_s \Big) dt \Big| \Fc_\tau\bigg],
			~\mbox{a.s., for all}~
			\tau \in \Tc.
		\end{equation}
    		
		\noindent $\mathrm{(ii)}$
		Let $L$ be a solution to the representation \eqref{eq:BErepresentation}, which is progressively measurable and upper-right continuous.
		Then
		\begin{equation}
			L_\tau ~ = ~
			\essinf_{\sigma \in \Tc_{\tau, +}} \ell_{\tau,\sigma},
		\end{equation}
		where $\Tc_{\tau,+} = \{\sigma \in \Tc_\tau: \sigma > \tau ~\mbox{on}~ \{\tau < T\}\}$, $\ell_{\tau,\sigma}$ is the unique $\Fc_\tau$-measurable random variable such that
		\begin{equation*}
			\E\bigg[\int_{\tau}^{\sigma}
			f(t, \ell_{\tau,\sigma})
			dt\Big| \Fc_\tau\bigg]
			~ = ~
			\E[Y_\tau - Y_\sigma | \Fc_\tau].
		\end{equation*}
		In particular, such an upper-right continuous solution $L$ is unique up to optional section.
		
		\vspace{0.5em}

		\noindent $\mathrm{(iii)}$ Let the optional process $L$ be a solution to the representation \eqref{eq:BErepresentation}.
		Then for any $\ell \in \R$,
		the stopping times
		\begin{equation*}
			\tau_{\ell}
            ~\vcentcolon=~
            \inf \big\{ t\geq 0 ~: L_{t-} \geq \ell \big\},
			~
            \tauh_{\ell}
            ~\vcentcolon=~
            \inf \big\{ t\geq 0 ~: L_{t-} > \ell \big\}
		\end{equation*}
		are the smallest solution and the biggest solution of the optimal stopping problem:
		\begin{equation}\label{OptimalStopping:BErepresentation}
			\sup_{\tau \in \Tc}
			~\E\bigg[Y_\tau + \int_{0}^{\tau}h(t, \ell) dt \bigg].
		\end{equation}
    	\end{theorem}

	\begin{remark}
		In the deterministic setting with $\Fc = \{\Om, \emptyset\}$, so that $f(\cdot, \ell)$ and $Y$ can be seen as deterministic functions on $[0,T]$.
		Let $f(\cdot, \ell) = \ell$ for all $\ell \in \R$.
		Then given $Y: [0,T] \longrightarrow \R$, with the representation process $L: [0,T) \longrightarrow \R$,
		the function $\int_{\cdot}^{T} \sup_{r \in [t,s)}L_r ds$ turns to be the convex envelope of $- Y|_{[t,T]}$,
		and $L_t$ is the derivative of the convex envelope of $- Y|_{[t,T]}$ at time $t$.

 	\end{remark}

\section{Fixed-point theorems and partially ordered Polish space}
\label{sec:fixed_point}

\subsection{Some fixed-point theorems}

	We recall in this section three fixed-point theorems which are used in the paper.
	It seems that most of the fixed-point theorems in the literature can be considered as variations or extensions of them.
	
	\vspace{0.5em}
	

	Let us first present Schauder fixed-point theorem, which is frequently used in the fields of differential equations and game theory.
	
	\begin{theorem}[Schauder fixed-point theorem, Theorem 3.2 \cite{Bonsall1962}]\label{thm:Schauder}
		Let $V$ be a Hausdorff locally convex topological vector space and $K \subset V$ be a nonempty convex closed subset.
		Let $T: K \longrightarrow K$ be continuous, and such that $T(K)$ is precompact.
		Then $T$ has a fixed point.
	\end{theorem}

	The second one is Tarski's fixed point theorem, under some order structure condition.

	\begin{definition}
		$\mathrm{(i)}$
		Let $L$ be a set, the relation $\leq$ is called a partial order if, for all $\mathbf{l} _1, \mathbf{l} _2, \mathbf{l} _3 \in L$, one has
		$$
			\mathbf{l} _1 \leq \mathbf{l} _1,
			~~~~~
			\mathbf{l} _1 \leq \mathbf{l} _2,~ \mathbf{l} _2 \leq \mathbf{l} _3
			\Longrightarrow
			\mathbf{l} _1 \leq \mathbf{l} _3,
			~~~~~
			\mathbf{l} _1 \leq \mathbf{l} _2, ~\mathbf{l} _2 \leq \mathbf{l} _1
			\Longrightarrow
			\mathbf{l} _1 = \mathbf{l} _2.
		$$
		\noindent $\mathrm{(ii)}$
		A partially ordered set $L$ is called a complete lattice if,
		for any subset $L_0 \subset L$, there exist some $\mathbf{l}^-, \mathbf{l}^+ \in L$, such that
		\begin{itemize}
			\item $\mathbf{l}^- \le \mathbf{l} \le \mathbf{l}^+$, for all $\mathbf{l} \in L_0$,
			
			\item if there exists another pair $\lt^-, \lt^+ \in L$, such that $\lt^- \le \mathbf{l} \le \lt^+$ for all $\mathbf{l} \in L_0$,
			then $\lt^- \le \mathbf{l}^-$ and $\mathbf{l}^+ \le \lt^+$.
		\end{itemize}

    \noindent $\mathrm{(iii)}$
        Suppose that $f$ is a function from a partially ordered set $L_1$ to a partially ordered set $L_2$, we say it is order preserving (order reversing) if for any $\mathbf{l} _1, \mathbf{l} _2 \in L_1$ with $\mathbf{l} _1 \leq \mathbf{l} _2$, $f(\mathbf{l} _1) \leq f(\mathbf{l} _2), (f(\mathbf{l} _2) \leq f(\mathbf{l} _1))$.
    \end{definition}

	The Tarski's fixed point theorem gives an existence result, which is based on the iteration method.
	\begin{theorem}[Tarski's fixed point theorem]\label{thm:Tarski}
		Let $(L,\leq)$ be a complete lattice and let $T : L \to L$ be a monotonic function w.r.t. $\leq$.
		Then the set of all fixed points of $T$ is also a (nonempty) complete lattice under $\leq$.
		Moreover, there exists a unique least fixed point and a unique largest fixed point.
	\end{theorem}

\subsection{Partially ordered Polish space}

	We recall the notion of the partially ordered Polish space,
	and present a result from Kamae and Krengel \cite{KamaeKrengel1978},
	that the space of all probability measures on a partially ordered Polish space is still a partially ordered Polish space.

	\begin{definition}\label{def:pos}
		A partially ordered Polish space is a Polish space $X$ equipped with a partial order $\le$ such that the graph set $\{(x,y) \in X^2: x \le y\}$ is a closed subset of $X^2$.
	\end{definition}

	Let $(E,\le_e)$ be a partially ordered Polish space, $\Pc(E)$ be the space of all (Borel) probability measures on $E$ equipped with the weak convergence topology.
	A function $f: E \longrightarrow \R$ is said to be increasing if $f(e_1) \le f(e_2)$ for all $e_1 \le_e e_2$.
	Then based on the partial order $\le_e$ on $E$, we introduce a partial order $\le_p$ on $\Pc(E)$ as follows:
	Let $\Gamma$ denote the space of all bounded increasing measurable function $f: \E \longrightarrow \R$, for $\P, \Q \in \Pc(E)$, we say\
	$$
		\P \le_p \Q,
		~~\mbox{if and only if}~~
		\int_E f(x) \P(dx) \le  \int_E f(x) \Q(dx), ~\mbox{for all}~f \in \Gamma.
	$$

	\begin{theorem}[Kamae and Krengel \cite{KamaeKrengel1978}]\label{thm:p.o.PolishSpace}
		Let $(E, \le_e)$ be a partially ordered Polish space. Then $(\Pc(E), \le_p)$ is also a partially ordered Polish space.
	\end{theorem}

	\begin{example}\label{exam:p.o.PolishSpace}
		On the Polish spaces $\V^+$, $\V^+_\eta:= \{\mathbf{l} \vee \eta: \mathbf{l} \in \V^+\}$ and $\D$, where $\eta \in \R$ is a constant, let us introduce the following partial order $\le_e$:
		$$
			\mathbf{x}^1 \le_e \mathbf{x}^2
			~~\mbox{if and only if}~~
			\mathbf{x}^1_t \le \mathbf{x}^2_t,~\mbox{for all}~t \in [0,T] ~(\mbox{or}~[0,T)).
		$$
		Then all of $(\V^+, \le_e)$, $(\V^+_\eta, \le_e)$ and $(\D, \le_e)$ are partially ordered Polish spaces.
	\end{example}

\section{On the space $\V^+$}\label{sec:spaceV}

	It is wellknown that the space of all left-continuous and increasing $\R$-valued paths on a closed interval is a Polish space under the L\'evy metric.
	Here, our space $\V^+$ is defined as the space of all increasing and left-continuous paths $\mathbf{l}: [0,T) \longrightarrow \R \cup\{-\infty\}$,
	such that $\mathbf{l}(0) = -\infty$ and $\mathbf{l}$ is $\R$-valued on $(0,T)$.
	For completeness, we show that it is also a Polish space under the L\'evy metric $d_L$.
	Recall that the L\'evy metric $d_L$ defined by
	\begin{equation*}
		d_L(\mathbf{l}_1, \mathbf{l}_2)
		\vcentcolon=
		\inf \big\{
			\eps \geq 0 ~:
			\mathbf{l}_1\big((t - \eps)\vee 0\big) - \eps \leq \mathbf{l}_2(t),
			~\mathbf{l}_2\big((t - \eps)\vee 0\big) - \eps \leq \mathbf{l}_1(t),
			~\forall t \in (0,T)
		\big\},
	\end{equation*}
	when $T < +\infty$,
	and
	\begin{equation*}
		d_L(\mathbf{l}_1,\mathbf{l}_2)
		~ \vcentcolon= ~
		\sum_{n = 1}^{+\infty}
		2^{-n}(d_L(\mathbf{l}_1|_{[0,n)},\mathbf{l}_2|_{[0,n)}) \wedge 1),
	\end{equation*}
	when $T = +\infty$.

	\begin{proposition}\label{prop:ac_[0,T]Polish}
		The space $(\V^+, d_L)$  is a complete separable metric space.
		Moreover,
		for $\mathbf{l}_n, \mathbf{l} \in \V^+$, $n \in \N$, the sequence $\{\mathbf{l}_n\}_{n = 1}^{+\infty}$ converges to $\mathbf{l}$ in $d_L$ if and only if $\mathbf{l}_n$ converges point-wise to $\mathbf{l}$ on continuity points of $\mathbf{l}$ on $(0,T)$.
	\end{proposition}

	\begin{proof}
	$\mathrm{(i)}$ Let us first consider the case where $T < +\infty$.
	
	\vspace{0.5em}

	\noindent \textbf{Step 1. $d_L$ is a metric on $\V^+$.}

	\vspace{0.5em}

	First, for any $\mathbf{l}_1, \mathbf{l}_2 \in \V^+$,
	as the set
	$$
		\{\eps \geq 0 ~:
			\mathbf{l}_1\big((t - \eps)\vee 0\big) - \eps \leq \mathbf{l}_2 (t),
			~
			\mathbf{l}_2\big((t - \eps)\vee 0\big) - \eps \leq \mathbf{l}_1 (t),
			~\forall~ t \in (0,T)
		\}
		\neq
		\emptyset,
	$$
	it follows that $d_L(\mathbf{l}_1, \mathbf{l}_2) \ge 0$.

	\vspace{0.5em}

	Next, by its definition, it is easy to check that $d_L(\mathbf{l}_1, \mathbf{l}_2) = 0$ if and only if $\mathbf{l}_1 = \mathbf{l}_2$, and that $d_L(\mathbf{l}_1, \mathbf{l}_2) = d_L(\mathbf{l}_2, \mathbf{l}_1)$.

	\vspace{0.5em}

	For the triangle inequality, let $\mathbf{l}_1, \mathbf{l}_2, \mathbf{l}_3 \in \V^+$,
	there exist two sequences of nonnegative constants $(\eps^i_n)_{n \ge 1}$ with $\lim_{n \to \infty}\eps^i_n = d_L(\mathbf{l}_i,\mathbf{l}_{i + 1})$ such that,
	for each $n \ge 1$,
        \begin{equation*}
            \mathbf{l}_i\big((t - \eps^i_n)\vee 0\big) - \eps^i_n \leq \mathbf{l}_{i + 1}(t),
            ~
            \mathbf{l}_{i + 1}\big((t - \eps^i_n)\vee 0\big) - \eps^i_n \leq \mathbf{l}_i(t),
                ~\mbox{for all}~ t \in (0,T),
                ~ i = 1, 2.
        \end{equation*}
	Then
	\begin{align*}
		&
		\mathbf{l}_1\big((t - \eps^1_n - \eps^2_n)\vee 0\big) - \eps^1_n - \eps^2_n
		~=~
		\mathbf{l}_1\big(((t - \eps^2_n) - \eps^1_n)\vee 0\big) - \eps^1_n - \eps^2_n
		\\
		\leq &~
		\mathbf{l}_2\big((t - \eps^2_n)\vee 0\big) - \eps^2_n
		\\
		\leq&~
		\mathbf{l}_3(t),
		~~~~\mbox{for all}~ t \in (0,T).
	\end{align*}
	Similarly, one has
	\begin{align*}
		\mathbf{l}_3\big((t - \eps^1_n - \eps^2_n)\vee 0\big) - \eps^1_n - \eps^2_n
		\leq \mathbf{l}_1(t),
		~\mbox{for all}~ t \in (0,T).
	\end{align*}
	By the definition of $d_L$, it follows that, for all $n \ge 1$,  $d_L(\mathbf{l}_1,\mathbf{l}_3) \leq \eps^1_n + \eps^2_n$.
	Let $n$ tends to $\infty$, one obtains that
	$$
		d_L(\mathbf{l}_1,\mathbf{l}_3) ~\leq~ d_L(\mathbf{l}_1,\mathbf{l}_2) + d_L(\mathbf{l}_2,\mathbf{l}_3).
	$$

	\noindent \textbf{Step 2.}
	Let $\mathbf{l} \in \V^+$, we denote by $\mathbb{T}_\mathbf{l}$ the set of all continuity points $t \in [0,T)$ of $\mathbf{l}$.
	We next proe that
	$$
		\lim_{n \to \infty}d_L(\mathbf{l}_n,\mathbf{l}) = 0
		~~\Longrightarrow~~
		\lim_{n \to \infty}\mathbf{l}_n(t) = \mathbf{l}(t),
		~\mbox{for all}~
		t \in \mathbb{T}_\mathbf{l}.
	$$

	\vspace{0.5em}

	First, for the sufficient condition part $(\Longrightarrow)$,
	let $\lim_{n \to \infty}d_L(\mathbf{l}_n,\mathbf{l}) = 0$, then there exists a sequence of constants $\eps_n$ with $\lim_{n \to \infty}\eps_n = 0$ such that
        \begin{equation*}
            \mathbf{l}_n\big((t - \eps_n)\vee 0\big) - \eps_n \leq \mathbf{l}(t),
            ~
            \mathbf{l}\big((t - \eps_n)\vee 0\big) - \eps_n \leq \mathbf{l}_n(t),
                ~\mbox{for all}~ t \in (0,T).
        \end{equation*}
	For fixed $t$, we can find $N$ large enough such that for any $n \geq N$, $t + \eps_n < T$, then we can rewrite the above inequalities as
        \begin{equation*}
            \mathbf{l}\big((t - \eps_n)\vee 0\big) - \eps_n
            ~
            \leq
            ~
            \mathbf{l}_n(t)
            ~
            \leq
            ~
            \mathbf{l}_n(t\vee \eps_n)
            ~
            \leq
            ~
            \mathbf{l}\big(t + \eps_n\big) + \eps_n,
                ~\forall~ t \in (0,T).
        \end{equation*}
	When $t \in \mathbb{T}_\mathbf{l}$,  letting $n \longrightarrow \infty$, it follows that
        \begin{equation*}
            \mathbf{l}(t)
            ~ \leq
            \liminf_{n \to \infty} \mathbf{l}_n(t)
            ~ \leq
            \limsup_{n \to \infty} \mathbf{l}_n(t)
            ~ \leq
            \mathbf{l}(t).
        \end{equation*}

	Next, for the necessary condition part $(\Longleftarrow)$, we fix an $\eps > 0$, then there exists a finite partition $\pi_\eps = \{t_1, \cdots, t_{K_\eps}\}$ of $[0,T)$ with $0 = t_0 < t_1 < \cdots < t_{K_\eps} < T =: t_{K_\eps + 1}$ and its norm $|\pi_\eps| \vcentcolon= \sup_{0 \leq i \leq K_\eps}|t_{i + 1} - t_i| < \eps.$

	\vspace{0.5em}

	Let $N = N(\eps, t_1, \cdots, t_{K_\eps})$ be large enough such that for any $n \ge 1$,
        \begin{equation*}
            |\mathbf{l}_n(t_i) - \mathbf{l}(t_i)|
            ~ \leq ~
            \eps,
            ~\mbox{for all}~
            i = 1, \cdots, K_\eps.
        \end{equation*}
        Therefore, for $t \in [t_i, t_{i + 1}]$, $i = 0, \cdots, K_\eps$, one has the estimation
        \begin{align*}
            & ~
            \mathbf{l}_n(t)
            ~ \geq ~
            \mathbf{l}_n(t_i)
            ~ \geq ~
            \mathbf{l}(t_i) - \eps
            ~ \geq ~
            \mathbf{l}\big((t - \eps) \vee 0\big) - \eps,
            \\ & ~
            \mathbf{l}(t)
            ~ \geq ~
            \mathbf{l}(t_i)
            ~ \geq ~
            \mathbf{l}_n(t_i) - \eps
            ~ \geq ~
            \mathbf{l}_n\big((t - \eps) \vee 0\big) - \eps,
        \end{align*}
        which implies that $d_L(\mathbf{l}_n,\mathbf{l}) \leq \eps$.

	\vspace{0.5em}

	\noindent \textbf{Step 3. $(\V^+,d_L)$ is complete and separable.}
	
	\vspace{0.5em}

	Let  $(\mathbf{l}_m)_{m = 1}^\infty$ be a Cauchy sequence in  $\V^+$, we aim to find some $\mathbf{l} \in \V^+$ such that
	$$
		\lim_{m \to \infty}d_L(\mathbf{l}_m,\mathbf{l}) = 0.
	$$

	We first define a sequence of maps $I_n : \V^+ \longrightarrow \V^+$ for each $n \ge 1$ by
	$$
		I_n(\mathbf{l})(0) := \mathbf{l}(0)= -\infty,
		~~\mbox{and}~
		I_n(\mathbf{l})(t) := -n \vee \mathbf{l}(t) \wedge n,
		~~\mbox{for all}~ t \in (0,T).
	$$
	Let $\V^+_{n} := I_n(\V^+)$, so that $(\V^+_n, d_L)$ is a compact Polish space.
	For each $n \ge 1$,
	we notice that $d_L(I_n(\mathbf{l}), I_n(w)) \leq d_L(\mathbf{l}, w)$, for all $\mathbf{l}, w \in \V^+$ and $n \ge 1$, so that  $\big(I_n(\mathbf{l}_m) \big)_{m = 1}^\infty$ is alaso a Cauchy sequence in the compact Polish space $\V^+_n$,
	then there exists some $\mathbf{l}^n \in \V^+_{n}$ such that $I_n(\mathbf{l}_m) \longrightarrow \mathbf{l}^n$ in $\V^+_{n}$.
	In particular, one has $I_n(\mathbf{l}^{n + 1}) = \mathbf{l}^n$.

	\vspace{0.5em}

	For any $\eps > 0$, choose $m = m(\eps, \mathbf{l}^n, \mathbf{l}^{n + 1})$ large enough such that $d_L\big(\mathbf{l}^n, I_n(\mathbf{l}_m)\big) \leq \eps$ and $d_L\big(\mathbf{l}^{n + 1}, I_{n + 1}(\mathbf{l}_m)\big) \leq \eps$.
	Then we obtain the estimation
	\begin{align*}
		&
		d_L\big(\mathbf{l}^n, I_n(\mathbf{l}^{n + 1})\big)
		~ \leq ~
		d_L\big(\mathbf{l}^n, I_n(\mathbf{l}_m)\big)
		+ d_L\big(I_n(\mathbf{l}_m), I_n(\mathbf{l}^{n + 1})\big)
		\\
		\leq ~ &
		d_L\big(\mathbf{l}^n, I_n(\mathbf{l}_m)\big)
                + d_L\big(I_{n+1}(\mathbf{l}_m), \mathbf{l}^{n + 1}\big)
		~\leq ~
		2\eps.
        \end{align*}
	Now, let us define $\tilde{\mathbf{l}}$ by
	$$
		\tilde{\mathbf{l}}_0 := -\infty,
		~~\mbox{and}~
		\tilde{\mathbf{l}}(t) \vcentcolon= \lim_{n \to \infty}\mathbf{l}^n(t),
		~\mbox{for all}~
		t \in [0,T),
	$$
	and let $\mathbf{l}$ be the left-continuous version of $\tilde{\mathbf{l}}$, then $\mathbf{l}$ is obviously increasing as the limit of increasing function on $[0,T)$.

	\vspace{0.5em}

	We next check that $\mathbf{l}$ is $\R$-valued on $(0,T)$ to show that $\mathbf{l} \in \V^+$.
	Assume that $\mathbf{l}$ is not $\R$-valued,
	then one has, for some $s \in (0,T)$, $\mathbf{l}(t) = +\infty$ and $\mathbf{l}^n(t) = n$, for all $t \in [s, T)$.
        Let $\eps_0 \vcentcolon= (T - s)/5$ and fix $s_0 \vcentcolon= (s + T)/2$, then there exists some $N = N(\eps_0)$, such that for any $l,m \geq N$, $d_L(\mathbf{l}_l,\mathbf{l}_m) \leq \eps_0$.
        Thus we have for any $n \in \N$, $d_L(I_n(\mathbf{l}_l), I_n(\mathbf{l}_m)) \leq \eps_0$, and $d_L(I_n(\mathbf{l}_l), \mathbf{l}^n) \leq \eps_0$. Now we focus on $\mathbf{l}_N$ and choose $n$ large enough such that $\mathbf{l}^n(s_0) = n > \mathbf{l}_N(s_0 + 2\eps_0) + 2\eps_0$, which contradicts with the inequality $d_L(I_n(\mathbf{l}_l), \mathbf{l}^n) \leq \eps_0$.
        Therefore, we conclude that $\mathbf{l} \in \V^+$.

  	\vspace{0.5em}

	Finally, for any continuity point $t \in \mathbb{T}_\mathbf{l}$ of $\mathbf{l} \in \V^+_n$,
	one has $\mathbf{l}(t) = \mathbf{l}^n(t)$ and $t$ is also a continuity point of $\mathbf{l}^n$,
	hence $\mathbf{l}(t) = \mathbf{l}^n(t) = \lim_{m \to \infty}I_n(\mathbf{l}_m(t)) = \lim_{m \to \infty}\mathbf{l}_m(t)$.
	This implies that $\mathbf{l}_m$ converges to $\mathbf{l}$ in $\V^+$ as $m$ tends to $+\infty$, and hence $(\V^+, d_L)$ is complete.
	To prove it is also separable, one can apply similar arguments in {\bf Step 2}.

	\vspace{0.5em}

	\noindent $\mathrm{(ii)}$ For the case $T = \infty$, it is obvious that $d_L$ here is still a metric.
	
	\vspace{0.5em}
	
	Moreover, for any sequence $\{\mathbf{l}_i\}_{i = 0}^
        \infty \subset \V^+$, we have
        $\lim_{i \to \infty}d_L(\mathbf{l}_i,\mathbf{l}) = 0$ if and only if
        $\lim_{i \to \infty}d_L(\mathbf{l}_i|_{[0,n)},\mathbf{l}|_{[0,n)}) = 0$ for all $n \in \N_+$
        and
        we can conclude that $\lim_{i \to \infty}d_L(\mathbf{l}_i,\mathbf{l}) = 0$
        if and only if
        $\lim_{i \to \infty}\mathbf{l}_i(t) = \mathbf{l}(t)$ for all continuity points of $\mathbf{l}$ by Step 2, .

	\vspace{0.5em}

	The space $(\V^+, d_L)$ is clearly separable,
	we finally prove the completeness of $(\V^+, d_L)$.
	Let $(\mathbf{l}_m)_{m = 1}^\infty$ be a Cauchy sequence in  $\V^+$.
        Then there exists some $\mathbf{l}^n \in \V^+$ such that
	$$
		\lim_{m \to \infty}d_L(\mathbf{l}_m|_{[0,n)},\mathbf{l}^n) = 0,
		~~\mbox{for all}~
		n \ge 1.
	$$
        We now verify that for all $n \in \N$,
        $\mathbf{l}^{n + 1}|_{[0,n)} = \mathbf{l}^{n}$.
        Since $\lim_{m \to \infty}d_L(\mathbf{l}_m|_{[0,n + 1)},\mathbf{l}^{n + 1}) = 0$, we have $\lim_{m \to \infty}d_L(\mathbf{l}_m|_{[0,n)},\mathbf{l}^{n + 1}|_{[0,n)}) = 0$
        and uniqueness of the limit of the Cauchy sequence $\{\mathbf{l}_m|_{[0,n)}\}_{m = 1}^{+\infty}$ implies our claim.
        Then $\mathbf{l} \vcentcolon= \sum_{n = 1}^{+\infty}\mathbf{l}^n \mathds{1}_{[n - 1,n)}$ is a well-defined increasing function on $\R$ and belongs to $\V^+$.
    \end{proof}

\section{Enlarged space $\Om \x [0,T]$ and the stable convergence topology}
\label{sec:stable_cvg}

	Given the (abstract) filtered probability space $(\Om, \Fc,  \P)$, we introduce an enlarged measurable space $(\overline{\Om}, \Fcb)  := (\Om \x [0,T], \Fc \ox \Bc([0,T]))$.
	Let $\Pc(\overline{\Om})$ denote the collection of all probability measures on $(\Omb, \Fcb)$, we equip it with the stable convergence topology of Jacod and M\'emin \cite{JacodMemin1981}, and recall some basic facts here.
	
	\vspace{0.5em}

	Let $\Pc(\Omb)$  (resp. $\Pc(\Om)$, $\Pc([0,T])$) denote the collection of all probability measures on $(\overline{\Om}, \Fcb)$ (resp. $(\Om, \Fc)$, $([0,T], \Bc([0,T]))$).
	Further, let $B_{mc}(\overline{\Om})$ (resp. $B_{mu}(\overline{\Om})$) denote the collection of all bounded $\Fcb$-measurable functions $\xi: \Omb \longrightarrow \R$ such that for every $\om \in \Om$, the mapping $\theta \in [0,T] \longmapsto \xi(\om, \theta)$ is continuous (resp. upper semi-continuous).
	The stable convergence on $\Pc(\Omb)$ is the coarsest topology making $\P \longmapsto \E^{\P}[\xi]$ continuous for all $\xi \in B_{mc}$.
	We also equip $\Pc([0,T])$ with the weak convergence topology and $\Pc(\Om)$ with the coarsest topology such that $\P \mapsto \E^\P[\xi]$ is continuous for all bounded measurable functions on $(\Om, \Fc)$.
	Let us first recall a result from Jacod and M\'emin \cite{JacodMemin1981}.

	\begin{theorem}\label{thm:stable_convergence}
		$\mathrm{(i)}$ Let $T < \infty$, a subset $\overline{\Pc} \subset \Pc(\overline{\Om})$ is relatively compact w.r.t. the stable topology
		if and only if
		$\{\Po|_{\Om}:\Po \in \overline{\Pc}\}$ is relatively compact in $\Pc(\Om)$.

		\vspace{0.5em}

		\noindent $\mathrm{(ii)}$
		Let $(\Po_n)_{n \in \N} \subset \Pc(\overline{\Om})$ be a sequence such that $\Po_n$ converges to some $\Po \in \Pc(\overline{\Om})$ under the stable topology,
		then one has that
		\begin{equation*}
			\lim_{n \to \infty}\E^{\Po_n}[\xi] = \E^{\Po}[\xi],
			~\mbox{for all}~
			\xi \in B_{mc}(\overline{\Om}),
		\end{equation*}
		and
		\begin{equation*}
			\limsup_{n \to \infty}\E^{\Po_n}[\xi] \le \E^{\Po}[\xi],
			~\mbox{for all}~
			\xi \in B_{mu}(\overline{\Om}).
		\end{equation*}
	\end{theorem}

	Let $\F = (\Fc_t)_{t \in [0,T]}$ be a filtration on the probability space $(\Om, \Fc, \P)$,
	we introduce $\Fcb^0$ and $\Fbb^0 = (\Fcb^0_t)_{t \in [0,T]}$ by
	$$
		\Fcb^0 := \Fc \otimes  \{\emptyset, [0,T]\},
		~~~
		\Fcb^0_t := \Fc_t \otimes \{\emptyset, [0,T]\},
		~~
		t \in [0,T].
	$$
	Further, we introduce a canonical element $\Theta : \Omb \longrightarrow [0,T]$ by
	$$
		\Theta(\omb) ~:=~ \theta, ~~
		\mbox{for all}~
		\omb = (\om, \theta) \in \Omb.
	$$
	The following result is mainly a direct adaptation to our context from Carmona, Delarue and Lacker \cite[Theorem 6.4]{CarmonaDelarueLacker2017}, we nevertheless provide a proof for completeness.
	\begin{proposition}\label{prop:H_hypothesis}
		Let $T < \infty$, then the set
		$$
		\Pcb_0
		:=
		\Big\{
			\Po \in \Pc(\Omo) ~:
			\Po|_{\Om} = \P,
			~\Po[\Theta \le t|\Fcb^0] = \Po[\Theta \le t|\Fcb^0_t],
			~~\mbox{for all}~
			t \in [0,T]
		\Big\}
		$$
		is compact in $\Pc(\Omo)$ under the stable convergence topology.
	\end{proposition}
	\begin{proof}
		Notice that $[0,T]$ is compact and $\Po|_{\Om} = \P$ for all $\Po \in \Pcb_0$,
		then $\Pcb_0$ is relatively compact under the stable convergence topology.
		
		\vspace{0.5em}
		
		It remains to prove that the set $\Pcb_0$ is closed.
		Let $\{\Po_n\}_{n \in \N} \subset \Pcb_0$ be such that $\Po_n$ converges to some $\Po \in \Pc(\Omo)$ under the stable convergence topology.
		Let $h_t$, $g_t$, $g_T$ be bounded measurable functions on $([0,T], \sigma\{\Theta \wedge t\})$, $(\Om, \Fc_t)$ and $(\Om, \Fc)$, respectively.
		We assume further that the function $\theta \mapsto h_t(\theta)$ is continuous.
		Then, by Theorem \ref{thm:stable_convergence},
		\begin{align*}
            \int_{\Omo}
                h_t(\theta)g_t(\om)g_T(\om)
            \Po(d\om,d\theta)
            ~ = & ~
            \lim_{n \to \infty}
            \int_{\Omo}
                h_t(\theta)g_t(\om)g_T(\om)
            \Po_n(d\om,d\theta)
            \\ ~ = & ~
            \lim_{n \to \infty}
            \int_{\Omo}
                h_t(\theta)g_t(\om)\E^{\Po_n}[g_T|\Fcu_t](\om)
            \Po_n(d\om,d\theta)
            \\ ~ = & ~
            \lim_{n \to \infty}
            \int_{\Omo}
                h_t(\theta)g_t(\om)\E^{\Po}[g_T|\Fcu_t](\om)
            \Po_n(d\om,d\theta)
            \\ ~ = & ~
            \int_{\Omo}
                h_t(\theta)g_t(\om)\E^{\Po}[g_T|\Fcu_t](\om)
            \Po(d\om,d\theta).
        \end{align*}
        Thus, it holds that
        \begin{equation*}
            \int_{\Omo}
                h_t(\theta)g_t(\om)g_T(\om)
            \Po(d\om,d\theta)
            ~ = ~
            \int_{\Omo}
                h_t(\theta)g_t(\om)\E^{\Po}[g_T|\Fcu_t](\om)
            \Po(d\om,d\theta),
        \end{equation*}
        for all bounded measurable functions $h_t$, $g_t$, $g_T$ on $([0,T], \sigma\{\theta \wedge t\})$, $(\Om, \Fc_t)$ and $(\Om, \Fc)$, respectively.
        Finally, by the definition of conditional expectation, one can conclude the proof.
    \end{proof}
\end{appendix}

\end{document}